\documentclass[12pt,reqno, hidelinks]{amsart}

\usepackage{amscd,amsmath,amsfonts,amsthm,epsfig, graphics, latexsym, mathrsfs}
\usepackage{amssymb}
\usepackage{hyperref}
\usepackage{afterpage}
\usepackage[alphabetic,initials]{amsrefs}

\usepackage{mathrsfs}
\usepackage{amsmath}
\usepackage{amssymb}
\usepackage{amsthm} 
\usepackage{graphicx}
\usepackage{color}
\usepackage{cite}
\usepackage{amsfonts}
\usepackage[all]{xy}
\usepackage{fancyvrb}
\usepackage{tikz}
\usepackage{dsfont}
\usetikzlibrary{matrix}
\usepackage[margin=1.2in]{geometry}
\usetikzlibrary[positioning,patterns]

\newtheorem{Theorem}{Theorem}[section]
\newtheorem{Prop}[Theorem]{Proposition}
\newtheorem{Lemma}[Theorem]{Lemma}
\newtheorem{Cor}[Theorem]{Corollary}

\theoremstyle{definition}

\newcommand{\appsection}[1]{\let\oldthesection\thesection
\renewcommand{\thesection}{Appendix \oldthesection}
\section{#1}\let\thesection\oldthesection}

\theoremstyle{remark}

\newtheorem{Remark}[Theorem]{Remark}

\DeclareFontFamily{OT1}{pzc}{}
\DeclareFontShape{OT1}{pzc}{m}{it}
              {<-> s * [1.1] pzcmi7t}{}
\DeclareMathAlphabet{\mathpzc}{OT1}{pzc}
                                 {m}{it}

\title[A boundary divisor in the moduli space of stable quintic surfaces]
{A boundary divisor in the moduli space of\\  stable quintic surfaces}
    \author{Julie Rana}

\begin{document}
\maketitle

\begin{abstract}

We give a bound on which singularities may appear on Koll\'ar--Shepherd-Barron--Alexeev stable surfaces for a wide range of topological invariants and use this result to describe all stable numerical quintic surfaces (KSBA-stable surfaces with $K^2=\chi=5$) whose unique non Du Val singularity is a Wahl singularity. We then extend the deformation theory of Horikawa in~\cite{horikawa1975} to the log setting in order to describe the boundary divisor of the moduli space $\overline{\mathcal{M}}_{5,5}$ corresponding to these surfaces. Quintic surfaces are the simplest examples of surfaces of general type and the question of describing their moduli is a long-standing question in algebraic geometry. 

\end{abstract}

\section{Introduction}\label{Introduction}

Let $\mathcal{M}_{K^2, \chi}$ be the moduli space of minimal surfaces of general type, and  $\overline{\mathcal{M}}_{K^2, \chi}$ its KSBA compactification~\cite{ksb1988, alexeev1994b}. Here stable surfaces are surfaces with ample canonical class and at most semi log canonical singularities. The moduli spaces $\overline{\mathcal{M}}_{K^2, \chi}$ are complicated; they may have many connected components~\cite{catanese1986} and arbitrary singularities~\cite{vakil2006}. Recently there has been substantial interest in describing singular stable surfaces explicitly, as a means to understanding the structure of the moduli spaces themselves. We are especially interested in those singularities with a one-parameter $\mathbb{Q}$-Gorenstein smoothing, as these may, in the absence of obstructions, give a divisor in the boundary of the moduli space, corresponding to equisingular deformations of the singularity (see, for example,~\cite{hacking2012}). 

An important type of semi log canonical singularity is the cyclic quotient singularity. Those cyclic quotient singularities which admit a one-parameter smoothing are of type $\frac{1}{n^2}(1, na-1)$ where $a$ and $n$ are relatively prime~\cite{ksb1988}. We refer to these as Wahl singularities~\cite{wahl1981}. 
In consideration of the above observation, we focus on surfaces whose unique non Du Val (or ADE) singularity is a Wahl singularity. We begin Section~\ref{wahlsings} with the following simple observation.

\begin{Lemma}\label{kawamata1} Let $W$ be a stable surface whose unique non Du Val singularity is a Wahl singularity, and let $X$ be its minimal resolution. Let $S$ be the minimal model of $X$, obtained by contracting all $(-1)$ curves on $S$. If $K_{S}$ is big and nef, then $K_W^2>K_{S}^2$.
\end{Lemma}

We remark that Lemma~\ref{kawamata1} is similar to a result of Kawamata~\cite[2.4, 4.6]{kawamata1992}, but in his case the surface $W$ must be the central fiber of a $\mathbb{Q}$-Gorenstein degeneration whose generic fiber is a smooth connected surface. In this paper, we study the case where the difference $K_W^2-K_{S}^2$ is as small as possible: What happens when $K_W^2=K_{S}^2+1$?

As with all cyclic quotient singularities, the minimal resolution of a Wahl singularity 
consists of a string of exceptional curves with negative self-intersections.  If this string contains $r$ exceptional curves, then we say that the singularity itself has length $r$. It is tempting to try to prove restrictions on the types of Wahl singularities that may appear on a given surface by bounding their lengths. This is possible; in~\cite{lee1999}, Y. Lee shows that if $W$ is a surface of general type whose unique non-Du Val singularity is a Wahl singularity of length $r$, then $r\le 400 (K_{S}^2)^4$, where $S$ is the minimal model of the minimal resolution of $W$. The following result greatly improves Lee's bound, although it applies only to those surfaces for which $K_W^2=K_{S}^2+1$.

\begin{Theorem}\label{r is two - 1} Let $W$ be a surface with a unique Wahl singularity $p$ of length $r$ and at most Du Val singularities elsewhere, and let $S$ be the minimal model of the minimal resolution of $W$. If $K_W$ and $K_{S}$ are big and nef and if $K_{S}^2=K_W^2-1$, then $r=1$ or $2$. That is, $p$ is a $\frac{1}{4}(1,1)$ or $\frac{1}{9}(1,2)$ singularity.
\end{Theorem}

Using Horikawa's descriptions of surfaces lying on the Noether line~\cite{horikawa1976}, we improve the result further for surfaces near it:

\begin{Theorem}\label{r is one - 1} With the same hypotheses as in Theorem~\ref{r is two - 1}, assume moreover that $K_W^2=2p_g-3$. If $S$ is of general type then $p$ is a $\frac{1}{4}(1,1)$ singularity. Moreover, if $p$ is a $\frac{1}{4}(1,1)$ singularity and $K_W^2>3$, then $S$ is of general type.
\end{Theorem}

Beginning in Section~\ref{classification section}, we apply this result to the moduli space $\mathcal{M}_{5,5}$ of numerical quintic surfaces, or minimal surfaces with $K^2=\chi=5$. This moduli space was described by Horikawa in~\cite{horikawa1975}, and is a union of two $40$-dimensional irreducible components meeting, transversally at a general point, in a 39-dimensional irreducible variety. Figure~\ref{picture of M55-1} gives a schematic diagram of $\mathcal{M}_{5,5}$. Each component parametrizes \emph{smooth} surfaces with $K^2=\chi=5$, although surfaces in components IIa and IIb are not quintic surfaces in the usual sense. 

\begin{figure}[h!]
  \centering
        \includegraphics[scale=.3]{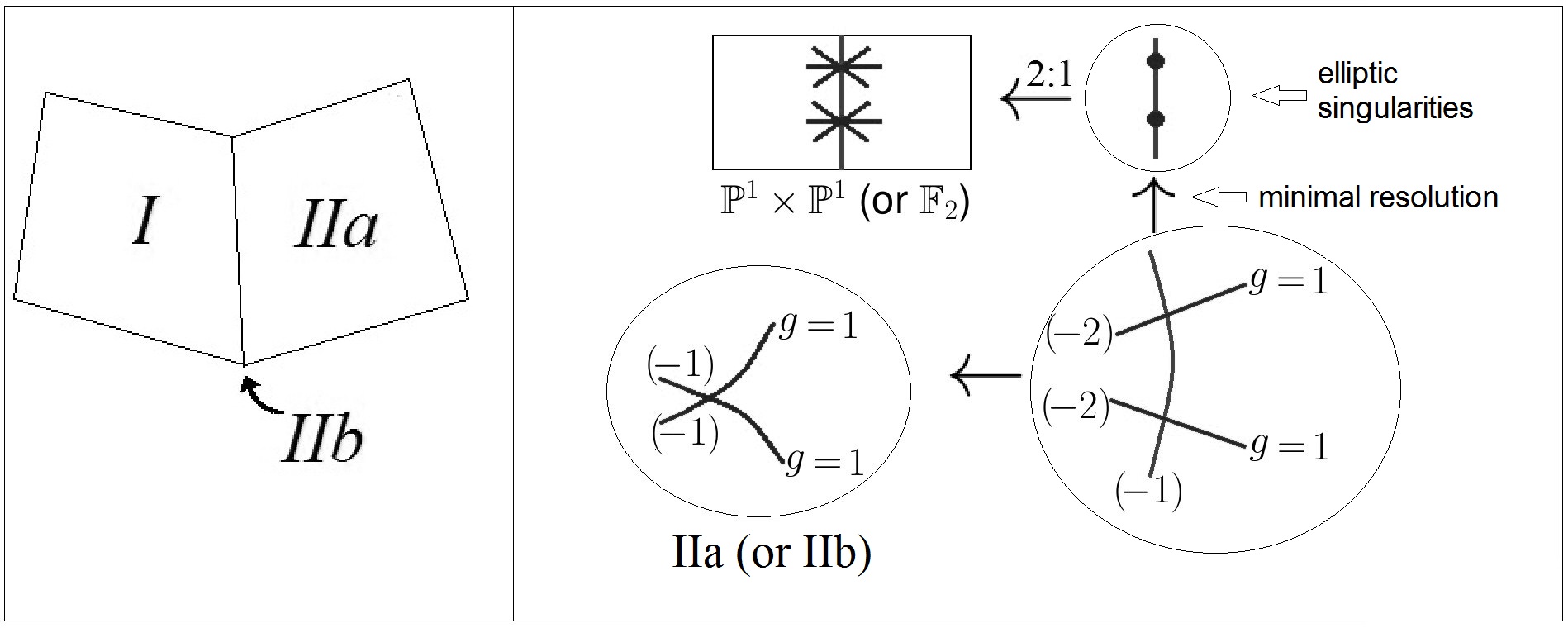}
  \caption{On the left, a visualization of $\mathcal{M}_{5,5}$. Components I and IIa are 40-dimensional; IIb is 39-dimensional. On the right, how to obtain a numerical quintic surface of type IIa or IIb from double covers of $\mathbb{P}^1\times\mathbb{P}^1$, or $\mathbb{F}_2$, respectively. \label{picture of M55-1}}
\end{figure}

Theorem~\ref{r is one - 1}, together with Horikawa's description of surfaces with small $K^2$~\cite{horikawa1976}, suggests that it is possible to describe all stable surfaces lying one above the Noether line whose unique non Du Val singularity is a $\frac{1}{4}(1,1)$ singularity. We do this for the case of stable numerical quintic surfaces by looking at the minimal resolution $X$ of a stable numerical quintic surface. In particular, we prove that the surface $X$, which contains a rational curve of self-intersection $-4$, arises from the double cover of a smooth or nodal quadric, with branch locus intersecting a given curve in one of a few specified ways.  
We note here three such constructions, which we refer to as surfaces of types 1, 2a, and 2b, which will be essential in describing the divisor in $\overline{\mathcal{M}}_{5,5}$ corresponding to surfaces whose unique non Du Val singularity is a $\frac{1}{4}(1,1)$ singularity. The minimal resolution of a surface of type 1 is a double cover of $\mathbb{P}^1\times\mathbb{P}^1$, branched over a sextic intersecting a given diagonal tangentially at $6$ points. The preimage of the diagonal is two $(-4)$-curves, intersecting at 6 points. Contracting one of these $(-4)$-curves gives a stable numerical quintic surface of type 1. The minimal resolutions of type 2a (respectively, 2b) surfaces are themselves minimal resolutions of double covers of $\mathbb{P}^1\times\mathbb{P}^1$ (respectively, a quadric cone), the branch curve of which is a sextic $B$ intersecting a given ruling at two nodes of $B$ and transversally at two other points. 

The minimal models of the stable numerical quintic surfaces we study all arise from double covers of a smooth or nodal quadric surface. Thus, our approach is to describe equisingular deformations of these surfaces by deforming the quadric, together with its branch locus, in such a way that the $(-4)$-curve on $X$ is preserved. In doing so, we are met with an interesting difficulty: the $(-4)$-curve may break on the special fiber. The hope is that one may avoid this by performing a sequence of flops, but this is not immediate. We use representation theoretic tools to prove that such a sequence does exist in a number of important cases.
For instance, we have the following general result, which we use to describe the closure of the locus of type 1 surfaces in $\overline{\mathcal{M}}_{5,5}$.

\begin{Theorem}\label{simultaneous resolution for even intersection - 1}Let $Z$ be a smooth surface, $B$ a divisor on $Z$ with at most Du Val singularities, and $D$ a smooth irreducible divisor on $Z$. Let $(\mathcal{Z,B,D})$ be a family of triples over the unit disk in $\mathbb{C}$, with special fiber $(Z, B, D)$, and such that the divisors $\mathcal{D}_t$ and $\mathcal{B}_t$ are reduced, irreducible and smooth for $t\neq 0$. Suppose that at each point $p\in \mathcal{D}_t\cap \mathcal{B}_t$ over the general fiber, the local intersection $(\mathcal{D}_t\cdot\mathcal{B}_t)_p$ is even. Let $f:\mathcal{Y}\rightarrow \mathcal{Z}$ be the double cover branched over $\mathcal{B}$. Then there exists, after a possible finite base change, a simultaneous resolution of singularities $\psi:\mathcal{X}\rightarrow \mathcal{Y}$ such that the closure of one of the two components of $\psi^{-1}(f^{-1}(\mathcal{D}))_t$ over the general fiber has irreducible special fiber.
\end{Theorem}

In Section~\ref{wahldiv}, we explore the deformation theory of surfaces of types 1 and 2a, as these surfaces correspond to $39$-dimensional loci in $\overline{\mathcal{M}}_{5,5}$. To begin with, we describe explicit $\mathbb{Q}$-Gorenstein smoothings of type 1, 2a, and 2b surfaces to numerical quintic surfaces, showing that these loci lie on the boundary of the components of type I, IIa, and IIb of $\mathcal{M}_{5,5}$, respectively. Note that the smoothings of types 2a and 2b which we describe are simple extensions of an example of Friedman found in~\cite{friedman1983}. For surfaces of types 1 and 2a, we then prove vanishing of the cohomology groups in which obstructions to deformations of these surfaces lie, and conclude that the closures of these loci are smooth Cartier divisors in $\overline{\mathcal{M}}_{5,5}$ at their general points. This implies that  $\overline{\mathcal{M}}_{5,5}$ is smooth generically along these divisors. 

Section~\ref{2b} begins with a proof that obstructions to deformations of 2b surfaces do not vanish. Understanding the deformations of such surfaces proves to be the key to our full description of the divisor in $\overline{\mathcal{M}}_{5,5}$ corresponding to stable numerical quintic surfaces whose unique non Du Val singularity is a $\frac{1}{4}(1,1)$ singularity. Indeed, our study of these surfaces, together with the description of the closures of the 1 and 2a loci, and Horikawa's description of $\mathcal{M}_{5,5}$, allows us to prove the following theorem.

\begin{Theorem}\label{Wahl divisor} The locus of stable numerical quintic surfaces whose unique non Du Val singularity is a $\frac{1}{4}(1,1)$ singularity forms a divisor in $\overline{\mathcal{M}}_{5,5}$ which consists of two 39-dimensional components $\bar{1}$ and $\overline{\mbox{2a}}$ meeting, transversally at a general point, in a 38-dimensional component $\overline{\mbox{2b}}$. These components are the closures of the loci of the surfaces of types 1, 2a, and 2b described above. This divisor is smooth and Cartier at general points of the $\bar{1}$ and $\overline{\mbox{2a}}$ components, and is Cartier at general points of the  $\overline{\mbox{2b}}$ component.  Moreover, the types  $\bar{1}$, $\overline{\mbox{2a}}$, and  $\overline{\mbox{2b}}$ components belong to the closures of the components in $\mathcal{M}_{5,5}$ of types I, IIa, and IIb, respectively.
\end{Theorem}

We remark that Theorem~\ref{Wahl divisor} answers a question Friedman posed in~\cite{friedman1983}, specifically that of explicitly describing deformations of 2b surfaces. 

The proof of Theorem~\ref{Wahl divisor} is outlined at the beginning of Section~\ref{2b}. We note here a few key facts which we prove in Section~\ref{2b}. The first is that the space of obstructions to $\mathbb{Q}$-Gorenstein deformations of a 2b surface is one-dimensional, and so the moduli space of $\mathbb{Q}$-Gorenstein deformations of 2b surfaces is a hypersurface singularity. Together with our description of the closures of the loci of types 1 and 2a surfaces and Horikawa's description of $\mathcal{M}_{5,5}$, this implies that it is enough to understand the equisingular deformations of a generic 2b surface. To describe these deformations, we locate a subfunctor of the functor of $\mathbb{Q}$-Gorenstein deformations of 2b surfaces, corresponding to deformations of covers. These deformations are unobstructed, so the there is a smooth component in the moduli space of equisingular deformations of a 2b surface. This observation implies that it is enough to show that the second order part of the Kuranishi function, given by the Schouten bracket, does not vanish and is not a square. We describe this bracket by extending the deformation theory of Horikawa in~\cite{horikawa1975}.

\noindent \emph{Acknowledgments.} This paper is a revision of my thesis. I am especially grateful to my advisor, Jenia Tevelev, for his guidance and support throughout. I would also like to thank Eduardo Cattani, Stephen Coughlan, Paul Hacking, and Radu Laza for many helpful discussions.


\section{Restrictions on singularities}\label{wahlsings}

We give bounds on which Wahl singularities may appear on a stable surface with limited invariants.

The two-dimensional quotient singularities which admit $\mathbb{Q}$-Gorenstein smoothings are called T-singularities, and are those cyclic quotient singularities of the form $\frac{1}{dn^2}(1, dna-1)$ where $a$ and $n$ are coprime \cite{ksb1988}. Those which admit only a one-parameter $\mathbb{Q}$-Gorenstein smoothing are T-singularities with $d=1$. They were studied first by Wahl \cite{wahl1981} and so are called \emph{Wahl singularities}. 

The minimal resolution of a surface with a Wahl singularity of the form $\frac{1}{n^2}(1, na-1)$ contains a string of exceptional curves $C_1,\ldots, C_r$ such that
$$ \label{dotsofcs}
C_i\cdot C_j = \left\{ \begin{array}{rl}
 				1 &\mbox{ if $i=j\pm1$} \\
				-b_i & \mbox{if $i=j$}\\
				  0 &\mbox{ otherwise}
      					 \end{array} \right. 
$$
where $[b_1,\cdots, b_r]$ is the Hirzebruch-Jung continued fraction expansion of $\frac{n^2}{na-1}$.  We say that the T-string $C_1,\ldots, C_r$ and the singularity corresponding to it have \emph{length $r$}. 

The T-string of a Wahl singularity has an especially useful iterative description by Wahl.

\begin{Prop}\label{wahltype}\cite{wahl1981} The cyclic quotient singularity $\frac{1}{4}(1,1)$ is a Wahl singularity of length 1 with $b_1=4$. Moreover,  every Wahl singularity has a T-string $C_1,\ldots, C_r$ where $[b_1,\cdots, b_r]$ is one of the following types:
\begin{itemize}
\item[i)] if $[b_1,\ldots, b_{r-1}]$ is a Wahl singularity then
$$[2,b_1,\ldots,b_{r-1}+1]$$
and
$$[b_1+1,b_2,\ldots, b_{r-1}, 2]$$
are also Wahl singularities and 
\item[ii)] The T-string of any Wahl singularity may be found by starting with the resolution $[4]$ and iterating the steps described in $i)$.
\end{itemize}
\end{Prop}

Because they are quotient singularities, Wahl singularities are log terminal~\cite[4.7]{kollarmori}. 
Thus, if $W$ contains a unique Wahl singularity and is otherwise smooth, and if $X$ is its minimal resolution containing the T-string $C_1,\ldots, C_r$, then we can write
$$K_X=\phi^*K_W+\sum_{i=1}^ra_iC_i$$
where $-1<a_i<0$. There is a very simple relationship between $K_X^2$ and $K_W^2$, also discovered by Wahl.
\begin{Lemma}\label{X W and r}\cite{wahl1981} Let $W$ be a surface with a unique Wahl singularity of length $r$ and at most Du Val singularities otherwise. Let $X$ be is the minimal resolution of $W$. Then $K_X^2=K_W^2-r$.
\end{Lemma}

To describe the possible Wahl singularities which may occur on a surface with given invariants, one might hope to bound $r$ in terms of $K_W^2$ and $K_S^2$, where $S$ is the minimal model of $X$. The best known bound to date was discovered by Y. Lee.

\begin{Theorem}\label{lee's bound}\cite[Th. 23]{lee1999} Suppose $W$ is a surface of general type with a unique Wahl singularity 
of length $r$. Let $X$ be its minimal resolution and $S$ the minimal model of $X$. If $K_S$ is ample then $r\le400(K_S^2)^4$. 
\end{Theorem}

We prove a much stronger bound, at the cost of restricting to a smaller class of surfaces.

Let $W$ be a surface with a unique Wahl singularity of length $r$ and possibly Du Val singularities, let $\psi:X\rightarrow W$ be its minimal resolution, and $\pi: X\rightarrow S$ be the minimal model of $X$ as in Figure~\ref{firstmaps}. 
\begin{figure} [h]
\centering
\[
\xymatrix{
&X \ar[ld]_\phi \ar[rd]^\pi \\
W&  &S}
\]
\caption{The surfaces $W$, $X$, and $S$.\label{firstmaps}}
\end{figure}
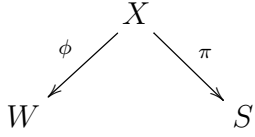
If $\pi$ contracts $n$ $(-1)$-curves, then $K_X^2=K_S^2+n$. By Lemma~\ref{X W and r}, we have $K_X^2=K_W^2-r$. We bound $r$ by investigating the relationship between $n$ and $r$. The following Lemma shows that if $K_W$ and $K_S$ are big and nef, then $r>n$; that is, $K_W^2>K_S^2$.

\begin{Lemma}\label{kawamata}
If $K_W$ and $K_S$ are big and nef then $K_W^2>K_S^2$.
\end{Lemma}
\begin{proof} 

Let $W$ be a surface with a unique Wahl singularity of type $\frac{1}{n^2}(1,na-1)$ at $p$ and at most Du Val singularities elsewhere. Since resolving the Du Val singularities on $W$ does not affect $K_W^2$ and nefness of $K_W$, we can assume without loss of generality that $W$ is smooth away from $p$. Choose $m>0$ such that $n | m$. Then $mK_W$ is Cartier.

Since $K_S$ and $K_W$ are big and nef, we have
$$h^i(S, mK_S)=h^i(S, (m-1)K_S+K_S)=0$$
and
$$h^i(W, mK_W)=h^i(W, (m-1)K_W+K_W)=0$$
for $i>0$ by the Kawamata--Viehweg vanishing theorem. 
In particular, 
$$\chi(S, mK_S) =h^0(S, mK_S) \textrm{ and } \chi(W, mK_W) =h^0(W, mK_W).$$

We claim that 
$$h^0(W, mK_W)>h^0(X, mK_X)=h^0(S, mK_S)$$
for $m$ sufficiently large.
To see this, write 
$$K_X=\phi^*(K_W)+\sum_{i}a_iC_i,$$
where $-1<a_i<0$, because $p$ is log terminal. Choose $m$ sufficiently large and divisible so that the denominators of the $a_i$ divide $m$ for all $i$. Then
$$\phi^*(mK_W)=mK_X+C,$$
where $C=-m\sum_{i}a_iC_i$ is an effective Cartier divisor. Consider the restriction exact sequence
$$0\rightarrow \mathcal{O}(mK_X)\rightarrow\mathcal{O}(\phi^*(mK_W))\rightarrow \mathcal{O}_C\rightarrow 0.$$
To show that $h^0(W, mK_W)>h^0(X, mK_X)$, it suffices to show that the induced map
$$H^0(X, \phi^*(mK_W))\rightarrow H^0(C, \mathcal{O}_C)$$
is nonzero. By the Kawamata-Shokurov base point free theorem,
we can choose a section $s$ of $mK_W$, for $m$ sufficiently large and divisible, such that $s(p)\neq 0$. Thus, the map is indeed nonzero.

Since $p$ has index $n$, the divisor $mK_W$ is Cartier and the usual Riemann--Roch Theorem holds~\cite{reidypg}. 
Thus, 
\begin{eqnarray*}
 \chi(W, \mathcal{O}_W) +\frac{m(m-1)}{2}K_W^2 &=&\chi(W, mK_W)= h^0(W, mK_W)\\
&>& h^0(S, mK_S)= \chi(S, mK_S)\\
&= & \chi(S, \mathcal{O}_S) +\frac{m(m-1)}{2}K_S^2.
\end{eqnarray*}
Since $\psi$ is the resolution of a rational singularity, we have 
$$\chi(W, \mathcal{O}_W)=\chi(X, \mathcal{O}_X)=\chi(S, \mathcal{O}_S),$$
and so $K_W^2>K_S^2$ as we wished to show.
\end{proof}

\begin{Remark} Kawamata makes a similar statement and argument, but requires that $W$ be the central fiber of a $\mathbb{Q}$-Gorenstein degeneration $\mathcal{X}\rightarrow \Delta$ whose generic fiber is a smooth connected surface~\cite[2.4, 4.6]{kawamata1992}.
\end{Remark}

Because it is difficult to give a useful bound on $r$ without any assumptions on $n$, we begin by restricting to the case that $K_W^2=K_S^2+1$. We will then use Noether's inequality together with Lemma~\ref{kawamata} to show that this holds in the case that $W$ is a stable numerical quintic surface.

\begin{Theorem}\label{r is two} Suppose $W$ is a surface with a unique Wahl singularity $p$ of length $r$ and at most Du Val singularities elsewhere. Let $X$ be its minimal resolution, and $\pi:X\rightarrow S$ the minimal model of $X$ as in Figure~\ref{firstmaps}. If $K_W$ and $K_S$ are big and nef, and if $K_W^2=K_S^2+1$, then $p$ is a $\frac{1}{4}(1,1)$, or $\frac{1}{9}(1,2)=\frac{1}{9}(1,5)$ singularity.
\end{Theorem}

\begin{Remark} Although we do not have a specific example, the assumption that $K_S$ is of general type is likely essential. In~\cite{lee-park2011}, Y. Lee and J. Park give an infinite family of examples of $\mathbb{Q}$-Gorenstein degenerations of minimal surfaces of general type with $K^2=2p_g-4$ to surfaces that contain \emph{two} Wahl singularities of type $\frac{1}{(n-2)^2}(1, n-3)$.The central fibers of the minimal resolutions of these families are minimal elliptic surfaces.
\end{Remark}

The proof of Theorem~\ref{r is two} requires two lemmas, but we begin with some notation.

Let us write $\pi$ as a composition of birational maps, each of which contracts a single (-1)-curve to a point $x_j\in X_j$:
$$
\xymatrix{X=X_n\ar[r]^{\pi_n} &X_{n-1}\ar[r]^{\pi_{n-1}}&\cdots\ar[r]^{\pi_2}&X_1\ar[r]^{\pi_1}&X_0=S}
$$
For $j\in\{1,\ldots,n\}$, let $F_j=\pi_{j}^{-1}(x_{j-1})\subset X_{j}$ be the (-1)-curve on $X_{j-1}$ obtained by blowing up the smooth point $x_{j-1}\in X_{j-1}$. Let 
$$E_j=(\pi_j\circ\pi_{j+1}\circ\cdots\circ\pi_{n})^{-1}(x_{j-1})\subset X.$$
We call each $E_j$ an ``exceptional divisor" of $\pi$. With this notation, we can write
$$K_X=\pi^*(K_S)+\sum_{i=1}^nE_j.$$

We note that because the maps $\pi_i$ are birational, the self-intersection of $E_j$ is $(-1)$ and $E_i\cdot E_j=0$ for $i\ne j$. We have $E_n=F$ for some $(-1)$-curve $F$. Moreover, each $E_j$ contains at least one $(-1)$-curve and $E_j$ is not necessarily reduced, but its reduction is a tree of rational curves. Finally, each $E_j$ contains no loops of curves and pairs of curves in $E_j$ intersect at most once.

\begin{Lemma}\label{technical lemma 1} $\sum_{j=1}^n\sum_{i=1}^rE_j\cdot C_i\le r$.
\end{Lemma}
\begin{proof}

By adjunction
$$K_X\cdot \sum_{i=1}^rC_i=\sum_{i=1}^r(b_i-2).$$
It is easy to see by induction using Proposition~\ref{wahltype} that
\begin{equation}\label{claim1}\sum_{i=1}^r(b_i-2)= r+1.\end{equation} 

Since $K_S$ is nef, we have
$$\pi^*K_S\cdot \sum_{i=1}^rC_i\ge 1.$$
Therefore,
\begin{equation}\label{KXwithCs}
K_X\cdot \sum_{i=1}^rC_i=\sum_{i=1}^r(\pi^*K_S+\sum_{j=1}^nE_j)\cdot C_i\ge 1+\sum_{i=1}^r\sum_{j=1}^nE_j\cdot C_i
\end{equation}
and so  
$$\sum_{i=1}^r\sum_{j=1}^nE_j\cdot C_i\le K_X\cdot \sum_{i=1}^rC_i-1.$$
Combining this with Equation~\eqref{claim1} gives
\begin{equation*}
\sum_{i=1}^{r}\sum_{j=1}^{n}E_j\cdot C_i\le \sum_{i=1}^{r}C_i\cdot K_X-1=\sum_{i=1}^{r}(b_i-2)-1=r.
\end{equation*}
\end{proof}

\begin{Lemma}\label{technical lemma 2} $\sum_{i=1}^r\sum_{j=1}^{n}E_j\cdot C_i\ge2n.$
\end{Lemma}
\begin{proof} 

The claim is obvious for $n=0$. Fix an exceptional divisor $E=E_j$ for some $j$ and a curve $C=C_i$ for some $i$. If $C\subset E$, then $C\cdot E_j=-1$ if and only if 
$$(\pi_j\circ\pi_{j+1}\circ\cdots\circ\pi_n)(C)=x_j$$ 
and 
$$(\pi_{j+1}\circ\pi_{j+2}\circ\cdots\circ\pi_n)(C)=F_j.$$ 
Otherwise, $C\cdot E_j=0$.
Thus, $\sum_{i=1}^rC_i\cdot E\ge-1$. Since we want $\sum_{i=1}^rC_i\cdot E\ge 2$, it suffices to show that there are at least three points of intersection (counted with multiplicity) among curves in the T-string which are not in $E$ and curves in $E$. 

Given a T-string $\mathcal{C}$ containing curves $C_1,\ldots, C_r$, let 
$$
\xymatrix{\bullet\ar@{-}[r]&\bullet\ar@{-}[r]&\cdots\ar@{-}[r]&\bullet\ar@{-}[r]&\bullet}
$$
be the dual graph of the T-string, where the $i^{\textrm{th}}$ vertex corresponds to the curve $C_i$. If $C_i\subset E$, we replace the $i^{\textrm{th}}$ vertex in the above graph by a box, and denote the resulting graph by $\Gamma_{E}$. For instance, if  $\Gamma_{E}$ is
$$
\xymatrix{\Box\ar@{-}[r]&\bullet\ar@{-}[r]&\Box\ar@{-}[r]&\bullet\ar@{-}[r]&\Box}
$$
then there are at least 4 points of intersection among curves in $\mathcal{C}\backslash E$ and curves in $E$. With this notation we can immediately see that if there are less than 3 such intersections then $\Gamma_{E}$ must have one of the following forms:
\begin{enumerate}
\item[1)]\label{badone}
$$
\xymatrix{\bullet\ar@{-}[r]&\cdots\ar@{-}[r]&\bullet\ar@{-}[r]&\Box\ar@{-}[r]&\cdots\ar@{-}[r]&\Box\ar@{-}[r]&\bullet\ar@{-}[r]&\cdots\ar@{-}[r]&\bullet}
$$
\item[2)]\label{badtwo}$$
\xymatrix{\Box\ar@{-}[r]&\cdots\ar@{-}[r]&\Box \ar@{-}[r]&\bullet\ar@{-}[r]&\cdots\ar@{-}[r]&\bullet\ar@{-}[r] &\Box\ar@{-}[r]&\cdots\ar@{-}[r]&\Box}
$$
\item[3)]\label{badthree}$$
\xymatrix{\Box\ar@{-}[r]&\cdots\ar@{-}[r]&\Box \ar@{-}[r]&\bullet\ar@{-}[r]&\cdots\ar@{-}[r]&\bullet}
$$
\end{enumerate}
Since $n\ge 1$, there is a (-1)-curve $F$ in $E$.
Because $C_i^2<-1$ for all $i$, we also have that $C_i\cdot F\ge0$ for each $i$. We claim moreover that $\phi^*K_W\cdot F>0$. Suppose for a contradiction that $\phi^*K_W\cdot F\le0$. Since $K_W$ is nef, this implies that $\phi^*K_W\cdot F=0$. The surface $W$ is a resolution of Du Val singularities on a stable surface $W'$. Let $\theta: W\rightarrow W'$ be the resolution of Du Val singularities. Since $K_{W'}$ is ample, this implies that $F$ is contracted by $\theta$. But then $F$ is a $(-2)$ curve, a contradiction. 

Writing $K_X=\phi^*K_W+\sum_{i=1}^ra_iC_i$, we have
\begin{equation*}
\sum_{i=1}^rC_i\cdot F\ge-\sum_{i=1}^ra_iC_i\cdot F=\phi^*K_W\cdot F-K_X\cdot F=\phi^*K_W\cdot F+1>1.
\end{equation*}
In particular,
\begin{equation}\label{F dot C}\sum_{i=1}^rC_i\cdot F \ge 2.
\end{equation}
Thus $F$ intersects at least two of the curves $C_i$, or one curve $C_i$ with multiplicity at least two. Moreover, if a curve $C_i$ intersecting $F$ is contained in $E$, then $\pi_{k+1}\circ\cdots\circ\pi_n(C_i)=F_k$ for some $k$. Thus, $\pi_{k+1}\circ\cdots\circ\pi_n(C_i)$ is a smooth curve and so $C_i\cdot F=1$.
Because $E$ does not contain loops of curves, we see that in Cases 1 and 3 the curve $F$ must intersect at least one $C_i$ which is not in $E$. In Case 1, this gives our third point of intersection. In Case 3 it gives a second. 

We now have only to deal with Cases 2 and 3, for both of which we now have $$\sum_{i=1}^rC_i\cdot E\ge1.$$ Suppose there are $k$ exceptional curves $E$ such that $\sum_{i=1}^rC_i\cdot E=1$. We claim that $k=0$.

Suppose for a contradiction that $k>0$.
By the above argument and Lemma~\ref{technical lemma 1} we have
$$r\ge\sum_{j=1}^n\sum_{i=1}^rE_j\cdot C_i\ge 2(n-k) +k =2n-k.$$
Since $r=n+1$, we have that $k\ge n-1$. On the other hand, since $E_n=F$ is a single $(-1)$-curve, we have $k\le n-1$. Thus, $k=n-1=r-2$. 
In particular, this implies that $r\ge 3$, that all but two curves in $\mathcal{C}$ are contained in exceptional divisors, and that all exceptional divisors other than $E_n$ satisfy  $\sum_{i=1}^rC_i\cdot E=1$. This means that there is only one (-1) curve which must therefore be contained in all of the exceptional divisors.

Let us begin with Case 2. If the (-1)-curve $F$ intersects both a bullet and a box in $\Gamma_{E_1}$, then since $\Gamma_{E_i}$ is obtained from $\Gamma_{E_1}$ by replacing some boxes with bullets,  this gives the third intersection point for all $E_i$. So we can assume that it intersects two boxes as in Figure~\ref{string1}.
\begin{figure}[h]
\centering
\includegraphics{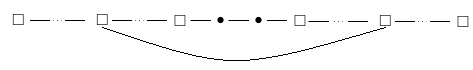}
\caption{$\Gamma_{E_1}$. The curved line along the bottom represents the (-1)-curve $F$.}\label{string1}
\end{figure}

Every exceptional divisor $E_j$ other than $E_n=F$ satisfies $\sum_{i=1}^rC_i\cdot E_j=1$ and must be a subset of $E_1$. Each $E_j$ also contains $F$, so the only possibility is that $F$ intersects $C_1$ and $C_r$. 
However, by~\cite[3.2]{kawamata1992} we have $a_1+a_r=-1$, 
so
$$-1=K_X\cdot F=(\phi^*K_W+\sum_{i=1}^ra_iC_i)\cdot F=\phi^*K_W\cdot F-1.$$
Therefore, $K_W\cdot\phi(F)=0$. Since $\phi(F)$ has positive arithmetic genus and $K_W$ is nef, this is a contradiction.

The final case to consider is Case 3.  Here $\Gamma_{E_1}$ must be of the form:
\begin{figure}[h]
\centering
\includegraphics{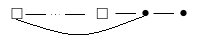}
\end{figure}

\noindent where the curved line along the bottom represents the $(-1)$-curve $F$. Here, $E_1$ is a chain of curves with a $(-1)$-curve at the end. Contracting $F$ under $\pi_n$ gives another $(-1)$-curve, and so $C_r$ is necessarily a $(-2)$-curve. Contracting $\pi_n(C_r)$ under $\pi_{n-1}$ must also give a $(-1)$-curve, so that $C_{r-1}$ must also be a $(-2)$-curve. Continuing in this way, we see that $E_1$ must consist of $n-1$ $(-2)$-curves and a $(-1)$-curve $F$. Thus, $\mathcal{C}$ must correspond to the Wahl singularity with T-string  $[r+3, 2,\ldots, 2]$ or $[r, 5, 2, \ldots, 2]$. 

Suppose first that $b_2=2$. Then using the fact that $K_S$ is nef and that $C_1\cdot F\ge 1$, we have 
$$0=K_X\cdot C_{2}=\pi^*K_S\cdot C_{2}+\sum_{j=1}^nE_j\cdot C_{2}\ge \pi^*K_S\cdot C_2+1\ge 1$$
and we have a contradiction. 

The only Wahl singularity left to consider is that with Hirzebruch-Jung continued fraction $[r, 5, 2, \ldots, 2 ]$. In this case, $\Gamma_{E_1}$ together with $F$ is the graph shown in Figure~\ref{string3}.
\begin{figure}[h]
\centering
\includegraphics{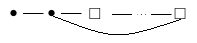}
\caption{The remaining possibility for $\Gamma_{E_1}$. \label{string3}}
\end{figure}

\noindent Since $K_X\cdot C_2=3$ and $C_2\cdot \sum_{j=1}^nE_j\ge n$ we have
\begin{equation*}
0\le \pi^*K_S\cdot C_2= (K_X-\sum_{j=1}^nE_j)\cdot C_2= 3-\sum_{i=1}^nE_j\cdot C_2= 3-n.
\end{equation*}
This gives $n\le 3$, and so $r\le 4$. If  $r=3$, then $C_2^2=-5$. The image $\pi(C_2)$ has self-intersection $0$ and arithmetic genus $1$. Therefore, by adjunction $K_S\cdot \pi(C_2)=0$, contradicting the fact that $K_S$ is big and nef.

Similarly, if $r=4$ then the $\pi(C_2)$ has self-intersection $1$ and arithmetic genus $1$. By adjunction, we have $K_S\cdot \pi(C_2)=-1$, contradicting the fact that $K_S$ is nef.

Since all possibilities lead to a contradiction, we conclude that $k=0$.
\end{proof}

We can now prove Theorem~\ref{r is two}.

\begin{proof}[Proof of Theorem~\ref{r is two}]
We must show that $r\le 2$. By Lemma~\ref{technical lemma 1} we have
$$\sum_{j=1}^n\sum_{i=1}^rE_j\cdot C_i\le r.$$
On the other hand, Lemma~\ref{technical lemma 2} tells us that 
$$\sum_{j=1}^n\sum_{i=1}^rE_j\cdot C_i\ge 2n.$$
Since $n=r-1$, we have that $r\le 2$, so $p$ is a $\frac{1}{4}(1,1)$ or $\frac{1}{9}(1,2)$ singularity.
\end{proof}

Now suppose that $W$ be a stable surface whose unique non Du Val singularity is a Wahl singularity $p$ of length $r$. Let $\phi: X\rightarrow W$ be the minimal resolution of $W$, and let $\pi:  X\rightarrow S$ be the minimal model of $W$, which is obtained from $X$ by contracting $n$ $(-1)$-curves.

\begin{Theorem}\label{r is one} Suppose that $K_W$ is big and nef and satisfies $K_W^2=2p_g-3$.  
If $S$ is of general type then $p$ is a $\frac{1}{4}(1,1)$ singularity. Moreover, if $p$ is a $\frac{1}{4}(1,1)$ singularity and $K_W^2>3$, then $S$ is of general type.
\end{Theorem}

Noether's inequality (that for surfaces $S$ of general type, we have $K_S^2\ge 2p_g-4$) implies the following corollary of Lemma~\ref{kawamata}.

\begin{Cor}\label{kawamata corollary}
If the surface $W$ satisfies $K_W^2=2p_g-4$, then $S$ is not of general type.
\end{Cor}

The significance of the equality $K_W^2=2p_g-3$ in Theorem~\ref{r is one} is that such surfaces lie one above the ``Noether line" $K_W^2=2p_g-4$. That is, $K_W^2$ is the smallest it can be and still have $S$ be of general type.

For the proof of Theorem~\ref{r is one}, we recall Horikawa's description of minimal surfaces of general type with $K^2=2p_g-4$ in ~\cite{horikawa1976}. For $d\ge 0$, the Hirzebruch surface $\mathbb{F}_d$ is the $\mathbb{P}^1$-bundle over $\mathbb{P}^1$ whose zero section $\Delta_0$ has self-intersection $-d$. We denote by $\Gamma$ a generic fiber of $\mathbb{F}_d$ and note that $\mathbb{F}_0=\mathbb{P}^1\times\mathbb{P}^1$.

\begin{Theorem}\label{horikawa}\cite{horikawa1976} Let $S$ be a minimal algebraic surface with $K^2=2p_g-4$ for $p_g\ge 3$.
Then $S$ is the minimal resolution of one of either: 
\begin{enumerate}
\item ($K^2=2$) a double cover of $\mathbb{P}^2$ branched over a curve of degree 8,
\item  ($K^2=8$) a double cover of $\mathbb{P}^2$ branched over a curve of degree 10,
\item a double cover  of $\mathbb{F}_d$, where $p_g\ge \max( d+4, 2d-2)$ and $p_g-d$ is even,  branched over $B\sim 6\Delta_0+(p_g+3d+2)\Gamma$, or
\item ($K^2=4, 6,$ or $8$) a double cover of the Hirzebruch surface $\mathbb{F}_{p_g-2}$ branched over $B\sim6\Delta_0+(4p_g-4)\Gamma$.
\end{enumerate}
In each case, the branch curve has at most ADE singularities.
\end{Theorem}

We call a surface as in Theorem~\ref{horikawa} a \emph{Horikawa surface}. These surfaces are key to the proof of Theorem~\ref{r is one}.

\begin{proof}[Proof of Theorem~\ref{r is one}]
By taking a resolution of Du Val singularities $W'\rightarrow W$, we can assume that $W$ has no Du Val singularities. We first show that if $p$ is a $\frac{1}{4}(1,1)$ singularity and $K_W^2\ge3$, then $S$ is of general type. Since $K_W^2\ge 3$ and $K_W^2=2p_g-3$, we have $p_g\ge 3$. Because $p$ has length $1$, we have $K_X^2=K_W^2-1=2p_g-4\ge 2$. Thus, $K_S^2\ge K_X^2\ge 2$. By the Enriques-Kodaira classification, $S$ is of general type. 

Now suppose that $S$ is of general type. Then $S$ satisfies Noether's inequality $K_S^2\ge 2p_g-4$. On the other hand, by Lemma~\ref{kawamata}, we have $K_S^2<K_W^2=2p_g-3$. Therefore $K_S^2=2p_g-4$. Since the maps $\pi$ and $\phi$ in Figure~\ref{firstmaps} do not affect the invariants $p_g$ and $q$, the surface $S$ must be a Horikawa surface. Furthermore, we have that $K_W^2=K_S^2-1$, so by Lemma~\ref{r is two}, the only possible Wahl singularities on $W$ have length 1 or 2. 

If $p\in W$ is a Wahl singularity of length 2, then the resolution of $p$ in $X$ is a T-string $\{C_1, C_2\}$ where, without loss of generality, $C_1^2=-2$ and $C_2^2=-5$. Since $K_X^2=K_W^2-2=K_S^2-1$, the surface $X$ is the blowup of $S$ in a single point. Let $E$ be the exceptional curve of $\pi$. We have:
\begin{equation}\label{KX in terms of W}
K_X=\phi^*K_W-\frac{1}{3}C_1-\frac{2}{3}C_2
\end{equation}
\begin{equation}\label{KX in terms of S}
K_X=\pi^*K_S+E
\end{equation}

We multiply Equation~\eqref{KX in terms of S} with $C_1$ and $C_2$ and use that $K_S$ is nef to find that $E\cdot C_1=0$ and $E\cdot C_2\le 3$. On the other hand, if we multiply Equation~\eqref{KX in terms of W} with $E$ and use that $K_W$ is nef, we see that $E \cdot C_2 \ge 2$. 

If $E\cdot C_2=3$, then $\pi^*K_S\cdot C_2 =0$, so $K_S\cdot \pi(C_2)=0$. Since $K_S$ is bif and nef, the only possibility is that $\pi(C_2)$ is a $(-2)$-curve. But $\pi(C_2)$ is singular, so this is not possible. 

Now suppose that $E\cdot C_2 =2$. Then $K_S\cdot \pi(C_2)=1$ and $\pi(C_2)^2=-1$. This implies that $\pi(C_2)$ is a nodal or cuspidal cubic. We will use the fact that $S$ is a Horikawa surface to show that in fact such a curve cannot exist on $S$.

By Theorem~\ref{horikawa}, the surface $S$ is the minimal resolution of a surface $Y$ with at most Du Val singularities, which is in turn a double cover of $Z$ where $Z$ is either $\mathbb{P}^2$ or a Hirzebruch surface $\mathbb{F}_d$.  Let $\psi: S\rightarrow Y$ be the minimal resolution of $Y$ and $f:Y\rightarrow Z$ the double cover branched over a curve $B$. See Figure~\ref{secondmaps}.
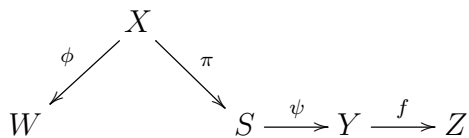
\begin{figure}[h]
\centering
\[
\xymatrix{
&X \ar[ld]_\phi \ar[rd]^\pi \\
W&  &S \ar[r]^\psi & Y\ar[r]^f&Z}
\]
\caption{The surfaces $W$, $X$, $S$, $Y$ and $Z$ and their corresponding maps. Here, $Z$ is either $\mathbb{P}^2$ or $\mathbb{F}_d$ for some $d$.  \label{secondmaps}}
\end{figure}

We must consider four cases, corresponding to the cases in Theorem~\ref{horikawa}. Let $C=\pi(C_2)$, and let $D=f(\psi(C))$ be the image of $C$ on $Z$. 

Case I. ($K^2=2$) Suppose that $Z=\mathbb{P}^2$ and $B\sim 8H$, where $H$ is a hyperplane class. Then 
$K_Y=f^*(-3H+4H)=f^*(H)$, so
\begin{equation*}
1=K_S\cdot C = \psi^*K_Y\cdot C= K_Y\cdot \psi(C)= f^*(H)\cdot \psi(C).
\end{equation*}
Since $f^*H\cdot \psi(C)$ is odd, this implies that $f^*H\cdot \psi(C)=H\cdot D=1$, so $D\sim H$. 

But then $\psi^*(f^*(f(\psi(C))))$ is a union of smooth curves meeting transversally, one component of which is $C$, whereas $C$ is singular.

Case II. ($K^2=8$) If $Z=\mathbb{P}^2$  and $B\sim|10H|$, then $K_Y=f^*(2H)$. In particular $K_Y\cdot F$ is even for any $F$. However, $K_Y\cdot \psi(C)= K_S\cdot C=1$, so this case is impossible.

Case III. Suppose that $Z=\mathbb{F}_d$  and $B\sim|6\Delta_0+(p_g+3d+2)\Gamma|$ where $p_g\ge \max(d+4, 2d-2)$ and $p_g-d$ is even. Then $$K_Y=f^*\left(\Delta_0+\frac{p_g+d-2}{2}\Gamma\right).$$ 
We know $K_Y\cdot\psi( C)=1$, so if $f(\psi(C))\sim (a\Delta_0+b\Gamma)$ where $a$ and $b$ are nonnegative, then   
$$a\frac{m-d-1}{2}+b=1.$$
Since $p_g\ge d+4$ and $f(\psi(C))$ is irreducible, there are two possibilities: $f(\psi(C))\sim \Delta_0$ or $f(\psi(C))\sim \Gamma$. But then in either case, $\psi^{*}(f^{*}(f(\psi(C))))$ is a union of smooth curves meeting transversally, with $C$ as one of the components, whereas $C$ is singular. Therefore, this case is impossible.

Case IV. Suppose that $Z=\mathbb{F}_{p_g-2}$ and $B\sim 6\Delta_0+4(p_g-1)\Gamma$. In this case, $K_Y=f^*(\Delta_0+(p_g-2)\Gamma)$. If $f(\psi(C))\sim (a\Delta_0+b\Gamma)$, where $a$ and $b$ are nonnegative, then intersecting $f(\psi(C))$ with $\Delta_0+(p_g-2)\Gamma$ implies that $b=1$. Since $f(C)$ is irreducible, we have that $a=0$, and so $f(\psi(C))\sim \Gamma$. But again, $\psi^{*}(f^{*}(f(\psi(C))))$ is a union of smooth curves meeting transversally, with $C$ as one of the components, whereas $C$ is singular, and we have a contradiction.

Therefore the only possible length Wahl singularity on $W$ has length 1, so is a $\frac{1}{4}(1,1)$ singularity.
\end{proof}


\section{Stable numerical quintic surfaces with a unique $\frac{1}{4}(1,1)$ singularity}\label{classification section}

A stable numerical quintic surface $W$ is a stable surface with $K^2=5$, $p_g=4$ and $q=0$. We classify all stable numerical quintic surfaces $W$ whose unique non Du Val singularity is a $\frac{1}{4}(1,1)$ singularity. By Theorem~\ref{r is one}, the minimal resolution $\phi: X\rightarrow W$ is a minimal surface such that $K_X^2=K_W^2=4$, $p_g=4$ and $q=0$, so $X$ is a Horikawa surface. Moreover, $X$ contains a $(-4)$-curve $C$, the exceptional divisor of $\phi$. On the other hand, given a Horikawa surface with $K^2=p_g=4$ and $q=0$ and containing a $(-4)$-curve, we can contract $C$ to obtain a stable numerical quintic surface with a unique $\frac{1}{4}(1,1)$ singularity. Thus, the classification of surfaces such as $W$ becomes a question of classifying all Horikawa surfaces with $K^2=p_g=4$ and $q=0$ that contain a $(-4)$-curve.

Theorem~\ref{r is one} suggests that in order to describe surfaces $W$ ``one above the Noether line'' whose unique non Du Val singularity is a $\frac{1}{4}(1,1)$ singularity, we might instead describe pairs $(X, C)$, where $X$ is a Horikawa surface and $C$ is a $(-4)$-curve contained in $X$. Because Horikawa surfaces are all described as minimal resolutions of double covers $f: Y\rightarrow Z$, we can attempt to ``find'' a $(-4)$ curve on a Horikawa surface by describing how such a $(-4)$ curve must arise from a curve on $Z$ intersecting the branch locus in a certain way. 

We begin in~\ref{Double covers} with some notation and basic results about double covers. We then use these results in~\ref{classification of 1/4(1,1)} to describe all stable numerical quintic surfaces whose unique non Du Val singularity is a $\frac{1}{4}(1,1)$ singularity. In~\ref{dimension counts}, we count dimensions of a number of loci in $\overline{\mathcal{M}}_{5,5}$ of such surfaces, and continue in~\ref{1 and 2a closure} to prove that every stable numerical quintic surface whose unique non Du Val singularity is a $\frac{1}{4}(1,1)$ singularity lies in the closure of one of two distinguished loci.

\subsection{Double covers}\label{Double covers}

Let $f: Y\rightarrow Z$ be a double cover of a smooth surface $Z$ branched over a curve $B$ with at most ADE singularities, and let $\psi: X\rightarrow Y$ be the minimal model of $Y$, obtained by resolving all Du Val singularities on $Y$. Then by~\cite[Lemma 5]{horikawa1975}, the surface $X$ is the double cover of a smooth surface $\tilde{Z}$ with smooth branch locus $B'$ obtained as follows:

Let $p=p_0$ be a singular point of $B=B_0$ and let $\sigma_1:Z_1\rightarrow Z=Z_0$ be the blowup of $Z$ at $p$. Let $E_1$ be the exceptional divisor of $\sigma_1$, and let $B_1'=\sigma^*(B)-2E_1$. Define $f_1: Y_1\rightarrow Z_1$ to be the double cover of $Z_1$ branched over $B_1'$. Then there exists a map $\psi_1: Y_1\rightarrow Z_1$ such that the following diagram is commutative.
\[\xymatrix{
Y_1\ar[r]^{f_1} \ar[d]_{\psi_1} & Z_1 \ar[d]^{\sigma_1}\\
Y\ar[r]^f& Z
}\]
If $B_1'$ is smooth, then $Y_1$ is smooth and so we can take $B'=B_1'$, $X=Y_1$, $\tilde{Z}=Z$ and $\tilde{f}=f_1$. Otherwise, repeat the process, taking $p$ to be a singularity of $B_1'$. In this way, we obtain a map $\sigma: \tilde{Z}\rightarrow Z=Z_0$ which is a composition of maps $\sigma_1\circ\cdots\circ\sigma_m$ where $\sigma_i: Z_i\rightarrow Z_{i-1}$ is the blowup of a single smooth point $p_{i-1}\in Z_{i-1}$, where $p_{i-1}$ is a singular point of $B_i'=\sigma_{i}^*(B_{i-1}')-2E_i$.

We remark that the resolution given is not necessarily the log resolution of $B$, because we consider singularities of the curves $B_i'=\sigma_{i}^*(B_{i-1}')-2E_i$, as opposed to non-nodal singularities of the preimage of $B$.

Now suppose that $D$ is a smooth curve contained in $Z$, and let $\tilde{D}$ be the proper transform of $D$ under the map $\sigma$. We denote by $(B\cdot D)_p$ the local intersection multiplicity of $B$ and $D$ at $p\in B\cap D$. If $p\in B\cap D$ is an ADE singularity of $B$, let $D_i$ be the proper transform of $D$ under $\sigma_{1}\circ\cdots\circ\sigma_{i}$, and let $q_i$ be the point of $D_i$ such that $\sigma_{1}\circ\cdots\circ\sigma_{i}(q_i)=p$. Then for some $l>0$ we can rearrange the blowups so that $q_j=p_j$ for $j\le l$ and $q_j\neq p_j$ for $j>l$. That is, $l$ is the smallest integer for which either $B_{l}'$ is smooth at $q_{l}$ or $B_l'$ does not contain $q_l$. In addition, all maps $\sigma_{l+1}, \ldots, \sigma_{m}$ blowup points away from $q_l\in D_l$, so that 
$$(B'\cdot \tilde{D})_{q}=(B_l'\cdot D_l)_{q_l}.$$
We call $l$ the \emph{separation number} of $p$ and note that $l$ depends on both the singularity of $B$ at $p$ as well as how the branches of $B$ at $p$ intersect $D$.

We state here three lemmas, the proofs of which are almost immediate, which will be useful in Theorem~\ref{classify} below.

\begin{Lemma}\label{B dot D changes} Suppose that $p\in B\cap D$ is an ADE singularity of $B$ and that $D$ is smooth. Then $(B_1' \cdot D_1)_{q_1}=(B\cdot D)_p-2$. In particular, if $l$ is the separation number of $p$, then $(B'\cdot \tilde{D})_q=(B\cdot D)_p-2l$.
\end{Lemma}
\begin{proof} We have
$$(B_1'\cdot D_1)_{q_1}=((\sigma^*B-2E_1)\cdot (\sigma^*D-E)) = (B\cdot D)_p-2,$$
as desired.
\end{proof}

\begin{Lemma}\label{D in B} Suppose that the branch locus $B$ of $f$ is reducible and contains an irreducible smooth curve $D$. Let $\bar{B}=B - D$ and let $p$ be a point of $D\cap \bar{B}$. Let $\bar{B}_1=B_1'-D_1$. Then $(\bar{B}_1\cdot D_1)_{q_1}=(\bar{B}\cdot D)_p-1$. In particular, the separation number of $p$ is equal to the local intersection $(\bar{B}\cdot D)_p$. 
\end{Lemma}
\begin{proof} Since $D$ is smooth and $B$ has ADE singularities, any singularity of $B$ has either $2$ or $3$ branches at $p$, of which $D$ is locally a smooth one. If $B$ has two branches at $p$, then $p$ is either an $A_n$ singularity of $B$ for $n$ odd, a $D_n$ singularity of $B$ for $n$ odd, or an $E_7$ singularity of $B$. If $B$ has $3$ branches at $p$, then $p$ is a $D_n$ singularity of $B$ for $n$ even.

In each case, $B_1'=\sigma_1^*(\bar{B})-E_1 +\sigma_1^*(D)-E_1$. Since 
$$((\sigma_1^*\bar{B}-E_1) \cdot (\sigma_1^*D-E_1))_{q_1}=(\bar{B}\cdot D)_{p}-1,$$
we have obtained the desired result.
\end{proof}

Lemma~\ref{D in B} says in particular that if $\bar{B}\cap D$ consists of $k$ distinct points with separation numbers $l_1,\ldots, l_k$, then 
$$\tilde{D}^2=D^2-\sum_{i=1}^kl_i=D^2-(\bar{B}\cdot D).$$

Given $g(y)=y^k(a_k+a_{k+1}y+\textrm{h.o.t.})\in\mathbb{C}[[y]]$, where $a_k\in \mathbb{C}^*$, we call $k$ the \emph{minimal degree} of $g(y)$, and take $k=\infty$ if $g(y)=0$.

\begin{Lemma}\label{B dot D even} Suppose that $p\in B\cap D$ is an $E_8$ singularity of $B$. Then $(B\cdot D)_p$ is either $3$ or $5$.
\end{Lemma}
\begin{proof}
Note that $B$ is unibranched and has multiplicity $3$ at $p$. Thus, if the tangent cone of $B$ at $p$ is transversal to $D$, then $(B\cdot D)_p=3$. On the other hand, if the tangent cone of $B$ at $p$ is tangent to $D$, then choose coordinates on $Z$ so that $B$ has local equation $x^3+y^5$. Then $D$ is locally given by $x-f(y)$ where $f(y)$ has minimal degree $k\ge 2$.Then $(B\cdot D)_p$ is the minimal degree of $f(y)^3+y^5$. Since $f(y)$ has minimal degree at least $2$, this implies that $(B\cdot D)_p=5$.
\end{proof}

\subsection{The classification}\label{classification of 1/4(1,1)}

We continue to use the notation of subsection~\ref{Double covers}. 
Let $\Gamma$ be a fiber of $Z$ and $\Delta$ an irreducible curve in the linear system $|\mathcal{O}_{\mathbb{F}_0}(1,1)|$ on $\mathbb{F}_0$ or $|\Delta_0+2\Gamma|$ on $\mathbb{F}_2$. 

\begin{Theorem} \label{classify} 
There is a one-to-one correspondence between stable numerical quintic surfaces with at most Du Val singularities and a unique $\frac{1}{4}(1,1)$ singularity, and triples $(Z, B, D)$, where $Z=\mathbb{F}_d$ for $d=0$ or $2$,  $B\sim 6\Delta$ has at most ADE singularities, and $D\sim \Gamma$ or $D\sim \Delta$ intersects $B$ as follows:
\begin{enumerate}
\item $D\sim \Gamma$, there exists $p\in D\cap B$ such that $(B\cdot D)_p$ is odd, and $B$ has either $1$ or $2$ singularities along $D$ and intersects $D$ transversally elsewhere. Moreover,
\begin{enumerate}
\item if two singularities of $B$ are contained in $D$, then each singularity $p$ has separation number $1$, and either $(B\cdot D)_p=2$ 
or $(B\cdot D)_p=3$.

\item if one singularity $p$ of $B$ is contained in $D$, then $p$ has separation number $2$, and either $(B\cdot D)_p=4$ or $(B\cdot D)_p=5$.

\end{enumerate}
Figures~\ref{distinct points}, \ref{same points 4}, and \ref{same points 5} show all possible ways $B$ and $D$ may intersect in this case.

\item $D\sim \Delta$, $D\not \subset B$,  and for all $p\in D\cap B$, $(B\cdot D)_p$ is even. 

Figure~\ref{even case resolution} details the possible ways $B$ and $D$ may intersect in this case.
\item $D\sim \Delta$ and $D\subset B$. 
\end{enumerate}
\end{Theorem}

\begin{proof} 
Suppose that $W$ is a stable numerical quintic surface whose unique non Du Val singularity is a $\frac{1}{4}(1,1)$ singularity and let $X$ be its minimal resolution. Then $X$ is a Horikawa surface with $K^2=p_g=4$ and $q=0$, containing a $(-4)$-curve $C$. Let $\hat{\psi}: X\rightarrow \hat{Y}$ be the canonical model of $X$, so that $\hat{Y}$ has at most Du Val singularities. By~\cite{horikawa1976}, $\hat{Y}$ is a double cover of a smooth or singular quadric $\hat{Z}$, with branch locus away from any singularity of $\hat{Z}$. We resolve both $A_1$ singularities of $\hat{Y}$ lying over the singularity of $\hat{Z}$. Then there exists a map $\psi: X\rightarrow Y$, where $Y$ is the double cover $f: Y\rightarrow Z$ of $Z$, where $Z=\mathbb{F}_2$ or $\mathbb{F}_0$, branched over $B\sim 6\Delta$ with at most ADE singularities~\cite{horikawa1976}. We claim that the curve $D=\psi(f(C))$ is linearly equivalent to either $\Delta$ or $\Gamma$.

The canonical class $K_Z$ of $Z$ is linearly equivalent to $-2\Delta$. Let $L$ be a divisor such that $B\sim 2L$. Then since $f$ is a double cover, the canonical class $K_Y$ is given by $f^*(K_Z+L)=f^*(\Delta)$. Thus,  $K_Y\cdot f^*D=2\Delta\cdot D$. 

Let $\bar{C}=\psi(C)\subset Y$. If $D$ is not contained in the branch locus $B$, then $f^*(D)$ is either a union of two curves $\bar{C}$ and $\bar{C}'$ or $f^*D=\bar{C}$, depending upon how the curve $D$ intersects the branch locus $B$. More precisely,  $f^*(D)=\bar{C}+\bar{C}'$ if and only if the multiplicity of $B$ and $D$ is even at each point of intersection. We consider the three cases, $f^*(D)=\bar{C}$, $f^*(D)=\bar{C}+\bar{C}'$, and $D\subset B$, separately.

\textbf{Case I.} Suppose that there exists $p\in D\cap B$ such that $(B\cdot D)_p$ is odd. Then $f^*(D)=\bar{C}$ and we have 
$$2\Delta\cdot D=K_Y\cdot f^*(D)=K_Y\cdot \bar{C}= 2,$$
so $\Delta\cdot D=1$. Since $C$ is irreducible the curve $D$ is also irreducible. Thus, $D\sim\Gamma$. Note that $B\cdot D=6$. 

On the other hand, since $\tilde{f}$ is the double cover of a smooth surface and $C^2=-4$, the curve $\tilde{f}(C)$ is a $(-2)$-curve $\tilde{D}$ on $\tilde{Z}$. Since $\tilde{D}$ has genus $0$ and $\tilde{f}$ is a double cover, the Riemann--Hurwitz formula gives $B'\cdot\tilde{D}=2$. Because $C$ is smooth, the branch divisor $B'$ intersects $\tilde{D}$ transversally. Commutativity of the diagram 
\[\xymatrix{
X\ar[r]^{\tilde{f}} \ar[d]_{\psi} & \tilde{Z} \ar[d]^{\sigma}\\
Y\ar[r]^f& Z
}\]
implies that $\sigma(\tilde{D})=D$. Noting that $D^2=0$ and $\tilde{D}^2=-2$ we see that the map $\sigma$ blows up exactly two points $p_1$ and $p_2$ on $D$, which may be infinitely near. 

Suppose that $p_1$ and $p_2$ are distinct, and let $p=p_1$. Then $p$ has separation number $1$. Moreover, because $C$ is smooth, either $B'$ intersects $\tilde{D}$ transversally at $q$, or $B'$ and $D$ do not intersect at $q$. That is, $(B'\cdot \tilde{D})_q=0$ or $1$. By Lemma~\ref{B dot D changes}, this implies that $(B\cdot D)_p=2$ or $3$. Conversely, if $(B\cdot D)_p=2$ or $3$, then since $B$ is singular, $p$ has separation number $1$.

If $(B\cdot D)_p=2$, then $p$ is an $A_n$ singularity of $B$, and any branches of $B$ at $p$ intersect $D$ transversally. See Figures~\ref{distinct points}(a) and~\ref{distinct points}(b) for the local intersection of $B$ and $D$. 

Now suppose that $(B\cdot D)_p=3$. If $p$ is an $A_n$ singularity of $B$ for $n$ odd, then since $(B\cdot D)_p=3$, one branch of $B$ intersects $D$ transversally at $p$ while the other intersects $D$ at $p$ with multiplicity $2$. For $n >1$, both branches of $B$ are tangent to each other, so this is not possible. Thus, $p$ is an $A_1$ singularity of $B$ and $B$ intersects $D$ at $p$ as in Figure~\ref{distinct points}(c). 

If $p$ is an $A_n$ singularity for $n$ even, then the tangent cone of $B$ at $p$ is tangent to $D$. Choose local coordinates on $Z$ so that $B$ has local equation $x^2-y^{n+1}$ and $D$ has local equation $x-f(y)$ where $f(y)$ has minimal degree $k\ge 2$. Then $(B\cdot D)_p=3$ if and only if $n=2$. In this case, the proper transform $B_1$ of $B$ is smooth and transversal to $D$, as desired. See Figure~\ref{distinct points}(d) for the local picture.

If $B$ has a $D_n$, $E_6$, $E_7$ or $E_8$ singularity at $p$, then since $(B\cdot D)_p=3$, the tangent cone of each branch of $B$ at $p$ must be transversal to $D$. The local intersection of $B$ and $D$ is shown in Figure~\ref{distinct points}(e), (f), (g), (h), and (i).

Figure~\ref{distinct points} summarizes all possible singularities of $B$ along $D$ that may occur if $\sigma$ blows up two distinct points, as well as how the exceptional curves on $\tilde{Z}$ intersect $\tilde{D}$ and the branch divisor $B'$ of $\tilde{f}$.

\begin{figure}
\centering
\includegraphics[scale=.5]{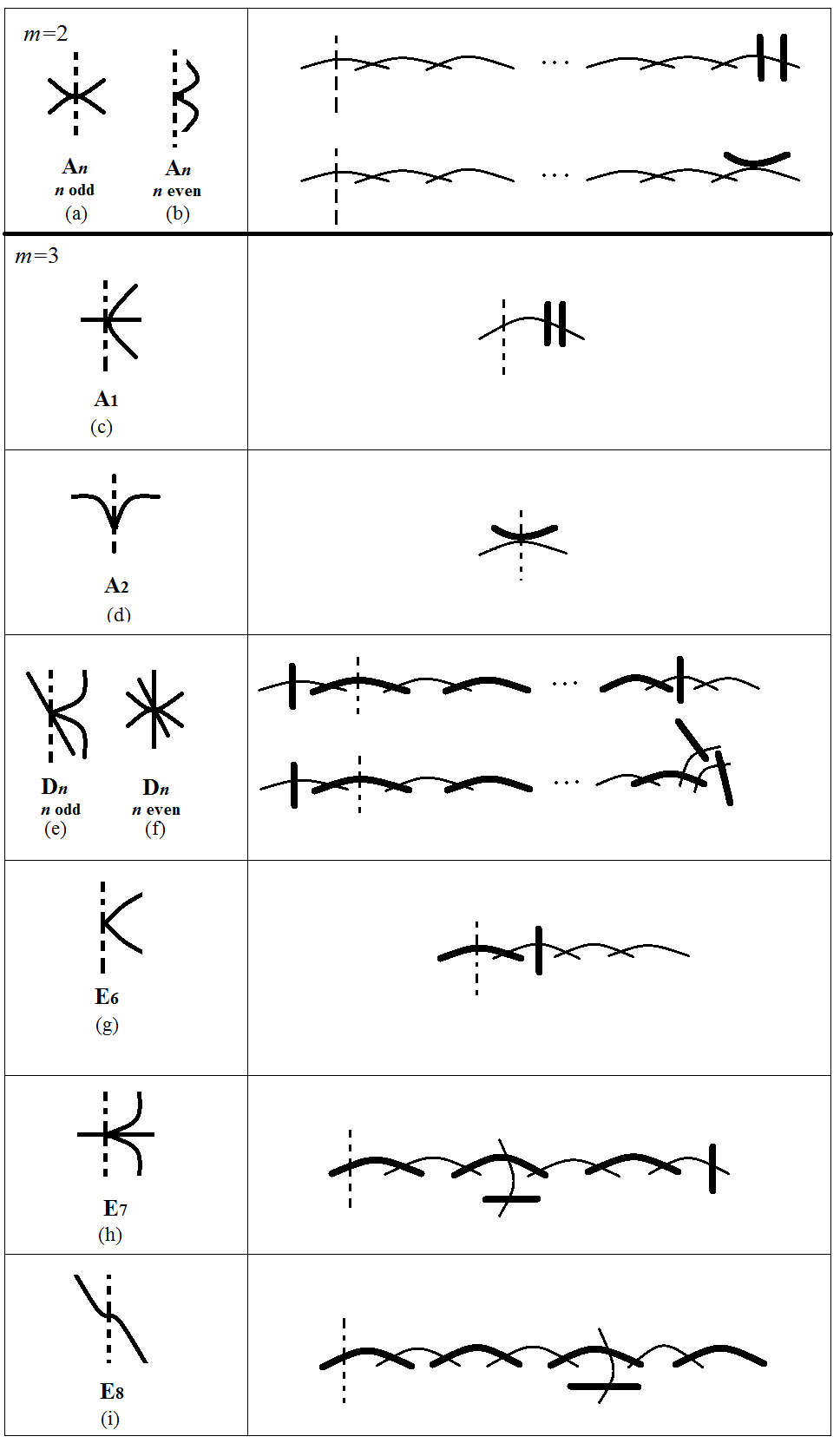}
\caption{On the left, the possible singularities of $B$ along $D\sim \Gamma$ if $p_1\neq p_2$. In each case, the vertical line represents the curve $D$. On the right, the curve $\tilde{D}$, dashed, together with the exceptional divisor of $\sigma$ and the proper transform of $B$. The solid concave down curves denote exceptional divisors. The branch locus of $\tilde{f}$ is denoted by solid bold curves. \label{distinct points}}
\end{figure}

We now consider the case where $p_1$ and $p_2$ are infinitely near. Denote by $p$ the center of the the blowup. Then $p$ has separation number $2$. Moreover, because the curves $\tilde{D}$ and $B'$ are transversal at $q$, we have $(B\cdot D)_p=4$ or $5$. We show that these two properties (that is, $(B\cdot D)_p=4$ or $5$ and $p$ having separation number $2$) hold if and only if $B$ and $D$ intersect at $p$ in one of the ways listed. By Lemma~\ref{B dot D changes}, if $(B\cdot D)_p=4$ or $5$ then $p$ has separation number at most $2$. Thus, it is enough to list all possible intersections with $(B\cdot D)_p=4$ or $5$ and $B_1'$ singular at $q_1$. We note that for $p$ to have separation number greater than 1, at least one branch of $B$ at $p$ must have tangent cone tangent to $D$.

Our method is to consider the possible singularities of $B$ case-by-case, choosing local coordinates at $p$ so that the equation of $B$ is a standard form (for instance $x^2-y^{n+1}$ for an $A_n$ singularity and $y(x^2-y^{n-2})$ for a $D_n$ singularity). In these coordinates, $D$ has local equation $x-f(y)$ or $y-g(x)$ where $f(y)$ (or $g(x)$) has minimal degree $k\ge 2$, and a simple case-by-case calculation tells us which values of $k$ and $n$ will give $(B\cdot D)_p=4$ or $5$. We then determine which of these values will give $B_1'$ singular at $q_1$.

Consider the case $(B\cdot D)_p=4$. Suppose that $p$ is an $A_n$ singularity of $B$ for $n$ odd. If $n=1$, then $B_1$ is smooth, so $p$ has separation number $1$. If $n>1$, then 
$(B\cdot D)_p=4$ if and only if 
\begin{enumerate}
\item[(1)] $n=3$ and $k>2$ (Figure~\ref{same points 4}(a)),
\item[(2)] $n=3$ and $f(y)=ay^2+\mbox{ h.o.t.}$ for $a\neq1$ (Figure~\ref{same points 4}(b)), or
\item[(3)] $n>3$ and $k=2$ (Figure~\ref{same points 4}(c)).
\end{enumerate}
Note that in each case, $B_1$ is singular at $q_1$, so $p$ has separation number $2$.

If $p$ is an $A_n$ singularity of $B$ for $n$ even, then 
$(B\cdot D)_p=4$ if and only if $k=2$ and $n>2$. Since $n>2$, $B_1'$ is singular at $q_1$. Figure~\ref{same points 4}(d) shows the local intersection of $B$ and $D$. 

Suppose that $p$ is a $D_n$ singularity of $B$ for $n$ odd and that the singular branch  of $B$ at $p$ has tangent cone parallel to $D$. Since the smooth branch is transversal to $D$, the singular branch interesects $D$ with multiplicty $3$. 
Since $n$ is odd and $(B\cdot D)_p$ is even,  we have $(B\cdot D)_p =n-1=4$. Thus, $p$ is a $D_5$ singularity of $B$. See Figure~\ref{same points 4}(e) for a visualization of how $B$ and $D$ intersect at $p$. Note that $B_1'$ is singular at $q_1$ as desired.

Now suppose that $p$ is a $D_n$ singularity of $B$ for odd $n$ such that the smooth branch of $B$ at $p$ is tangent to $D$. 
Then $(B\cdot D)_p$ is the minimum of $k+2$ or $k(n-1)$. Since $k>1$ and $n\ge 5$, this implies that $k=2$. The curve $B_1'$ is singular at $q_1$ as desired. See Figure~\ref{same points 4}(f) for the local intersection of $B$ and $D$ at $p$.

If $p$ is a $D_n$ singularity of $B$ for $n$ even, then two branches of $B$ at $p$ are transversal to $D$ and the third is tangent to $D$ with multiplicity $2$. For $n\ge 6$, two branches of $B$ have the same tangent cone, so the branch locus $B$ intersects $D$ at $p$ as in Figure~\ref{same points 4}(g). The local picture for $n=4$ is similar. 

If $p$ is an $E_6$ singularity of $B$, then we can choose coordinates so that $x^3-y^4$ is the local equation of $B$ at $p$ and the local equation of $D$ at $p$ is $x-f(y)$, where $f(y)$ has minimal degree $k\ge 2$. We quickly see that $(B\cdot D)_p=4$ as desired. The intersection of $B$ and $D$ at $p$ is shown in Figure~\ref{same points 4}(h). 

If $p$ is an $E_7$ singularity of $B$, then choose coordinates so that $B$ is locally given by $x(x^2-y^3)$ and $D$ has local equation $x-f(y)$, where $f(y)$ has minimal degree $k\ge 2$. Then $(B\cdot D)_p$ is the minimum of $3k$ and $k+3$. But we require $(B\cdot D)_p=4$, and since $k\ge 2$, this is impossible.

By Lemma~\ref{B dot D even}, $p$ is not an $E_8$ singularity. 

See Figure~\ref{same points 4} for a summary of the ways in which $B$ and $D$ intersect at $p$ if $p_1=p_2$ and $(B\cdot D)_p=4$.

\begin{figure}
\centering
\includegraphics[scale=.6]{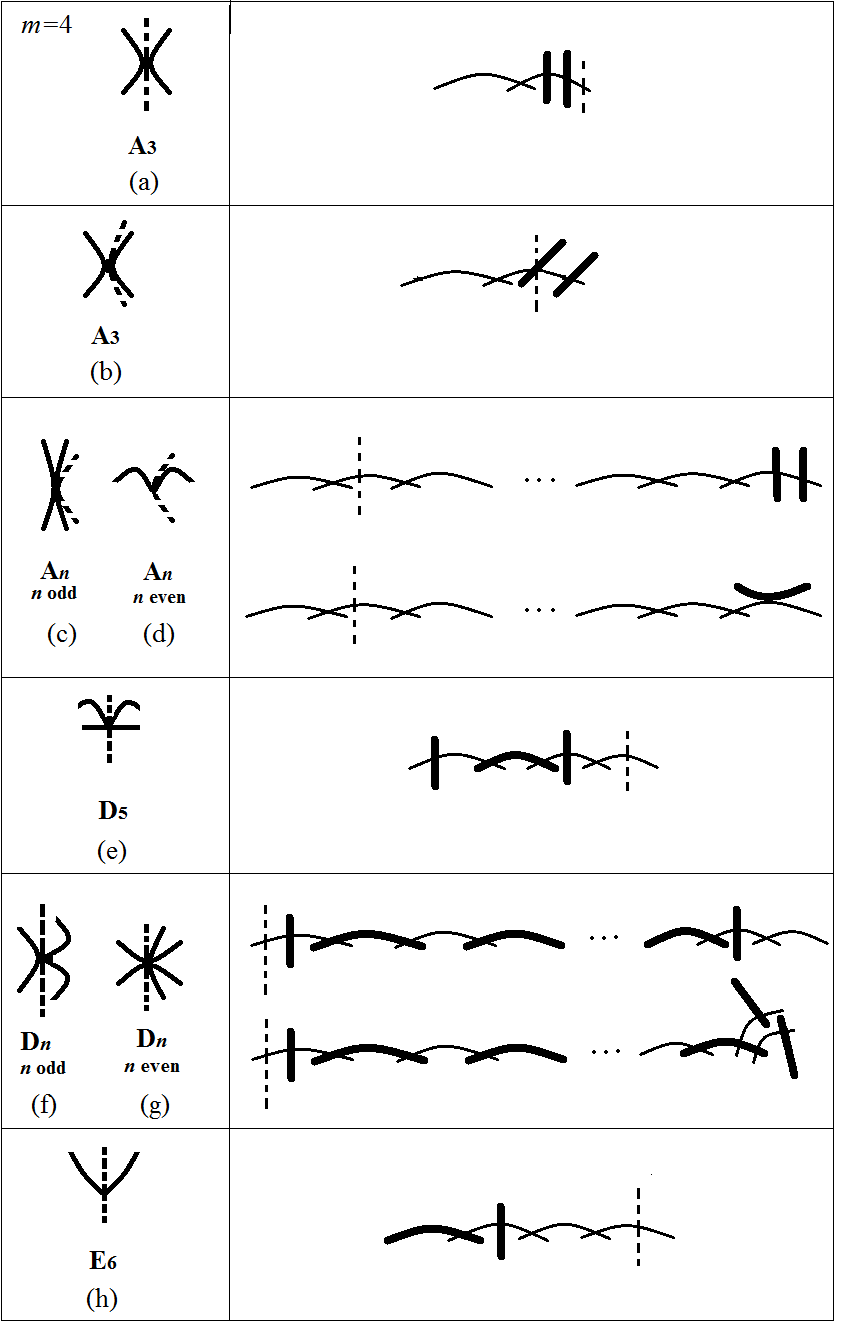}
\caption{On left, the possible singularities of $B$ along $D$ if $p_1= p_2$ and $(B\cdot D)_p=4$. In each case, the dashed line represents $D$. On the right, the curve $\tilde{D}\sim \Gamma$, dashed, together with the exceptional divisor of $\sigma$. The solid concave down curves denote exceptional divisors. The branch locus of $\tilde{f}$ is denoted by solid bold curves. \label{same points 4}}
\end{figure}

We move on to the case $(B\cdot D)_p=5$.  
Suppose that $p$ is an $A_n$ singularity of $B$ where $n$ is odd. If $n=1$, then the singularity of $B$ at $p$ is resolved after a single blowup, so we can assume that $n>1$. Choose coordinates so that the local equation of $B$ at $p$ is $x^2-y^{n+1}$ and the local equation of $D$ at $p$ is $x-f(y)$ where $f(y)=a_ky^k+a_{k+1}y^{k+1}+\mbox{ h.o.t.}$ for some $k\ge 2$. Then in order to have $(B\cdot D)_p$ odd, we must have $n+1=2k$ and $a_k=1$. Thus $(B\cdot D)_p=5=2k+1=n+2$, so $k=2$ and $n=3$. The intersection of $B$ and $D$ at $p$ is shown in Figure~\ref{same points 5}(a).

If $p$ is an $A_n$ singularity of $B$ where $n$ is even, then 
$(B\cdot D)_p$ is the minimum of $2k$ and $n+1$. Thus $(B\cdot D)_p=5$ if and only if $n=4$ and $k\ge 3$. In this case, $B_1$ has an $A_2$ singularity at $q_1$.
See Figure~\ref{same points 5}(b) for the local picture. 

Next, suppose that $p$ is a $D_n$ singularity of $B$ where $n$ is odd and that the tangent cone of the singular branch $S$ is tangent to $D$ at $p$. Then we have $(S\cdot D)_p=4$. Using the same analysis as in previous cases, we  
see that this case occurs if and only if $n\ge 5$ and $k=2$.
See Figure~\ref{same points 5}(c) for the local picture.

If $p$ is a $D_n$ singularity of $B$ for $n$ odd such that the singular branch of $B$ at $p$ has tangent cone transversal to $D$, then the smooth branch is tangent to $D$ at $p$ with multiplicity $3$. See Figure~\ref{same points 5}(d) for the local picture. 

If $p$ is a $D_n$ singularity of $B$ where $n$ is even, then 
either two branches of $B$ are tangent to $D$ at $p$ with multiplicity $2$ each and the third is transversal, or two are transversal to $D$ and the third is tangent to $D$ with multiplicity $3$. In the former case, $p$ is a $D_6$ singularity and $B$ intersects $D$ at $p$ as in Figure~\ref{same points 5}(e). 
In the latter case, $n$ has no further restrictions and the local intersection is shown in Figure~\ref{same points 5}(f). 

We showed above that if $B$ has an $E_6$ singularity at $p$ such that the tangent cone of $B$ at $p$ is tangent to $D$, then $(B\cdot D)_p=4$, so we need not consider the singularity in this case.

If $p$ is an $E_7$ singularity of $B$, then the same analysis as above shows that $D$ is tangent to the tangent cone of $B$ at $p$ with multiplicity $2$. 
See Figure~\ref{same points 5}(g) for the local picture. 

Finally, suppose that $p$ is an $E_8$ singularity of $B$. An analysis of the local equations of $B$ and $D$ as above shows that as long as the tangent cone of $B$ at $p$ is tangent to $D$, we will have $(B\cdot D)_p=5$. In this case, the proper transform $B_1$ of $B$ has a cusp at $q_1$. 
See Figure~\ref{same points 5}(h) for the local picture of $B$ and $D$ at $p$.

See Figure~\ref{same points 5} for a summary of the ways in which $B$ and $D$ intersect at $p$ if $p_1=p_2$ and $(B\cdot D)_p=5$. 
This completes our discussion of Case I.

\begin{figure}
\centering
\includegraphics[scale=.6]{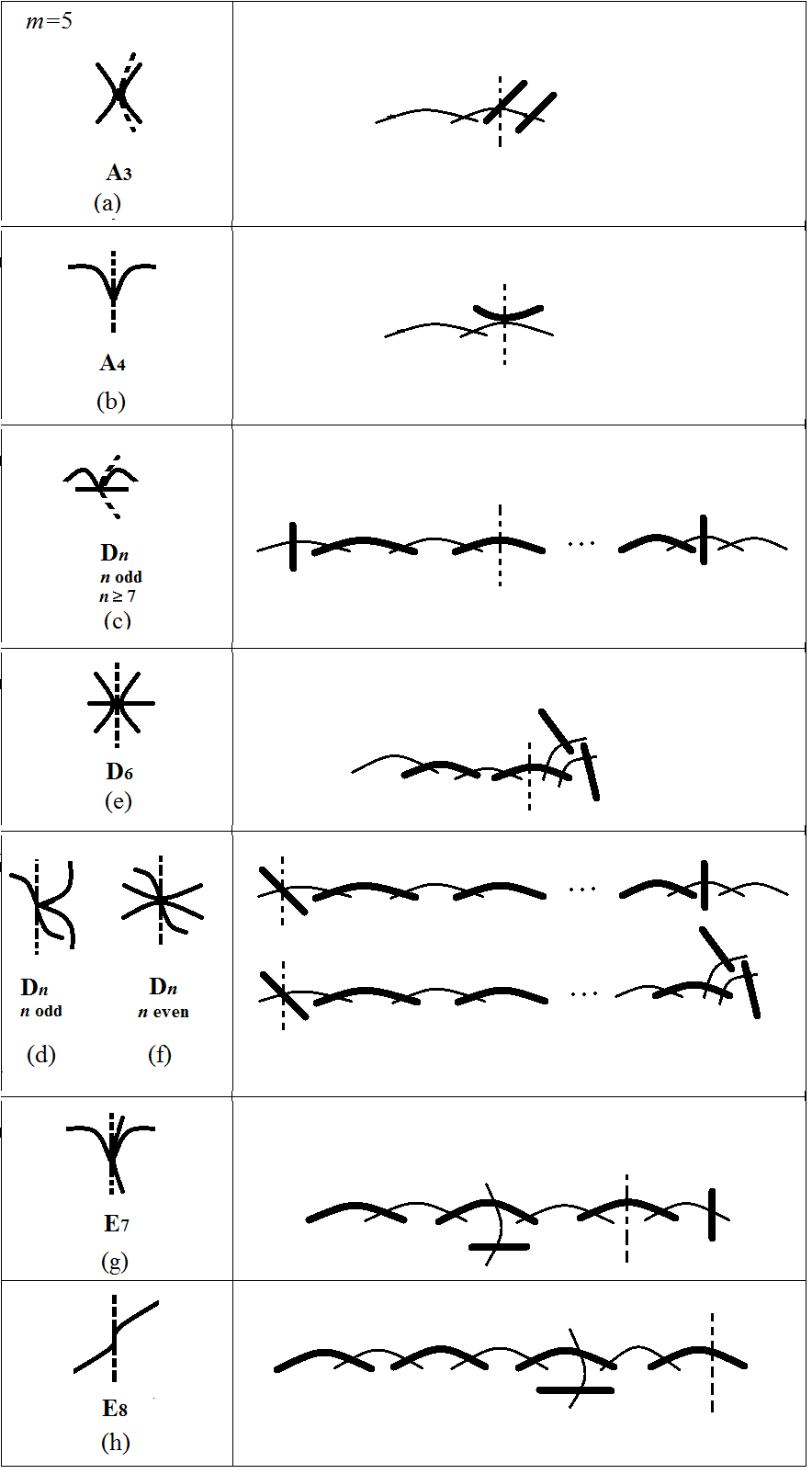}
\caption{On left, the possible singularities of $B$ along $D$ if $p_1= p_2$ and $(B\cdot D)_p=5$. In each case, the dashed line represents $D$. On the right, the curve $\tilde{D}\sim \Gamma$, dashed, together with the exceptional divisor of $\sigma$. The solid concave down curves denote exceptional divisors. The branch locus of $\tilde{f}$ is denoted by solid bold curves. \label{same points 5}}
\end{figure}

\textbf{Case II.} Suppose that $ D\not\subset B$ and $f^*(D)=\bar{C}+\bar{C}'$. Then for each point $p$ of $B\cap D$ the multiplicity $(B\cdot D)_p$ is even, $\bar{C}$ and $\bar{C}'$ are isomorphic, and we have
$$\Delta\cdot D=\frac{1}{2}K_Y\cdot f^*D=\frac{1}{2}K_Y\cdot (\bar{C}+\bar{C}')= 2.$$
Suppose that on $\mathbb{F}_2$, we have $D\sim a\Delta_0+b\Gamma$, where $a$ and $b$ are nonnegative. 
Then $D\cdot\Delta=b$, so that $b=2$. Multiplying $a\Delta_0+2\Gamma$ by $\Delta_0$, we see that in order for a divisor in the linear system $a\Delta_0+2\Gamma$ to be irreducible, we must have $a=1$. Thus, $D\sim \Delta_0+2\Gamma=\Delta$. A similar calculation on $\mathbb{P}^1\times\mathbb{P}^1$ shows that in either case $D\sim \Delta$.

We now show that if $D$ is an irreducible curve in the linear system $\Delta$ such that at each point $p\in D\cap B$ we have $(B\cdot D)_p$ even, then $\tilde{f}^{-1}(\tilde{D})$ is a union of two $(-4)$-curves $C$ and $C'$.

Suppose that $p_1, \ldots p_j$ are the singular points of $B$ lying on $D$. Let $l_i$ be the separation number of $p_i$. Then
\begin{equation*}
(C+C')^2=2\tilde{D}^2=2\left(D^2-\sum_{i=1}^jl_i\right)=2\left(2-\sum_{i=1}^jl_i\right),
\end{equation*}
where the second equality follows from Lemma~\ref{B dot D changes}.
On the other hand
\begin{equation*}
(C+C')^2=2C^2+2C\cdot C'=2C^2+B'\cdot \tilde{D}=2C^2+12-\sum_{i=1}^j2l_i=2C^2+2\left(6-\sum_{i=1}^jl_i\right),
\end{equation*}
where we again use Lemma~\ref{B dot D changes}.
Thus, 
$$C^2+6-\sum_{i=1}^jl_i=2-\sum_{i=1}^jl_i,$$
so $C^2=-4$ as desired.

By Lemma~\ref{B dot D changes}, a singularity $p$ of $B$ on $D$ may be an $A_n$, $D_n$,  $E_6$, or $E_7$ singularity, as long as the branches of $B$ intersect $D$ in such a way that the multiplicity of $B$ and $D$ at $p$ is even. By considering the local equations of each type of ADE singularity, we can determine all possible ways $B$ and $D$ may intersect, and also have both branches of $\tilde{f}^*(\tilde{D})$ over $p$ smooth. As an example, suppose that $B$ has a $D_n$ singularity at $p\in D$. We can choose local coordinates around $p$ so that the local equation of $B$ is $y(x^2+y^{n-2})$. Since $(B\cdot D)_p$ is even, the curve $D$ must be tangent to one of the tangent cones of $B$ at $p$. We can then write  the equation of $D$ at $p$ as either $x-f(y)$ or $y-g(x)$ where $f(y)$ (or $g(x)$) has minimal degree $k\ge 2$. Then $(B\cdot D)_p$ is the minimal degree of either $y(f(x)^2+y^{n-2})$ or $g(x)(x^2+(g(x))^{n-2}$. Using the fact that $(B\cdot D)_p$ is even, we can then determine the possibilities for $k$ and for each such $k$ resolve the singularity of $B$ as described above. This tells us in particular how $\tilde{D}$ intersects the exceptional locus of $\sigma$. The same analysis may be used for the other ADE singularities. See Figure~\ref{even case resolution} for the list of possible singularities as well as how the curve $\tilde{D}$ intersects the exceptional locus of $\sigma$. We remark that this analysis is more general in that it depends only on the fact that both branches of $\tilde{f}^*(\tilde{D})$ over $p$ are smooth, and not on the self-intersection of $\tilde{D}$.

\begin{figure}
\includegraphics[scale=.5]{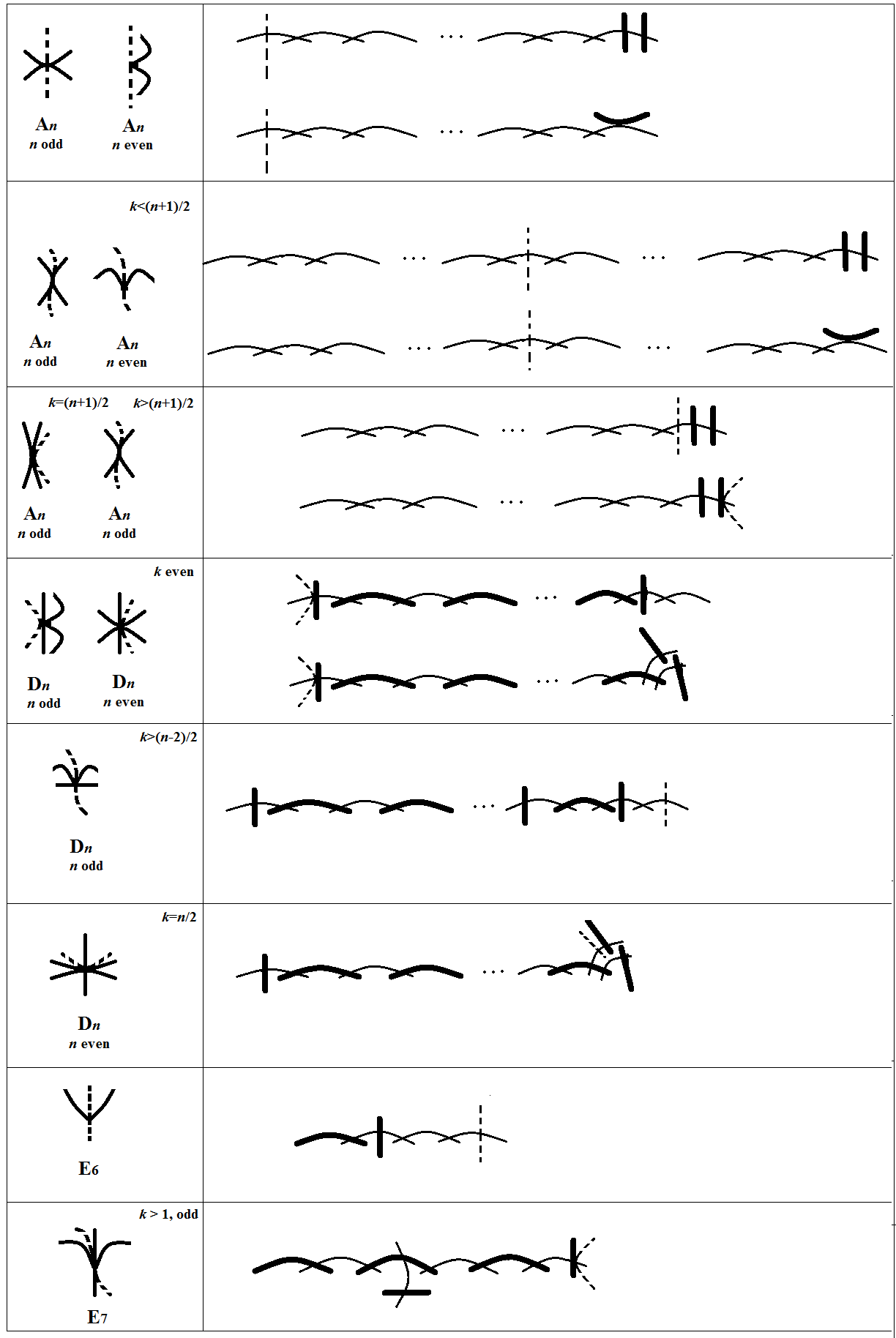}
\caption{On the left, all possible intersections of $B$ and $D$ if $(B\cdot D)_p$ is even. Here, $k$ denotes the multiplicity with which $D$ intersects the tangent cone of the given branch of $B$. On the right, the resolution of $B$ together with the exceptional curves and the proper transform of $D$ (dashed). The solid concave down curves denote exceptional divisors of $\sigma$. The branch locus of $\tilde{f}$ is denoted by solid bold curves. \label{even case resolution}}
\end{figure}

\textbf{Case III.} If $C\subset R$ then $f^*(D)=2C$, and so 
$$2\Delta\cdot D=K_Y\cdot f^*D=K_Y\cdot 2D=4.$$
Since $D$ is irreducible, we must have $D\sim \Delta$. The fact that $D\subset B$ implies that $B= D + \bar{B}$ where $\bar{B}$ is in the linear system 
$|5\Delta|$, so $D\cdot \bar{B}=10$. By Lemma~\ref{D in B}, 
we have
\begin{equation*}
\tilde{D}^2 = D^2-(\bar{B}\cdot  D)= -8
\end{equation*}
as desired.

The generic $\bar{B}$ intersects $D$ intersect in 10 distinct points and so the double cover $Y$ has 10 $A_1$ singularities.
\end{proof}

\subsection{Dimension counts}\label{dimension counts}

We begin by showing that every stable numerical quintic surface $W$ whose unique non Du Val singularity is a $\frac{1}{4}(1,1)$ singularity corresponds, up to isomorphism, to a unique triple $(Z, B, D)$, where $Z$ is a quadric surface, $B\subset  Z$ is a divisor in $|6\Delta|$, and $D\subset Z$ is the image of the $(-4)$-curve on the minimal resolution of $W$. This observation allows us to count the dimensions of a number of important loci in $\overline{\mathcal{M}}_{5,5}$.

\begin{Lemma}\label{triples}
Suppose that $X$ and $X'$ are the minimal resolutions of stable numerical quintic surfaces $W$ and $W'$, each of which has a unique $\frac{1}{4}(1,1)$ singularity and no other non Du Val singularities. Let $C$ and $C'$ be the $(-4)$-curves on $X$ and $X'$, respectively. Let $[W]$ and $[W']$ be the points of $\overline{\mathcal{M}}_{5,5}$ corresponding to $W$ and $W'$, respectively. The following are equivalent:
\begin{enumerate}
\item[1)] $[W]=[W']$.
\item[2)] There is an isomorphism $\theta: X\rightarrow X'$ such that $\theta(C)=C'$.
\item[3)] The triples $(Z, B, D)$ and $(Z', B', D')$ corresponding to $X$ and $X'$ are isomorphic; that is, there is an isomorphism $\eta: Z\rightarrow Z'$ such that $\eta(B)=B'$ and $\eta(D)=D'$.
\end{enumerate}
\end{Lemma}
\begin{proof}
1) $\iff$ 2) is clear. 

3) $\Rightarrow$ 2) follows by construction of $X$ and $X'$ from the triples given. For 2) $\Rightarrow$ 3), suppose that $\theta: X\rightarrow X'$ is an isomorphism such that $\theta(C)=C'$.
Let $Y$ and $Y'$ be the canonical models of $X$ and $X'$, respectively, and denote by $\bar{C}$ and $\bar{C'}$ the images of $C$ and $C'$, respectively. Then the isomorphism $\theta$ induces an isomorphism of $Y$ sending $Y$ to $Y'$ and $\bar{C}$ to $\bar{C'}$. The map $\phi_{K_Y}$ is a double cover $f:Y\rightarrow Z$, where $Z$ is either a quadric cone or a smooth quadric. Thus, the isomorphism $\theta$ induces an isomorphism $\eta:Z\rightarrow Z'$. Moreover, if $(Z, B,D)$ and $(Z', B', D')$ are the triples corresponding to $X$ and $X'$ under the correspondence of Theorem~\ref{classify}, then since $\theta(C)=C'$, we have $\eta(B)=B'$ and $\eta(D)=D'$, so the triples are isomorphic. 
\end{proof}

Let $p$ be a $\frac{1}{4}(1,1)$ singularity on a stable numerical quintic surface $W$, let $X$ be its minimal resolution, and let $C$ denote the $(-4)$ curve on $X$. We call $W$ a surface of type
\begin{itemize}
\item 1  if $Z=\mathbb{F}_0$, $D\sim \Delta$ and $B$ and $D$ intersect as in Figure~\ref{BandD}(d).
\item 1'  if $Z=\mathbb{F}_0$, $D\sim \Delta$ and $B$ and $D$ intersect as in Figure~\ref{BandD}(e).
\item 1'' if $Z=\mathbb{F}_2$, $D\sim \Delta$ and $B$ and $D$ intersect as in Figure~\ref{BandD}(d).
\item 1''' if $Z=\mathbb{F}_0$, $D\sim \Delta$, there exists $p\in B\cap D$ with $(B\cdot D)_p=4$, and $B$ and $D$ intersect as in Figure~\ref{BandD}(f).
\item 2a if $Z=\mathbb{F}_0$, $D$ is a fiber, and $B$ and $D$ intersect as in Figure~\ref{BandD}(a).
\item 2a' if $Z=\mathbb{F}_0$, $D$ is a fiber and $B$ and $D$ intersect as in Figure~\ref{BandD}(b). 
\item 2a'' if $Z=\mathbb{F}_0$, $D$ is a fiber, $B$ has an $A_2$ singularity along $D$ and $B$ and $D$ intersect as in Figure~\ref{BandD}(c).
\item 2b if $Z=\mathbb{F}_2$, $D$ is a fiber, and $B$ and $D$ intersect as in Figure~\ref{BandD}(a).
\end{itemize}

\begin{figure}[h]
\centering
\includegraphics[scale=.5]{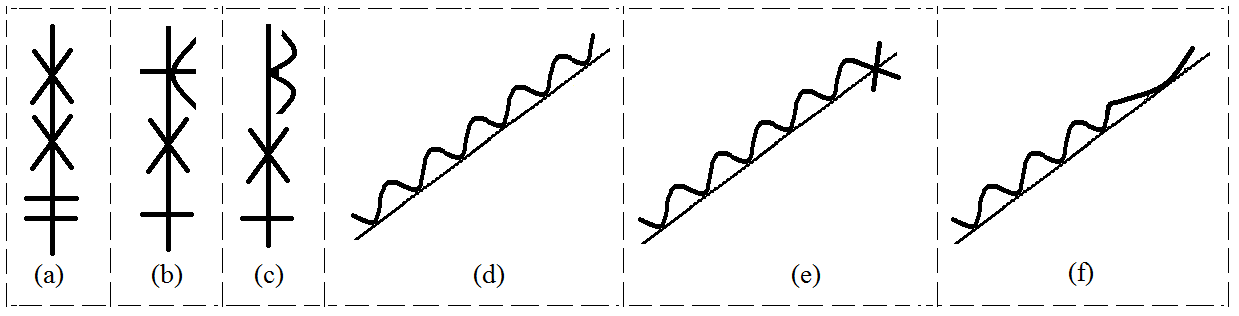}
\caption{Six ways $B$ and $D$ may intersect.\label{BandD}}
\end{figure}

\begin{Lemma}\label{counting} The stable numerical quintic surfaces of types 1 and 2a correspond to $39$-dimensional loci in $\overline{\mathcal{M}}_{5,5}$. Those of types 1', 1'', 1''', 2a', 2a'', and 2b correspond to $38$-dimensional loci in $\overline{\mathcal{M}}_{5,5}$. All other types of stable numerical quintic surfaces with a unique $\frac{1}{4}(1,1)$ singularity correspond to loci of dimension less than 38.
\end{Lemma}

\begin{proof}

Lemma~\ref{triples} implies that each triple $(Z,B,D)$ of Theorem~\ref{classify} corresponds to a unique stable numerical quintic surface, up to automorphisms of $Z$. We count the dimension of such triples in the given cases. The main difficulty is to check that requiring that the branch divisor obtain different types of singularities at different points imposes independent conditions on $B$. 

To create a triple $(Z, B, D)$:

\noindent 1. Fix a smooth or singular quadric $Z$.

\noindent 2. Choose a divisor $D\sim\Delta$ or $D\sim \Gamma$. Then by Riemann-Roch, since $K_Z=-2\Delta$ and $\Delta^2=2$, we have $h^0(Z, \mathcal{O}(D))=4$ if $D\sim \Delta$ and $h^0(Z, \mathcal{O}(D))=2$ if $D\sim \Gamma$. Projectivizing gives a 3-dimensional space of choices if $D\sim \Delta$ and a $1$-dimensional space if $D\sim \Gamma$. 

\noindent 3. Choose $k$ points on $D$ (through which $B$ will eventually pass). Since $D\simeq \mathbb{P}^1$, we have a $k$-dimensional space of choices for these points. (If $Z$ is a cone, we can choose the $k$ points so that none of them are the singularity of $Z$.)

\noindent 4. Choose a divisor $B$:
\begin{enumerate}
\item[4a.] To obtain $D\not\subset B$, choose $B\sim6\Delta$. Again by Riemann-Roch, $h^0(Z, \mathcal{O}(B))=49$. Projectiving gives a 48 dimensional space of possible branch curves $B$.
\item[4b.] To obtain $D\subset B$, choose $B'\sim 5\Delta$. By Riemann-Roch, $h^0(Z, \mathcal{O}(B'))=36$. Projectivizing gives a $35$-dimensional space of possible branch curves $B'$. By abuse of notation, take $B=B'$ (and note that the resulting triple will be of the form $(Z, B'+D, D)$, or with our abuse of notation, $(Z, B+D, D)$).
\end{enumerate}

\noindent 5. Consider the restriction exact sequence 
$$0\rightarrow \mathcal{O}_Z(B-D)\rightarrow\mathcal{O}_Z(B)\rightarrow\mathcal{O}_D(B)\rightarrow 0.$$
By Kodaira vanishing, $H^1(Z, \mathcal{O}_Z(B-D))=0$. Thus, the map 
$$H^0(Z, \mathcal{O}_Z(B))\rightarrow H^0(D,\mathcal{O}_D(B))$$ 
is surjective, and so we can find a curve $B\in |6\Delta|$ (or $B'\in |5\Delta|$) such that the restriction of $B$ to $D$ passes through any $m$ points on $D$, counted with multiplicities, where $m=(B\cdot D)$.
Thus, the requirement that $B$ pass through the given $m$ points, counted with multiplicities, is a codimension $m$ condition.

\noindent 6. The group of automorphisms of $Z$ is $6$-dimensional if $Z$ is smooth and $7$-dimensional if $Z$ is a cone. Thus, modding out by automorphisms of $Z$ is either a codimension $6$ condition or codimension $7$ condition.

Triples $(Z, B, D)$ where $D\subset B$ give a locus of dimension at most $3+10+35-10-6=32$, so we can assume for the rest of the proof that $D\not\subset B$.

\noindent 7. So far there is no guarantee that the most general $B$ is smooth at any given point, nor is it immediate that imposing the condition that $B$ obtain a certain mild singularity at a given point does not impose conditions on $B$ at the other $k-1$ points. Provided the multiplicity at each point is small enough, the fact that these conditions \emph{are} linearly independent follows from the fact that $B$ is sufficiently big. That is, for $n\le 5$, the divisor $B-nD$ is big and nef, so the cohomology group $H^1(Z, \mathcal{O}_Z(B-nD))$ is zero by Kodaira vanishing. Thus, the map 
$$H^0(Z, \mathcal{O}_Z(B))\rightarrow H^0(D,\mathcal{O}_{nD}(B))$$ 
induced by the restriction $\mathcal{O}_Z(B)\rightarrow\mathcal{O}_{nD}(B)$ is surjective. This means that we can choose $B$ in such a way that we can require the degree $1, 2,\ldots, n-1$ parts of the ``Taylor expansion" of its equation 
$$s|_{nD}=s_0+s_1d+s_2d^2+...$$
to be of any form we desire, where $d\in H^0(Z, \mathcal{O}_Z(D))$ is the equation of $D$ and $s_i\in H^0(D, \mathcal{O}_{D}(B-iD))$. 

Suppose we want to impose the condition that $B$ acquires a node at a given point $p$ for which $(B\cdot D)_p=2$. This is equivalent to requiring that the linear term in its Taylor expansion vanish at $p$, and that the discriminant of the quadratic term be non-vanishing at $p$. Therefore, this condition has expected codimension $1$. Since this is a requirement on the degree $1$ and $2$ parts of the Taylor expansion, taking $n=3$ implies that the requirements that $B$ be either smooth or obtain at most a node at each of its points are linearly independent conditions. That is, the condition that $B$ acquire a node at a point with multiplicity $2$ is indeed a codimension $1$ condition. 

Similarly, the requirement that $B$ acquire an $A_2$ singularity at a point $p$ for which $(B\cdot D)_p=2$ is equivalent to requiring that the linear term in its Taylor expansion vanish at $p$, the discriminant of the quadratic term also vanish at $p$, and the cubic term be nonvanishing. Since this is a requirement on the part of the Taylor expansion of degrees $1$, $2$, and $3$, taking $n=4$ implies that $B$ acquiring an $A_2$ singularity at the desired point is a codimension $2$ condition.

Requiring $B$ to have a node at a point $p$ for which $(B\cdot D)_p=3$ is equivalent to forcing the linear term in its Taylor expansion to vanish at $p$, and the coefficient of one monomial in the quadratic term to vanish at $p$. Again, this is a requirement on the degree $2$ part of the Taylor expansion, so taking $n=3$ implies that this is a codimension $2$ condition that does not impose conditions on the other points of $B\cap D$.

Let $l$ be the dimension of the set of triples such that $|B\cap D|=k$ (set theoretically). Then
$$ l=
\begin{cases}
33+k & \mbox{if } D\in|\Delta| \mbox{ on } \mathbb{F}_0\\
32+k & \mbox{if } D\in|\Delta| \mbox{ on } \mathbb{F}_2\\
37+k & \mbox{if } D\in|\Gamma| \mbox{ on } \mathbb{F}_0\\
36+k & \mbox{if } D\in|\Gamma| \mbox{ on } \mathbb{F}_2\\
\end{cases}.
$$
Thus, if $m$ is the codimension of the set of triples such that $B$ has prescribed singularities, then in order for the set of such triples to have dimension $38$ or $39$, we have 
\begin{eqnarray*}
m= k-6 \mbox{ or } m=k-5 & \mbox{if } D\in|\Delta| \mbox{ on } \mathbb{F}_0\\
m=k-6 \mbox{\hspace{1cm}} & \mbox{if } D\in|\Delta| \mbox{ on } \mathbb{F}_2\\
m=k-1 \mbox{ or } m=k-2 & \mbox{if } D\in|\Gamma| \mbox{ on } \mathbb{F}_0\\
m=k-2 \mbox{ or } m=k-3 & \mbox{if } D\in|\Gamma| \mbox{ on } \mathbb{F}_2
\end{eqnarray*}.

In particular, we see that if $D\sim \Delta$, then since $k\le 6$, we have $m=0$ or $1$. If $D\sim \Gamma$, then since $k\le4$, we have $m\le 3$.

For instance, the dimension of the locus of type 1 surfaces is $3+6+48-12-6=39$, and of type 1' surfaces is $3+6+48-12-1-6=38$. 

The dimension of the locus of type 2a surfaces is $1+4+48-6-1-1-6=39$, and of type 2b surfaces is $1+4+48-6-1-1-7=38$. 

Working through each of the remaining possibilities in Theorem~\ref{classify} gives the desired result.
\end{proof}

\subsection{The closures of the loci of type 1 and 2a surfaces}\label{1 and 2a closure}

The goal of this subsection is to prove the following. 

\begin{Theorem}\label{types 1 and 2a are main} Let $W$ be a stable numerical quintic surface  corresponding to the triple $(Z, B, D)$, and let $[W]$ denote its corresponding point in $\overline{\mathcal{M}}_{5,5}$. If $D\sim \Gamma$, then $[W]$ is in the closure of the locus of 2a surfaces. If $D\sim \Delta$, then $[W]$ is in the closure of the locus of surfaces of type 1. Thus, the closures of the loci of surfaces of types 1 and 2a contain all surfaces whose unique non Du Val singularity is a $\frac{1}{4}(1,1)$ singularity.
\end{Theorem}

We begin by showing that if $W$ is a surface corresponding to the triple $(Z, B, D)$ for $D\sim \Delta$ (respectively, $D\sim\Gamma$) then there is a family of triples $(\mathcal{Z}, \mathcal{B},\mathcal{D})$ with special fiber $(Z, B, D)$ whose general fiber is a triple corresponding to a surface of type 1 (respectively, 2a). We then take a double cover of $\mathcal{Z}$ branched over $\mathcal{B}$, followed (after a possible finite base change) by a simultaneous resolution of Du Val singularities. Ideally, this would give the minimal resolution of the desired family of stable numerical quintic surfaces, and then contracting a section of $(-4)$ curves would give the family itself. However, we may first need to do a sequence of flops in order to guarantee existence of a family of $(-4)$-curves with irreducible special fiber. The bulk of the work is showing that such a sequence exists. 

\begin{Lemma}\label{triples connect} Let $(Z, B, D)$ be a triple corresponding to a stable numerical quintic surface. If $D\sim \Delta$ (respectively, $D\sim\Gamma$), then $(Z,B,D)$ is the special fiber of a family of triples $(\mathcal{Z}, \mathcal{B}, \mathcal{D})$ with general fiber a triple corresponding to a stable numerical quintic surface of type 1 (respectively, 2a). 
\end{Lemma}
\begin{proof}
Suppose that $D\sim \Delta$ and consider the restriction sequence
$$0\rightarrow \mathcal{O}_Z(6\Delta-D)\rightarrow \mathcal{O}_Z(6\Delta)\rightarrow\mathcal{O}_D(12)\rightarrow 0.$$
Since $H^1(Z, \mathcal{O}_Z(6\Delta-D))=0$, the induced map
$$r: H^0(Z, \mathcal{O}_Z(6\Delta))\rightarrow H^0(D, \mathcal{O}_D(12))$$
is surjective. 

Let $T\subset H^0(Z, \mathcal{O}_Z(6\Delta))$ be the locus of effective divisors with at most Du Val singularities, and let $U\subset V= H^0(D,\mathcal{O}_D(6))$ be the locus of effective divisors consisting of six distinct points. Note that $U$ is open in $V$ and $T$ is open in $H^0(Z, \mathcal{O}_Z(6\Delta))$. Let $V\rightarrow H^0(D, \mathcal{O}_D(12))$ be the map sending a section to its square. Projectivizing, we obtain an injective map $i: \mathbb{P}(V)\rightarrow \mathbb{P}(H^0(D, \mathcal{O}_D(12)))$ which sends an effective divisor $E$ to the divisor $2E$. Being a projective map, the map $i$ is a closed. Therefore,
$$\overline{r^{-1}(i(U))\cap T}\cap T=r^{-1}(i(V))\cap T.$$
That is, the locus of branch divisors with at most Du Val singularities that intersect $D$ in any six points, with even multiplicities, is the closure of the locus of branch divisors with at most Du Val singularities that intersect $D$ in six distinct points with multiplicity two at each point.

A similar argument holds for $D\sim \Gamma$, where we consider instead the restriction map 
$$\mathcal{O}_Z(6\Delta)\rightarrow\mathcal{O}_D(6),$$
take $U\subset H^0(D, \mathcal{O}_D(6))$ to be the locus of effective divisors consisting of four distinct points, two of which have multiplicity two, and take $i$ to be the identity map.
\end{proof}

Suppose that $p$ is a Du Val singularity of $B$ on $D$, and let $\sum E_i \subset X$ be the exceptional divisor of $\phi$ over $p$. Let $Q=\sum \mathbb{Z}E_i$ denote the corresponding root lattice of type $A$, $D$, or $E$, with positive-definite inner product $\circ$. Note that $\circ$ is the negative of the intersection product $\cdot$ on the Picard group of $X$. That is,
$$E_i\cdot E_j =-E_i\circ E_j.$$
Let $P=\textrm{Hom}(Q, \mathbb{Z})=\sum \mathbb{Q}\omega_i$ be the corresponding weight lattice, where $\omega_i$ are the fundamental weights, satisfying $\omega_i(E_j)=\delta_{ij}$.
By means of the inner product, we identify $P$ with a lattice in $Q\otimes\mathbb{Q}$, and in this way think of $Q$ as a sublattice of $P$. In particular, we may extend in a unique way the inner product on $Q$ to $P$. 

Given any effective divisor $C\in \textrm{Pic}(X)$, we can associate to $C$ the weight $\omega_C$ satisfying
$$\omega_C(E_i)=C\cdot E_i$$
for all $i$. 
As an example, if $C$ is a divisor such that for some $j$ we have $C\cdot E_j = 1$ and $C\cdot E_i=0$ for $i\neq j$, then under this correspondence, we have $\omega_C=\omega_j$.

We will use the classification of Theorem~\ref{classify} to describe all ways the $(-4)$-curve $C\subset X$ intersects the exceptional divisor $\sum E_i$. We label the dual graphs of the exceptional divisor of a Du Val singularity as follows:

\begin{center}
\begin{tikzpicture}

	\draw (0,0) -- (1,0);
         \draw (2,0) -- (2.7,0);
         \draw (1,0) -- (2,0);
	\draw (3.3, 0) -- (4,0);
	\draw (4,0) -- (5,0);
	
	\node[label={1}] at (0,0) {$\bullet$};
	\node at (1,0) {$\bullet$};
	\node at (2,0) {$\bullet$};
	\node at (4,0) {$\bullet$};
	\node[label={$n$}] at (5,0) {$\bullet$};

	\node at (3,0) {$\cdots$};
	\node at (-.5,0) {$A_{n}$};

	\draw (7,0) -- (9.7,0);
	\draw (10.3, 0) -- (11,0);
	\draw (11,0) -- (12,-.5);
	\draw (11,0) -- (12,.5);
	
	\node[label={1}] at (7,0) {$\bullet$};
	\node at (8,0) {$\bullet$};
	\node at (9,0) {$\bullet$};
	\node [label={$n-2$}]at (11,0) {$\bullet$};
	\node[label={$n$}] at (12,-.5) {$\bullet$};
	\node[label={$n-1$}] at (12,.5) {$\bullet$};

	\node at (10,0) {$\cdots$};
	\node at (6.5,0) {$D_{n}$};
	 
\end{tikzpicture}

\begin{tikzpicture}

	\draw (0,0) -- (4,0);
	\draw (2,0) -- (2,1);
	
	\node[label={1}] at (0,0) {$\bullet$};
	\node at (1,0) {$\bullet$};
	\node at (2,0) {$\bullet$};
	\node[label={6}] at (2,1) {$\bullet$};
	\node at (3,0) {$\bullet$};
	\node[label={5}] at (4,0) {$\bullet$};

	\node at (-.5,0) {$E_6$};

	\draw (7,0) -- (12,0);
	\draw (9,0) -- (9,1);
	
	\node[label={1}] at (7,0) {$\bullet$};
	\node at (8,0) {$\bullet$};
	\node at (9,0) {$\bullet$};
	\node[label={7}] at (9,1) {$\bullet$};
	\node at (10,0) {$\bullet$};
	\node at (11,0) {$\bullet$};
	\node[label={6}] at (12,0) {$\bullet$};
	
	\node at (6.5,0) {$E_7$};
	 
\end{tikzpicture}

\begin{tikzpicture}

	\draw (0,0) -- (6,0);
	\draw (2,0) -- (2,1);
	
	\node[label={1}] at (0,0) {$\bullet$};
	\node at (1,0) {$\bullet$};
	\node at (2,0) {$\bullet$};
	\node at (3,0) {$\bullet$};
	\node[label={8}] at (2,1) {$\bullet$};
	\node at (4,0) {$\bullet$};
	\node at (5,0) {$\bullet$};
	\node[label={7}] at (6,0) {$\bullet$};
	
	\node at (-1,0) {$E_8$};
	 
\end{tikzpicture}
\end{center}

It will be useful to have a description of the root systems $A_n$, $D_n$, and $E_n$. 

For $A_n$, let $e_0, \ldots, e_n$ be an orthonormal basis of $\mathbb{R}^{n+1}$. A basis of simple roots $E_1, \ldots, E_n$ of $A_n$ is given by $E_i=e_{i-1}-e_{i}$. 

For $D_n$, we let $e_1, \ldots, e_n$ be an orthonormal basis of $\mathbb{R}^n$. A basis of simple roots is given by $E_i=e_{i}-e_{i+1}$ for $i =1, \ldots n-1$ and $E_n=e_{n-1}+e_n$.

To describe the root system of $E_n$, let $S$ be a Del Pezzo surface given by the blowup of $\mathbb{P}^2$ in $n$ points in general position. Then the root lattice $E_n$ can be described as the orthogonal complement of the canonical class $K_S$ in  $\textrm{Pic}(S)$, after a sign change. Let $h, e_1, \ldots e_n$ be a basis of $\textrm{Pic}(S)$, where $e_i$ are classes of the exceptional curves on $S$, and $h$ is the pullback of a hyperplane section. We note that $K_S=-3h+\sum_{i=1}^n e_i$. With respect to the pairing $\circ$ satisfying $e_i\circ e_j=\delta_{ij}$, $e_i\circ h=0$, and $h^2=-1$ (the negative of the intersection pairing on $S$), we can describe the simple roots of $E_n$ as 
$$E_i=e_i-e_{i+1} \textrm{ for } i=1,\ldots, n-1$$
$$E_n=h-e_1-e_2-e_3.$$

Our main tool for proving the existence of the desired sequence of flops is the following.

\begin{Lemma}\label{why flopping helps}
Let $\mathcal{X}$ be a family of surfaces over the unit disk in $\mathbb{C}$ with special fiber $X$. Let $\mathcal{D}\subset \mathcal{X}$ be a divisor whose restriction to $X$ is an effective sum of curves $\mathcal{D}|_X=C+\sum a_iE_i\subset X$, where the $E_i$ are distinct $(-2)$-curves, 
and $C\cdot E_i\ge 0$ for all $i$. Suppose there exists $j$ such that $\mathcal{D}\cdot E_j <0$. Let $\phi: \mathcal{X}'\rightarrow \mathcal{X}$ be the flop of $E_j$. Then 
$$(\phi^*\mathcal{D})|_X=C+a_j'E_j +\sum_{i\neq j} a_i E_i $$
where $a_j>a_j'\ge 0$. That is, flopping the curve $E_j$ results in a divisor $\phi^*\mathcal{D}$ with at least one fewer $(-2)$-curve on the special fiber.
\end{Lemma}
\begin{proof} Suppose that $a=\mathcal{D}\cdot E_j$ is negative. Flopping $E_j$ does not change the coefficients $a_i$ for $i\neq j$, and so 
$$(\phi^*\mathcal{D})|_X=(C+a_j'E_j+\sum_{i\neq j} a_i E_i)$$
for some nonnegative integer $a_j'$. Moreover, because $\mathcal{D}\cdot E_j<0$, flopping $E_j$ gives $\phi^*\mathcal{D}\cdot E_j> 0$. Thus
$$(C+a_j'E_j+\sum_{i\neq j} a_i E_i)\cdot E_j = a-(-2a_j)+(-2a_j')> 0,$$
so
$$a+2(a_j-a_j')> 0.$$
Since $a<0$, we have that $a_j> a_j'$.
\end{proof}

\begin{Lemma}\label{the odd case}
Let $(\mathcal{Z}, \mathcal{B},\mathcal{D})$ be a family of triples over the unit disk in $\mathbb{C}$, with general fiber corresponding to a stable numerical quintic surface of type 2a or 2b and such that $D=\mathcal{D}_0$ is a ruling. Let $\mathcal{Y}$ be the double cover of $\mathcal{Z}$ branched over $\mathcal{B}$. Then there exists, after a possible finite base change, a simultaneous resolution of singularities $\mathcal{X}\rightarrow\mathcal{Y}$ such the closure of the $(-4)$-curve on the general fiber is irreducible. 
\end{Lemma}

\begin{proof}

Let $(Z, B, D)$ be a triple corresponding to a stable numerical quintic surface, and suppose that $D$ is a ruling. Let $(\mathcal{Z}, \mathcal{B}, \mathcal{D})$ be a family of triples whose general fiber is of type 2a or 2b, and whose special fiber is the triple $(Z, B, D)$. Consider the maps
\[\xymatrix{
\mathcal{X} \ar[r]^{\psi} & \mathcal{Y}\ar[r]^f& \mathcal{Z}
}\]
where $f:\mathcal{Y}\rightarrow \mathcal{Z}$ is the double cover branched over $\mathcal{B}$, and  $\psi:\mathcal{X}\rightarrow \mathcal{Y}$ is any simultaneous resolution of Du Val singularities, which exists after a possible finite base change. Let $\mathcal{C}$ be the closure of the $(-4)$-curve on the general fiber of $\mathcal{X}$. Recall that $\mathcal{B}$ has two nodes lying on the general fiber of $\mathcal{D}$, and thus the general fiber of $\mathcal{X}$ contains two $(-2)$ curves intersecting the general fiber of $\mathcal{C}$. Let $\mathcal{E}_1$ and $\mathcal{E}_2$ be the closures of these $(-2)$-curves.
We claim that there exists a sequence of flops of $\mathcal{X}$ so that $\mathcal{C}$ has irreducible special fiber.

Let $C$ be the $(-4)$-curve on $X=\mathcal{X}_0$ and let $\sum E_i\subset X$ exceptional divisor over the singularities of $B$ on $D$. Then
$$\psi^{-1}(f^{-1}(\mathcal{D}))_0=C+\sum a_i E_i +\sum b_i E_i +\sum c_i E_i,$$
where
$$C+\sum a_i E_i=\mathcal{C}_0, \; \; \; \sum b_iE_i=(\mathcal{E}_1)_0, \; \; \; \sum c_iE_i=(\mathcal{E}_2)_0,$$
and the $a_i, b_i, c_i$ are nonnegative integers. Since the general fibers of $\mathcal{E}_1$ and $\mathcal{E}_2$ are disjoint $(-2)$-curves, the divisors $\sum b_iE_i$ and $\sum c_i E_i$
satisfy $(\sum b_iE_i)\cdot(\sum c_i E_i)=0$ and $(\sum b_iE_i)^2=(\sum c_i E_i)^2=-2$. Moreover
\begin{equation}\label{C + stuff is weight 0}
(C+\sum a_i E_i +\sum b_i E_i +\sum c_i E_i)\cdot E_j =0 \textrm{ for all $j$.}
\end{equation}
We pass to the weight lattice $\Lambda$ of the corresponding singularity (or singularities) on $Y=\mathcal{Y}_0$. Note that if $B$ has one singularity on $D$, then $\Lambda$ is irreducible. But if $B$ has two singularities, then the lattice $\Lambda$ is a direct sum $P\oplus P'$ of irreducible lattices, each corresponding to a singularity of $B$ on $D$. In either case, Equation~\eqref{C + stuff is weight 0} implies that the weight $$\omega_C-\sum a_i E_i -\sum b_i E_i -\sum c_i E_i\in \Lambda$$ is equal to $0$. Thus, as weights, we have
$$\omega_C-\sum a_i E_i =\sum b_i E_i +\sum c_i E_i.$$

We prove that if $(C+\sum a_i E_i)\cdot E_j \ge 0$ for all $j$, then $a_i=0$ for all $i$. Together with this, Lemma~\ref{why flopping helps} implies that if $a_i\neq 0$ for some $i$, then there exists a flop of $\mathcal{X}$ after which $\mathcal{C}_0$ will contain at least one fewer $(-2)$-curve. 

Suppose that for all $j$, we have $(C+\sum a_i E_i)\cdot E_j \ge 0$.  Passing again to the weight lattice, this is the statement that 
$$\left(\omega_C-\sum a_i E_i\circ E_j\right)\ge 0$$
for all $j$. The last inequality is equivalent to the statement that the weight 
$$\omega_C-\sum a_i E_i=\sum b_i E_i +\sum c_i E_i$$
is dominant. We therefore have 
\begin{eqnarray*}
\omega_C^2&=&\left(\sum b_i E_i +\sum c_i E_i\right)^2+2\left(\sum b_i E_i +\sum c_i E_i\right)\circ\left(\sum a_i E_i\right) +\left(\sum a_i E_i\right)^2\\
&=& 4 +2\left(\sum b_i E_i +\sum c_i E_i\right)\circ\left(\sum a_i E_i\right) +\left(\sum a_i E_i\right)^2\\
&\ge& 4,
\end{eqnarray*}
where equality holds if and only if $a_i=0$ for all $i$. Thus, our proof is complete once we show that $\omega_C^2=4$ independent of the singularity (or singularities) of $B$ on $D$. 

By Theorem~\ref{classify}, and referring to Figures~\ref{distinct points}, \ref{same points 4}, and~\ref{same points 5}, we have the following table detailing how $C$ intersects the curves $E_i$. We write $\omega_C|_P$ to denote the weight $\omega_C$ restricted to the irreducible lattice $P$ corresponding to the singularity of $B$, and note that if $\Lambda=P$ is irreducible (i.e., if $B$ has only one singularity on $D$), then $\omega_C|_{\Lambda}=\omega_C|_{P}=\omega_C$. Otherwise, $\omega_C=\omega_C|_{P}+\omega_C|_{P'}$ where $\omega_C|_{P}$ and $\omega_C|_{P'}$ are orthogonal weights in $\Lambda$, each corresponding to a singularity of $B$ on $D$.

\begin{table}[h!]\label{odd case table}
\begin{tabular}{c|c|c|c}
$(B\cdot D)_p$ & Singularity & Nonzero intersections $C\cdot E_i$  & $\omega_C|_P$\\
\hline
2 & $A_n$ & $C\cdot E_1=C\cdot E_n=1$ & $e_0-e_n$\\
\hline
3 & $A_1$ & $C\cdot E_1=1$ & $e_0-e_1$\\
3 & $A_2$ &  $C\cdot E_1=C\cdot E_2 =1$&$e_0-e_2$ \\
3 & $D_n$ & $C\cdot E_2=1$ & $e_1+e_2$\\
3 & $E_6$ &  $C\cdot E_6 =1$ & $-2h+\sum_{i=1}^6e_i$ \\
3 & $E_7$ &  $C\cdot E_1=1$ & $2h-\sum_{i=1}^7e_i$\\
3 & $E_8$ &  $C\cdot E_7=1$ & $3h-2e_8-\sum_{i=1}^7e_i$\\
\hline
4 & $A_n$, $n\ge4$ &  $C\cdot E_2=C\cdot E_{n-1}=1$ & $e_0+e_1-e_{n-1}-e_n$ \\
4 & $D_5$ &   $C\cdot E_4=C\cdot E_{5}=1$ & $e_1+e_2+e_3+e_4$\\
4 & $D_n$ &   $C\cdot E_1=2$ & $2e_1$\\
4 & $E_6$ &   $C\cdot E_1=C\cdot E_{5}=1$ & $2h+2e_1+\sum_{i=2}^5$ \\
\hline
5 & $A_3$ &   $C\cdot E_2=1$ & $e_0+e_1-e_2-e_3$\\
5 & $A_4$ &   $C\cdot E_2=C\cdot E_3=1$ & $e_0+e_1-e_3-e_4$ \\
5 & $D_n$, $n\ge 6$ &  $C\cdot E_4=1$ & $e_1+e_2+e_3+e_4$\\
5 & $D_n$ &  $C\cdot E_1=2$ & $2e_1$\\
5& $E_7$ & $C\cdot E_5=1$ & $3h-2e_6-2e_7-\sum_{i=1}^5e_i$\\
5 & $E_8$ &   $C\cdot E_1=1$ & $5h-e_1-2\sum_{i=2}^8e_i$\\
\end{tabular}
\end{table}

One quickly checks that if $(B\cdot D)_p$ is $2$ or $3$, then $(\omega_C|_P)^2=2$ and if $(B\cdot D)_p$ is $4$ or $5$, then $(\omega_C|_P)^2=4$. In the former case, $\omega_C$ is the sum of two perpendicular vectors $\omega_C|_P$ and $\omega_C|_{P'}$, and thus
$$\omega_C^2=(\omega_C|_P + \omega_C|_{P'})^2=4.$$ This completes the proof.
\end{proof}

\begin{Remark}
The proof of Lemma~\ref{the odd case} relies on the fact that in each case, the weight $\omega_C$ is dominant with square $4$. It is interesting to note that the vectors listed in Table~\ref{odd case table} are in fact all vectors with square $4$ or $2$ corresponding to dominant weights, up to the action of the corresponding Weyl group. 
\end{Remark}

In the case where $D\sim \Delta$ is a diagonal, have a more general statement. The proof will require the following.

\begin{Lemma}\label{from rep theory}~\cite[VIII, 7.3]{bourbaki} Let $\omega$ be a weight and suppose that $\omega-\sum b_j E_j$ is dominant. Then the weight $\omega-\sum b_j E_j$ is a (dominant) weight of the irreducible representation $V_{\omega}$ of highest weight $\omega$ of the corresponding Lie algebra.
\end{Lemma}

A weight $\omega$ is \emph{minuscule} if it is the only dominant weight of $V_\omega$. 

\begin{Theorem}\label{simultaneous resolution for even intersection} Let $Z$ be a smooth surface, $B$ a divisor on $Z$ with at most Du Val singularities, and $D$ a smooth irreducible divisor on $Z$. Let $(\mathcal{Z,B,D})$ be a family of triples over the unit disk in $\mathbb{C}$, with special fiber $(Z, B, D)$, and such that the divisors $\mathcal{D}_t$ and $\mathcal{B}_t$ are reduced, irreducible and smooth for $t\neq 0$. Suppose that at each point $p\in \mathcal{D}_t\cap \mathcal{B}_t$ over the general fiber, the local intersection $(\mathcal{D}_t\cdot\mathcal{B}_t)_p$ is even. Let $f:\mathcal{Y}\rightarrow \mathcal{Z}$ be the double cover branched over $\mathcal{B}$. Then there exists, after a possible finite base change, a simultaneous resolution of singularities $\psi:\mathcal{X}\rightarrow \mathcal{Y}$ such that the closure of one of the two components of $\psi^{-1}(f^{-1}(\mathcal{D}))_t$ over the general fiber has irreducible special fiber.
\end{Theorem}

\begin{proof}
We can work locally, so we assume that $B$ has a unique singularity at $p\in B\cap D$. Let
$f:\mathcal{Y}\rightarrow \mathcal{Z}$ be the double cover branched over $\mathcal{B}$ and  $\psi:\mathcal{X}\rightarrow\mathcal{Y}$ a simultaneous resolution of Du Val singularities of $\mathcal{Y}$, which exists after a possible finite base change. Let $\mathcal{C}_1$ and $\mathcal{C}_2$ be the closures in $\mathcal{X}$ of the two components of $\psi^{-1}(f^{-1}(\mathcal{D}))$ over the general fiber. 
Note that
$${f}^{-1}(\psi^{-1}(\mathcal{D}))_0=\left(C_1+\sum a_i E_i\right) +\left(C_2+\sum b_i E_i\right)$$
where $\sum E_i\subset X$ is the exceptional divisor over $p$, $C_1+\sum a_i E_i$ is the special fiber of $\mathcal{C}_1$, $C_2+\sum b_i E_i$ is the special fiber of $\mathcal{C}_2$,
and the $a_i$ and $b_i$ are nonnegative integers. 
If the special fiber of either $\mathcal{C}_1$ or $\mathcal{C}_2$ is irreducible, that is, if either all $a_i$ or all $b_i$ are zero, then we are done. Otherwise, we claim that there exists a sequence of flops $\phi:\tilde{\mathcal{X}}\rightarrow\mathcal{X}$ so that $\phi^*(\mathcal{C}_1)$ has irreducible special fiber. As in the proof of Lemma~\ref{the odd case}, we show that if $(C_1+\sum a_i E_i)\cdot E_j\ge0$ for all $j$, then $a_i=0$ for all $i$. We then apply Lemma~\ref{why flopping helps} to prove the claim.

Suppose that $(C_1+\sum a_i E_i)\cdot E_j\ge0$ for all $j$. Passing to the weight lattice, this is equivalent to the statement that the corresponding weight is dominant. By Lemma~\ref{from rep theory}, we have that $\omega=\omega_{C_1}-\sum a_i E_i$ is a dominant weight of the irreducible representation $V_{\omega_{C_1}}$. If $\omega_{C_1}$ is minuscule, then this implies that $a_i=0$ for all $i$ and we are done. 
We refer the reader to Figure~\ref{even case resolution}, which details all possible ways that $B$ and $D$ may intersect at $p$ together with how $\tilde{D}$ intersects the exceptional curves on $\tilde{Z}$ and the branch divisor of $\tilde{f}$. This gives the following table of weights $\omega_{C_1}$, depending on the singularity of $B$ at $p$.

\begin{center}
\begin{tabular}{c|c}
Singularity & $\omega_{C_1}$\\
\hline
$A_n$ & $\omega_i$ for any $i$ \\
$D_n$ & $\omega_1, \omega_{n-1}, \omega_n$ \\
$E_6$ & $\omega_1, \omega_5$ \\
$E_7$ & $\omega_6$
\end{tabular}
\end{center}

One can check that these are all minuscule fundamental weights. 
\end{proof}

\begin{Remark} The proof of Theorem~\ref{simultaneous resolution for even intersection} relies on the definition of a minuscule fundamental weight. It is interesting to note that the table at the end of the proof contains all minuscule fundamental weights of the simple Lie algebras.
\end{Remark}

We can now prove the main theorem of this section.

\begin{proof}[Proof of Theorem~\ref{types 1 and 2a are main}]

Let $W$ be a stable numerical quintic surface  corresponding to the triple $(Z, B, D)$, and let $[W]$ denote its corresponding point in $\overline{\mathcal{M}}_{5,5}$. By Lemma~\ref{triples connect}, there exists a family of triples $(\mathcal{Z, B, D})$ over the unit disk in $\mathbb{C}$ such that $(\mathcal{Z, B, D})_0=(Z, B, D)$ and if $D\sim \Gamma$ (respectively, $D\sim \Delta$), then the general fiber of $(\mathcal{Z, B, D})$ is a triple corresponding to a stable numerical quintic surface of type 2a (respectively, type 1).
Let $\mathcal{Y}$ be the double cover of $\mathcal{Z}$ branched over $\mathcal{D}$, and let $\tilde{\mathcal{X}}\rightarrow \mathcal{Y}$ be a simultaneous resolution of singularities of $\mathcal{Z}$, which exists after a possible finite base change. Let $\mathcal{C}$ be the closure of (one of the) $(-4)$-curve(s) on the general fiber of $\tilde{\mathcal{X}}$. By Lemma~\ref{the odd case} and Theorem~\ref{simultaneous resolution for even intersection}, there exists a sequence of flops $\phi:\mathcal{X}\rightarrow\tilde{\mathcal{X}}$ so that $\phi^*(\mathcal{C})$ has irreducible special fiber. Contracting $\phi^*(\mathcal{C})$ results in the desired family of stable numerical quintic surfaces.
\end{proof}


\section{Deformations of surfaces of types 1 and 2a}\label{wahldiv}

We describe the components of $\overline{\mathcal{M}}_{5,5}$ corresponding to surfaces of types 1 and 2a and show that their closures are generically Cartier divisors in the boundary of the type I and IIa components of $\mathcal{M}_{5,5}$. In Lemma~\ref{counting}, we showed that these components are both $39$-dimensional. 
In~\ref{families}, we construct explicit $\mathbb{Q}$-Gorenstein families of numerical quintic surfaces to show that these components are in the boundary of the respective components on $\mathcal{M}_{5,5}$. In~\ref{sheaf calculations}, we prove a number of technical results which we use in~\ref{1 and 2a} to show that the closures $\bar{1}$ and $\overline{\mbox{2a}}$ of these components are generically Cartier divisors. In particular, we show that the cohomology groups controlling obstructions to $\mathbb{Q}$-Gorenstein deformations of surfaces of types 1 and 2a vanish.

\subsection{Families of stable quintic surfaces}\label{families}

\subsubsection{Type 1}\label{Type 1}
We describe a family of quintic surfaces degenerating to a stable numerical quintic surface of type 1.

\begin{Theorem}\label{quintic example}
Consider the family $(\mathcal{X}, T)$ of surfaces 
$$S_t=\{q^2l+tqf+t^2g=0\}\subset \mathbb{P}^3 \times T_t$$
where $T$ is the unit disk in $\mathbb{C}$ and $f$ and $g$ are forms of degrees $3$, and $5$, respectively, such that
\begin{itemize}
\item $S_t$ is a smooth quintic surface for $t\in T^*$
\item $S_0$ is the union of a smooth double quadric $Z$ given by $q=0$ and a plane $L$ given by $l=0$ intersecting transversally.
\end{itemize} 
For general $f$ and $g$, the KSBA stable limit of the family $(\mathcal{X},T)$ is a stable numerical quintic surface of type 1. Conversely, any stable numerical quintic surface of type $1$ is the stable limit of such a family. 
\end{Theorem}
\begin{proof} The singular locus of $\mathcal{X}$ is the surface $Z$, so $\mathcal{X}$ is not normal. To compute the stable limit we first normalize the family. After normalization and an extremal contraction, we will see that the family of surfaces obtained has reduced special fiber and ample canonical class.

Let $\nu:\mathcal{X}^{\nu}\rightarrow \mathcal{X}$ be the normalization of $\mathcal{X}$. We determine the structure of $\mathcal{X}^{\nu}$. First note that the normalization is an isomorphism away from $Z$.

Let $U$ be a complex analytic neighborhood in $\mathcal{X}$ of a point $p \in Z$. Then on $U$, we can write 
$$q|_U=q_1+q_2, \;\; l|_U=l_0+l_1, \;\; f|_U=\sum_{i=0}^3 f_{i}, \;\; g|_U=\sum_{i=0}^5 g_{i}$$
where the subscripts indicate the degree of each term in linear coordinates centered at $p$. Giving $t$ weight 1, we can write the equation of $\mathcal{X}\cap U$ as  
$$q_1^2l_0+tq_1f_{0}+t^2g_{0}+\textrm{higher order terms}.$$
Let $B\subset Z$ be the ``discriminant curve" given by $\{f^2-4lg=0\}\subset Z\cap U$. If $p\not \in B$, then the equation of $\mathcal{X}\cap U$ factors into the product of two linear terms which are not equal. That is, $(p\in\mathcal{X})$ is locally analytically isomorphic to a threefold $\mathcal{Y}=(xy=0)\subset \mathbb{A}^4$. Thus, over the open set $Z\backslash B $,  the special fiber $\mathcal{X}_0^\nu$ is an unramified double cover of $Z\backslash B$.

Now consider a point $p\in B$. The equation of $\mathcal{X}\cap U$ may be written locally analytically as
$$
h=\begin{cases}
(q+\frac{1}{2}f_{0} t)^2+\textrm{h.o.t.} & \mbox{if } p\not\in L\\
t^2+\textrm{h.o.t.} & \mbox{if } p\in L
\end{cases}
$$
Thus, in order to determine the structure of $\mathcal{X}^{\nu}$ near $p$, we must consider the degree three part of $h$: 
$$h_3=q_1^2l_1+2q_1q_2l_0+tq_1f_{1}+tq_2f_{0}+t^2g_{1}.$$

Suppose first that $p\not\in L$. Then we may assume that $l_0=1$ and complete the square in the first few terms of $h$:
\begin{eqnarray*}
h&=&(q_1+\frac{1}{2}tf_{0})^2 + 2q_2(q_1+\frac{1}{2}tf_{0}) +q_1^2l_1+tq_1f_{1} + t^2g_{1} +\textrm{h.o.t.}\\
&=&(q_1+\frac{1}{2}tf_{0} +q_2)^2 +q_1^2l_1+tq_1f_{1} +t^2g_{1}+\textrm{h.o.t.}
\end{eqnarray*}
Let $y=q_1+\frac{1}{2}tf_{0}$ and note that $y$ is a linear form. This last equation now becomes
$$h=(y +q_2)^2 +y^2\alpha+yt\beta +t^2\gamma+\textrm{h.o.t.}$$
where 
$$\alpha=l_1,$$
$$\beta=f_{1}-l_1f_{0},$$ and 
$$\gamma=g_{1}-\frac{1}{2}f_{0}(f_{1}+\frac{1}{2}l_1)$$ are linear forms. Finally we can rewrite this as
\begin{eqnarray*}
h&=&(y +q_2)^2 +(y+q_2)(y\alpha+t\beta) -q_2(y\alpha+t\beta)+t^2\gamma + \textrm{h.o.t.}\\
&=&[(y+q_2)+\frac{1}{2}(y\alpha+t\beta)]^2+t^2\gamma+ \textrm{h.o.t.}\\
&=& z^2+t^2\gamma +\textrm{h.o.t.}
\end{eqnarray*}
where $z$ is a linear form.  Thus, in a complex analytic neighborhood of any point $p\in B\backslash L$, the threefold $\mathcal{X}$ is locally analytically isomorphic to the threefold $\mathcal{Y}=\{ z^2-t^2\gamma=0\}\subset\mathbb{A}^4_{\gamma, t, z, s}$ which is the product of $\mathbb{A}^1$ with the Whitney umbrella, or pinch point. The normalization of $\mathcal{Y}$ is $\mathbb{A}^3_{u,v,w}$ with normalization map  $(u,v,w)\mapsto(u^2, v, uv, w)$.
Here, the quadric $Z$ corresponds to the locus $(z=t=0)\subset\mathcal{Y}$, so away from $L$, the normalization $\mathcal{X}_0^{\nu}$ of $\mathcal{X}_0$ is the double cover of the smooth quadric $Z$, ramified along the discriminant curve $B$.  
Since $B\subset Z$ is a curve of degree 12, the surface $\mathcal{X}_0^{\nu}$ is the double cover of $Z\simeq \mathbb{P}^1\times\mathbb{P}^1$, 
ramified along a divisor in the linear system $|6\Delta|$.

Next, we consider a point $p\in L$. We begin by assuming that $p\in L\cap Z\backslash B$. Then $l_0=0$ and $f_{0}\neq 0$, so we can assume that $f_{0}=1$ and we have
$$h=tq_1+t^2g_{0}+q_1^2l_1+tq_1f_{1}+tq_2+t^2g_{1}+ \textrm{ h.o.t.}.$$
By choosing $g$ sufficiently general, we can assume that $g_{0}\neq 0$ and so take $g_{0}=1$. 
Thus, $h$ factors as
\begin{eqnarray*}
h&=& tq_1+t^2+q_1^2l_1+tq_1f_{1}+tq_2+t^2g_{1}+ \textrm{ h.o.t.}\\
&=& (t+q_1l_1+ \textrm{ h.o.t.} )\cdot(t+q_1 -q_1l_1+\textrm{ h.o.t.})
\end{eqnarray*}
The linear term of each factor is unique up to multiplication by a nonzero constant. In particular, the second factor does not vanish identically along $L$. Since $h(p)=0$ the first term must vanish along $L$. Thus, the normalization of $(p\in \mathcal{X})$ is an unramified double cover of $Z\backslash B$, of which one component (the component corresponding to the first factor of $g$ above) contains the entire proper transform of $L\backslash B$. 

For the six points $p\in L\cap B$, we have $l_0=0$ and $f_{0}=0$. We suppose first that $p$ is a smooth point of $B$. Then we can assume that $g_{0}=1$ and so we can write the local equation of $\mathcal{X}$ as
$$h=t^2+tq_1f_{1}+q_1^2l_1+t^2g_{1}+\textrm{ h.o.t.}$$
Completing the square gives
$$h=(t+\frac{1}{2}q_1f_{1})^2+q_1^2l_1+t^2g_{1}+\textrm{ h.o.t.}$$
Let $\alpha=t+\frac{1}{2}q_1f_{1}$ and note that we can write $t=\alpha-\frac{1}{2}q_1f_{1}$. Then $h$ can be rewritten in terms of $\alpha$ as
\begin{eqnarray*}
h&=&\alpha^2+q_1^2l_1+(\alpha-\frac{1}{2}q_1f_{1})^2g_{1}+\textrm{ h.o.t.}\\
&=& \alpha^2(1+g_{1})+q_1^2l_1+\textrm{ h.o.t.}\\
&=& y^2+q_1^2l_1+\textrm{ h.o.t.}
\end{eqnarray*}
Thus, the threefold $\mathcal{X}$ is again locally analytically isomorphic to the threefold $\mathcal{Y}=\{y^2-x^2z=0\}\subset\mathbb{A}^4_{x,y,z, s}$ which is the product of $\mathbb{A}^1$ with the Whitney umbrella. The normalization of $\mathcal{Y}$ is $\mathbb{A}^3_{u,v,w}$ with normalization map  $(u,v,w)\mapsto(u, uv, v^2, w)$.  In the coordinates of $\mathbb{A}^4_{x,y,z,s}$ the plane $L$ corresponds to the plane $P=(z=y=0)\subset \mathcal{Y}$. Because the normalization is an isomorphism over this locus, we have $P^{\nu}$ is the plane given by $v=0$. The quadric $Z$ corresponds to the locus $(x=y=0)\subset \mathcal{Y}$, which under the normalization becomes the plane $u=0$. Thus, we see that the proper transforms $L^{\nu}$ and $Z^{\nu}$ of $L$ and $Z$ intersect transversally after the normalization.

The plane $L$ intersects the quadric $Z$ in a conic $D$. Thus, for general $q$, $l$ and $B$, the curve $D=L\cap Z$ intersects the locus $B\cap Z$ tangentially at 6 points. Taking the double cover of $Z$ branched over $B$ gives a smooth surface $\tilde{W}$ with a smooth $(-4)$-curve $C$ given by the intersection of the plane $L$ with the surface $\mathcal{X}_0^{\nu}$.

We now show that an extremal contraction of $L^{\nu}$ results in a family of surfaces with ample canonical class. The canonical class $K_{X_0}$ is given by $K_{\mathcal{X}^{\nu}}|_{X_0}$. Since $K_{\mathcal{X_0}^{\nu}}|_{\tilde{W}}= K_{\tilde{W}}+C$ and 
$$K_{\mathcal{X}^{\nu}}|_{L}=K_{L}+C\sim -2H+H\sim -H,$$
we see that $L\subset \mathcal{X}^{\nu}$ can be contracted and that the surface $W$ obtained after contracting $C\subset \tilde{W}$ gives the stable limit. Note moreover that $C$ is a $(-4)$-curve on $\tilde{W}$, so this contraction produces a $\frac{1}{4}(1,1)$ singularity on $W$. By construction, the stable limit of the family is a stable numerical quintic surface $W$ with a $\frac{1}{4}(1,1)$ singularity of type 1 if $B$ is smooth, and of type 1' if $B$ has a node on $L\cap Z$. 

We claim that any stable numerical quintic surface of type 1 may be obtained as the stable limit of such a family. By Lemma~\ref{triples}, it suffices to show that given any triple $(Z, B, D)$ (where $Z$ is a fixed smooth quadric, $B\sim 6\Delta$ and $D\sim\Delta$ are smooth, and such that $B$ intersects $D$ with multiplicity $2$ at $6$ points) we can find a family of the desired form whose stable limit is a stable numerical quintic surface $W$ corresponding to $(Z, B, D)$ under the correspondence of Theorem~\ref{classify}.

Fix such a triple. Then $Z$ is isomorphic to a smooth quadric in $\mathbb{P}^3$ given by $q=0$. Let $l$ be the equation of the hyperplane $L$ in $\mathbb{P}^3$ such that $L\cap Z=D$. We claim that $B$ is also given by $V\cap Z$, where $V$ is a hypersurface of degree $6$ in $\mathbb{P}^3$. To see this, let $H$ be a general hyperplane section of $\mathbb{P}^3$ and consider the exact sequence
$$0\rightarrow\mathcal{O}_{\mathbb{P}^3}(-Z+6H)\rightarrow \mathcal{O}_{\mathbb{P}^3}(6H)\rightarrow\mathcal{O}_Z(6H)\rightarrow 0.$$
Since $H^1(\mathbb{P}^3, \mathcal{O}_{\mathbb{P}^3}(-Z+6H))=H^1(\mathbb{P}^3, \mathcal{O}_{\mathbb{P}^3}(4H))=0$, we see that global sections of $\mathcal{O}_{\mathbb{P}^3}(6H)$ surject onto global sections of $\mathcal{O}_Z(6H)$. Noting that $\mathcal{O}_Z(6H)\simeq\mathcal{O}_Z(6\Delta)$, this implies that the element $B\in |6\Delta|$ can be lifted to a hypersurface $V$ of degree $6$ in $\mathbb{P}^3$, proving the claim.

Next consider the exact sequence
$$0\rightarrow \mathcal{O}_Z(V-L)\rightarrow\mathcal{O}_Z(V)\rightarrow \mathcal{O}_{Z\cap L}(V)\rightarrow 0.$$
Since $B$ intersects $D$ at $6$ points with multiplicity $2$ each, this implies that the equation of $V|_{L}$ is of the form $f^2$, where the six points of $B\cap D$ are given by $f=q=0$. Therefore $V$ can be chosen to have equation $f^2-lg$, where $g$ is a general form of degree $5$. Then taking 
$$S_t=\{q^2l+tqf+t^2g=0\}\subset\Delta_t\times \mathbb{P}^3$$
gives the desired family.
\end{proof}

\begin{Remark} We remark that the family given in Theorem~\ref{quintic example} is one case of a more general degeneration of Castelnuovo surfaces (minimal surfaces of general type with $K^2=2p_g-7$ whose canonical maps are birational onto their images) described by Ashikaga and Konno~\cite[2.3]{ashikaga-konno1991}. Indeed, the family they give is the minimal resolution of the $\mathbb{Q}$-Gorenstein family we described in the proof of Theorem~\ref{quintic example}. 
\end{Remark}

\subsubsection{Types 2a and 2b}\label{Types 2a and 2b}

Friedman~\cite{friedman1983} constructed a family of stable numerical quintic surfaces with general fiber a numerical quintic surface of type IIb and special fiber a stable numerical quintic surface of type 2b. His construction easily generalizes to give a family of stable numerical quintic surfaces whose general fiber is a numerical quintic surface of type IIa and with special fiber a stable numerical quintic surface of type 2a. 

\begin{Theorem}\label{Friedman}\cite{friedman1983} There is a $\mathbb{Q}$-Gorenstein deformation $\mathcal{X}\rightarrow T$, where $T$ is the unit disk in $\mathbb{C}$, with general fiber $\mathcal{X}_t$, $t\neq 0$, a smooth numerical quintic surface of type IIa (respectively, IIb) and special fiber $\mathcal{X}_0$ a stable numerical quintic surface with a $\frac{1}{4}(1,1)$ singularity of type 2a (respectively, 2b).
Furthermore, this deformation induces a versal local $\mathbb{Q}$-Gorenstein deformation of a $\frac{1}{4}(1,1)$ singularity.
\end{Theorem}

We omit the proof of Theorem~\ref{Friedman}, but make two important observations. The first is that Friedman's construction is a degeneration of a family of IIb surfaces to a 2b surface. The construction of a family of IIa surfaces degenerating to a 2a surface is similar. For details, see \cite{mythesis}.

Secondly, we remark that Friedman's family induces a versal local $\mathbb{Q}$-Gorenstein deformation of the $\frac{1}{4}(1,1)$ singularity on the special fiber. To see this, note that if 
$(p\in W)$ is a germ of a $\frac{1}{4}(1,1)$ singularity, then $(p\in W)$ is analytically isomorphic to the singularity
$$(xy=z^2)\subset \frac{1}{2}(1,1,1).$$
Moreover, any deformation of $(p\in X)$ is analytically isomorphic to a deformation of the form
$$(xy=z^2+t^{\alpha})\subset \frac{1}{2}(1,1,1)\times\mathbb{A}^1_t,$$
for some integer $\alpha>0$ called the \emph{axial multiplicity} of the deformation. The resolution of the total space of such a deformation consists of two components intersecting with multiplicity $\alpha$. A versal local $\mathbb{Q}$-Gorenstein deformation of $(p\in X)$ has axial multiplicity $1$; that is, its resolution consists of two components meeting transversally. 
The special fiber of Friedman's family consists of two components meeting transversally and therefore induces a versal local $\mathbb{Q}$-Gorenstein deformation of the $\frac{1}{4}(1,1)$ singularity.

\begin{Remark} In~\cite[Corollary 1.2]{friedman1983}, Friedman uses Horikawa's description of the moduli space $\mathcal{M}_{5,5}$ to deduce the existence of a $\mathbb{Q}$-Gorenstein family $\tilde{\mathcal{X}}\rightarrow T$ of smooth quintic surfaces whose special fiber is an ``accordion'' of surfaces $V\cup W_1 \cup W_2 \cup \cdots \cup W_n$ where $V$ is the minimal resolution of a stable quintic surface of type 2b, $W_1, \ldots, W_{n-1}$ are copies of $\mathbb{F}_4$, and $W_n$ is a copy of $\mathbb{P}^2$, intersecting transversally as in Figure~\ref{non versal smoothing}.

\begin{figure}[h!]
  \centering
        \includegraphics{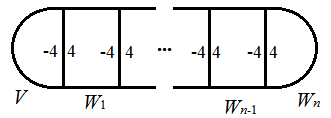}
  \caption{\label{non versal smoothing}}
\end{figure}

Here, the canonical class $K_{\tilde{\mathcal{X}}}$ is not ample, and the stable limit of $\tilde{\mathcal{X}}\rightarrow T$ is obtained by contracting the surfaces $W_1,..., W_n$. We now recognize the resulting special fiber as a stable numerical quintic surface of type 2b. Thus, Friedman's family is a $\mathbb{Q}$-Gorenstein smoothing of a 2b surface to a quintic surface. This family gives a local deformation of the $\frac{1}{4}(1,1)$ singularity with axial multiplicity $n$, so unless $n=1$, the induced deformation is not versal. Theorem~\ref{the best} in Section~\ref{main proof} implies that if $W$ is a 2b surface, then there exists a $\mathbb{Q}$-Gorenstein smoothing of $W$ to a quintic surface with $n=1$. 

Friedman also raises the question of describing deformations of 2b surfaces explicitly. Theorem~\ref{the best} answers this question.
\end{Remark}

\subsection{Some sheaf calculations}\label{sheaf calculations}

Let $X$ be a smooth surface and $D=\sum_{i=1}^k D_i$ a divisor in $X$ with simple normal crossings (in particular, each component divisor $D_i$ is smooth). Let $\Omega_X^1(\log D)$ denote the sheaf of logarithmic differentials. This sits in the short exact sequence of sheaves
$$0\rightarrow \Omega_X^1\rightarrow \Omega_X^1(\log D)\rightarrow \bigoplus_{i=1}^k\mathcal{O}_{D_i}  \rightarrow 0$$
where the map $\Omega_X^1(\log D)\rightarrow \bigoplus_{i=1}^k\mathcal{O}_{D_i}$ is the residue map. 

Now let $W$ be a surface whose only non Du Val singularity is a Wahl singularity and let $X$ be its minimal resolution. If $C$ is the exceptional divisor on $X$, then one can show that obstructions to $\mathbb{Q}$-Gorenstein deformations of $W$ lie in the cohomology group $H^2(X, T_X(\log C))$~\cite{lee-park2007}.  Thus, if $H^2(X, T_X(\log C))=0$, then the locus of such surfaces is generically smooth in $\overline{\mathcal{M}}_{K^2, \chi}$. The calculations of $H^2(X, T_X(\log C))$ in Theorems~\ref{vanishing for type 1}, \ref{vanishing for type 2a}, and~\ref{obstruction} require the following lemmas. 

\begin{Lemma}\label{blowup of two points} Let $\sigma: Y\rightarrow Z$ be the blowup of a smooth surface at a point $p$ lying in the smooth locus of a divisor $D\subset Z$ with normal crossings.  Let $\tilde{D}\subset Y$ be the proper transform of $D$. Then $\sigma_*\Omega_Y^1(\log \tilde{D})=\Omega_Z^1(\log D)\otimes \mathfrak{M}_{p}$, where $\mathfrak{M}_{p}$ is the ideal sheaf of $p$ on $Z$.
\end{Lemma}

\begin{proof}
It suffices to show the equality in a neighborhood of the exceptional divisor $E$. Let $V\subset Z$ be a coordinate neighborhood around $p$. Choose coordinates $(z,w)$ on $V$ so that $p$ is at the origin and the local equation of $D$ is $z$. Then $\sigma^{-1}(V)$ is covered by two neighborhoods $U_1$ and $U_2$.
Choose coordinates $(x,y)$  on $U_1$ so that $\sigma(x,y)=(x,xy)$ and the local equation of $E\cap U_1$ is $x$. Note that $\tilde D$ does not appear in $U_1$. Let coordinates on $U_2$  be $(u,v)$ so that $\sigma(u,v)=(uv, v)$. On $U_2$, the local equation of $E$ is $v$ and the local equation of $\tilde D$ is $u$. See Figure~\ref{blowup picture}.

\begin{figure}
\centering
\includegraphics[scale=.8]{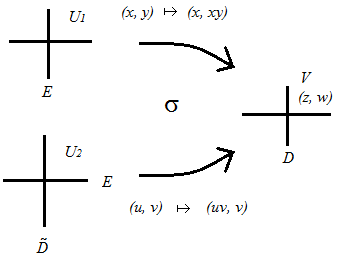}
\caption{The map $\sigma$.\label{blowup picture}}
\end{figure}

On $U_1$, we have
\begin{eqnarray}
\Omega_Y^1(\log \tilde D )(U_1)&=&\displaystyle{\left\{f \left(z,\frac{w}{z}\right)dz+g\left(z,\frac{w}{z}\right)d\left(\frac{w}{z}\right)\,  \middle \vert  \, f, g \in \mathcal{O}_{Z}(V)\right\}}\\
&=&\left\{\left[f\left(z,\frac{w}{z}\right)-\frac{w}{z^2}g\left(z, \frac{w}{z}\right)\right]dz+\frac{1}{z}g\left(z,\frac{w}{z}\right)dw\, \middle \vert \, f, g \in \mathcal{O}_{Z}(V)\right\}.\label{coh on Z}
\end{eqnarray}
On $U_2$ 
\begin{eqnarray*}
\Omega_Y^1(\log \tilde D)(U_2)&=&\left\{p\left(\frac{z}{w},w\right)\frac{d\left(\frac{z}{w}\right)}{\frac{z}{w}}+q\left(\frac{z}{w},w\right)dw\, \middle \vert \, p, q  \in \mathcal{O}_{Z}(V)\right\}\\
&=&\left\{\frac{1}{z}p\left(\frac{z}{w},w\right)dz+\left[q\left(\frac{z}{w},w\right)-\frac{1}{w}p\left(\frac{z}{w},w\right)\right]dw\, \middle \vert \, p, q  \in \mathcal{O}_{Z}(V)\right\}.
\end{eqnarray*}
These sections glue to a section of $\sigma_*(\Omega_Y^1(\log \tilde{D}))$ over $V$ if coefficients of $dz$ and $dw$ are equal:
\begin{equation}\label{coeff of dw}
\frac{1}{z}g\left(z,\frac{w}{z}\right)=q\left(\frac{z}{w},w\right)-\frac{1}{w}p\left(\frac{z}{w},w\right)
\end{equation}
\begin{equation}\label{coeff of dz}
\frac{1}{z}p\left(\frac{z}{w},w\right)=f\left(z,\frac{w}{z}\right)-\frac{w}{z^2}g\left(z, \frac{w}{z}\right)
\end{equation}
Replacing $\frac{1}{z}g(z,\frac{w}{z})$ in Equation~\eqref{coeff of dz} with its equivalent expression coming from Equation~\eqref{coeff of dw} yields the equality
$$f\left(z, \frac{w}{z}\right)=\frac{w}{z}q\left(\frac{z}{w},w\right).$$
From this last expression, we see that 
$$f\left(z, \frac{w}{z}\right)=\frac{1}{z}f'(z,w)$$ 
and 
$$q\left(\frac{z}{w},w\right)=\frac{1}{w}f'(z,w)$$
where $f'(z,w)$ is a polynomial with $f'(0,0)=0$. Plugging these into Equation~\eqref{coeff of dw} and multiplying through by $zw$ gives
$$wg\left(z,\frac{w}{z}\right)=z\left(f'\left(z,w\right)-p\left(\frac{z}{w},w\right)\right).$$
Since the right hand side is a polynomial in $z$, we can write $g\left(z, \frac{w}{z}\right)=zg'(z,w)$ for some polynomial $g'$ with $g'(0,0)=0$, and rewrite the above equality as
$$wg'(z,w)=f'(z,w)-p\left(\frac{z}{w},w\right).$$
Therefore, $p\left(\frac{z}{w},w\right)=wg'(z,w)-f'(z,w)$. We now have expressions for $f$, $g$, $p$, and $q$ as polynomials in $z$ and $w$, which we can use in Equation~\eqref{coh on Z}. This gives us
\begin{eqnarray*}
\sigma_*(\Omega_Y^1(\log \tilde{D}))(V)&=&\left\{\left[\frac{1}{z}f'(z,w)-\frac{w}{z}g'(z, w)\right]dz+ \, g'(z,w)dw\, \middle \vert \,  f', g'  \in \mathcal{O}_{Z}(V)\right\}\\
&=& \left\{f'(z,w)\frac{dz}{z} + g'(z,w)dw\, \middle \vert \, f', g'   \in \mathcal{O}_{Z}(V)\right\}
\end{eqnarray*}
where the only restrictions on $f'(z,w)$ and $g'(z,w)$ are that neither has a constant term; that is, they both lie in the maximal ideal $\mathfrak{M}_p=(z,w)\subset \mathcal{O}_Z(V)\simeq\mathbb{C}[z,w]$.  Thus,
$$\sigma_*(\Omega_Y^1(\log \tilde D))=\Omega_Z^1(\log D)\otimes \mathfrak{M}_{p}.$$
\end{proof}

\begin{Lemma}\label{decomposition into eigenspaces} Let $f:X\rightarrow Y$ be a double cover of a smooth surface $Y$, and let $B$ denote its smooth branch divisor. Let $C=f^{-1}(D)$ be the preimage of a smooth curve $D$ on $Y$, and suppose that $D$ intersects $B$ transversally. Then  
$$f_*(\Omega_X^1(\log C))=\Omega_Y^1(\log D)\oplus\Omega_Y^1((\log D+B)(-L))$$
and 
$$f_*(T_X(\log C))=T_Y(\log(D+B))\oplus T_Y(\log D)(-L)$$
where $B\sim 2L$.
Moreover, these decompositions break the sheaves into their invariant and anti-invariant subspace under the action of $\mathbb{Z}/2\mathbb{Z}$ by deck transformations.
\end{Lemma}

\begin{Remark} Lemma~\ref{decomposition into eigenspaces} is an extension of the double cover version of~\cite[Lemma 4.2]{pardini1991} to the log tangent sheaf.
\end{Remark}

\begin{proof}
In order to compute $f_*\Omega_X^1(\log C)$, note that it admits an action of $\mathbb{Z}/2\mathbb{Z}$ via deck transformations, so we can decompose it into its invariant and anti-invariant eigenspaces. 

Let $V$ be an open neighborhood of $p\in D \cap B$ and choose coordinates $(z,w)$ on $V$ so that $p$ is at the origin and the local equation of $D$ is $z$ and the local equation of $B$ is $w$. Then we have an open neighborhood $U$ of $f^{-1}(p)$ with local coordinates $(x,y)$ so that $f(x,y)=(x, y^2)$. Note that the ramification locus $R$ of $f$ has local equation $y$ and the curve $C$ on $X$ has local equation $x$. See Figure~\ref{double cover picture}.

\begin{figure}[!h]
\centering
\includegraphics[scale=.8]{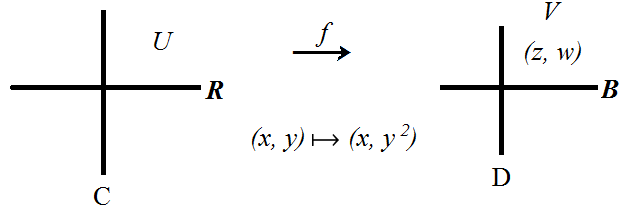}
\caption{The map $f$.\label{double cover picture}}
\end{figure}

On $U$ we have 
$$\Omega_X^1(\log C)(U)=\left<\frac{dx}{x},\,dy\right>_{\mathcal{O}_X(U)}.$$
Noting that $\mathcal{O}_Y(V)\simeq \mathbb{C}[x,y^2]$, we have
$$
f_*(\Omega_X^1(\log C))(V)=
\left<\frac{dx}{x},\,y\frac{dx}{x},\,dy, \,y dy \right>_{\mathcal{O}_Y(V)}
$$
The action of $\mathbb{Z}/2\mathbb{Z}$ sends $(x,y)$ to $(x,-y)$. Therefore the invariant subspace of $f_*(\Omega_X^1(\log C))(V)$ is
$$
f_*(\Omega_X^1(\log C))_+(V)=\left<\frac{dx}{x}, ydy\right>_{\mathcal{O}_Y(V)}=\left<\frac{dz}{z},  dw\right>_{ \mathcal{O}_Y(V)} = \Omega_Y^1(\log D)(V).
$$
The anti-invariant subspace of $f_*(\Omega_X^1(\log C))(V)$ is
$$
f_*(\Omega_X^1(\log C))_-(V)=\left< y\frac{dx}{x}, dy\right>_{\mathcal{O}_Y(V)}=
=y\left<\frac{dz}{z}, \frac{dw}{w}\right>_{\mathcal{O}_Y(V)}\\
=\Omega_Y^1((\log D+B)(-L))(V).
$$
One checks easily that these modules extend to the expected sheaves over all of $Y$. 
The proof for the log tangent bundle is similar.
\end{proof}

\subsection{Smooth boundary components of $\overline{\mathcal{M}}_{5,5}$}\label{1 and 2a}

We show that loci corresponding to surfaces of type 1 and 2a give generically smooth loci in the moduli space $\overline{\mathcal{M}}_{5,5}$. In both cases, we obtain this result by proving the vanishing of the cohomology group in which obstructions to $\mathbb{Q}$-Gorenstein deformations lie. By Theorem~\ref{counting}, the type 1 and 2a loci are $39$-dimensional, so we conclude that the closure of the 1 and 2a loci are generically smooth Cartier divisors in $\overline{\mathcal{M}}_{5,5}$.

\subsubsection{The type 1 component}

For this subsection, let $W$ be a stable numerical quintic surface of type 1 or 1'' and denote by $X$ its minimal resolution. Let $f:X\rightarrow Z$  be the double cover, where $Z= \mathbb{P}^1\times\mathbb{P}^1$ or $\mathbb{F}_2$, and $f$ is branched over a smooth curve $B\sim 6\Delta$, tangent to $D\sim \Delta$ at six points. Then $f^*(D)=C_1+C_2$ and the curves $C_1$ and $C_2$ are $(-4)$ curves on $X$. Let $R=f^*B$ denote the ramification locus of $f$, and let $L\subset Z$ be a curve such that $B\sim 2L$. 

In order to show that deformations of $W$ are unobstructed, it suffices to show that $H^2(W, T_W)=0$. Equivalently, as described above, we show that $H^2(X,T_X(\log (C_1)))=0$.  

\begin{Theorem}\label{vanishing for type 1} Let $X$ be the minimal resolution of a stable numerical quintic surface of type 1 or 1'', and let $C_1$ and $C_2$ be the $(-4)$-curves on $X$. Then $H^2(X,T_X(\log (C_1)))=0$.
\end{Theorem}
\begin{proof} 
By Serre duality, it is enough to show that $H^2(X,\Omega_X^1(\log (C_1))(K))=0,$ where $K=K_X$. 

The double cover $f:S\rightarrow Z$ gives rise to an action of $\mathbb{Z}/2\mathbb{Z}$ on $H^0(X, \Omega_X^1(\log(C_1+C_2))(K))$ via deck transformations. To begin with, we prove that the anti-invariant subspace
$H^0(X, \Omega_X^1\log(C_1+C_2)(K))_-$
vanishes. 

By the projection formula, noting that $K\sim f^*(\Delta)$, we have
$$f_*(\Omega_X^1\log(C_1+C_2)(K))=(f_*\Omega_X^1\log(C_1+C_2))(\Delta).$$
We claim that $$f_*\Omega_X^1(\log(C_1+C_2))_-\subset \Omega_{Z}(\log B)(-2\Delta).$$

To compute $f_*(\Omega_X^1\log(C_1+C_2))_-$, we need only consider a point in $C_1\cap C_2\cap R$.  Indeed, suppose that $U$ is a neighborhood of $p\in X$ such that $U\cap C_1\cap C_2\cap R=\emptyset$, and let $V$ denote the image of $U$ under $f$. By Lemma~\ref{decomposition into eigenspaces}, we have 
\begin{eqnarray*}
f_*(\Omega_X^1\log(C_1+C_2))_-(V)&=&\Omega_{Z}^1(\log (B+D))(-3\Delta)(V)\\
&\subset& \Omega_{Z}^1(\log B)(-2\Delta)(V),
\end{eqnarray*}
because $D\cap V=\emptyset$.

Now let $U$ be an open subset of $X$ containing $p\in C_1\cap C_2\cap R$, and let $V$ an open neighborhood of $f(p)$. Choose coordinates $(x,y)$ on $U$  so that $p$ is at the origin and the local equation of $R$ is $y$. We can then choose coordinates $(w,z)$ on $V$ such that the local equation of $B$ is $z$ and the local equation of $D$ is $z-w^2$. Then the local equations of $C_1$ and $C_2$ are $y-x$ and $y+x$. With these coordinates, the cover $f$ is given by the function $(x,y)\mapsto (x, y^2)$. See Figure~\ref{map to Z}.

\begin{figure}[!h]
\centering
        \includegraphics[scale=.7]{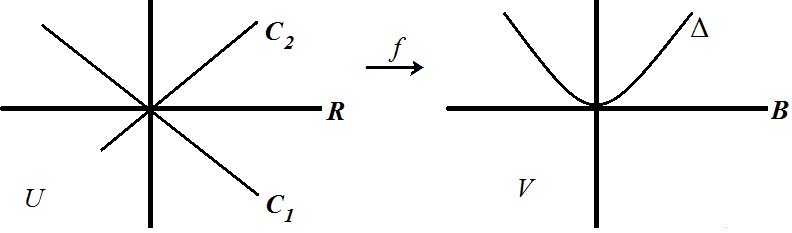}
  \caption{The map $f$. \label{map to Z}}
\end{figure}

The $\mathcal{O}_X(U)$-module  $\Omega_X^1\log(C_1+C_2)(U)$ is  generated by $\left\{\frac{d(y-x)}{y-x},\,\frac{d(y+x)}{y+x}\right\}$. As a module over $\mathcal{O}_Y(V)$, we have that $f_* \Omega_X^1\log(C_1+C_2)(V)$ is generated by
$$\left\{\frac{d(y-x)}{y-x}, d(y-x), \,\frac{d(y+x)}{y+x}, \, d(y+x)\right\}$$
Since the action of $\mathbb{Z}/2\mathbb{Z}$ sends $y$ to $-y$, we see quickly that the anti-invariant submodule is generated as an $\mathcal{O}_Y(V)$--module by
\begin{eqnarray*}
\left\{\frac{d(y-x)}{y-x}+\frac{d(y+x)}{y+x}, \, dy\right\}
&\subset&\left\{\frac{1}{y^2-x^2}(-2ydx+2xdy), \, \frac{1}{y^2-x^2} dy\right\}\\
&=&\frac{y}{z-w^2}\left\{-2dw, \, \frac{dz}{z}\right\}
\end{eqnarray*}
This last module we recognize as $\Omega_{Z}^1(\log B)(-3\Delta+D)(V)=\Omega_{Z}^1(\log B)(-2\Delta)(V)$. 
Thus, 
$$f_*\Omega^1_X(\log(C_1+C_2))_-\subset\Omega_{Z}^1(\log (B))(-2\Delta).$$

By the projection formula, using that $K\sim f^*\Delta$, we have
$$f_*\Omega^1_X(\log(C_1+C_2))(K)_-\subset \Omega_{Z}^1(\log B)(-\Delta).$$
To show that $H^0(Z, \Omega_{Z}^1(\log B)(-\Delta))=0$, consider the exact sequence
$$0\rightarrow \Omega_{Z}^1\rightarrow \Omega_{Z}^1(\log B)\rightarrow \mathcal{O}_B\rightarrow 0$$
where $\Omega_{Z}^1(\log B)\rightarrow \mathcal{O}_B$ is the residue map. Twisting by $-\Delta$ gives the exact sequence
$$0\rightarrow \Omega_{Z}^1(-\Delta)\rightarrow \Omega_{Z}^1(\log B)(-\Delta)\rightarrow \mathcal{O}_B(-\Delta)\rightarrow 0.$$
Looking at the corresponding long exact sequence in cohomology, it remains to show that $H^0(Z, \Omega_{Z}^1(-\Delta))=0$ and $H^0(B, \mathcal{O}_B(-\Delta))=0$. Both of these are obvious, the first because $H^0(Z, \Omega_{Z}^1(-\Delta))\subset H^0(Z, \Omega_{Z}^1)=0$ and the second because $-\Delta\cdot B=-12<0$. Thus, 
$$H^0(X, \Omega_X^1\log(C_1+C_2)(K))_-=0,$$
as we wished to show.

Now consider a one-form $\alpha\in\Omega_X^1(\log C_1)(K)$. Since
$$\Omega_X^1(\log C_1)(K)\subset\Omega_X^1(\log C_1+C_2)(K)$$
and the latter sheaf has no anti-invariant part, the one-form $\alpha$ must be invariant. But the action of $\mathbb{Z}/2\mathbb{Z}$ on cohomology interchanges $C_1$ and $C_2$, so $\alpha$ must not have a pole along $C_1$.
Thus,  
$$\alpha\in H^0(X, \Omega_X^1(K))\simeq H^2(X, T_X)^{\vee}.$$
By Horikawa~\cite{horikawa1976}, we have $H^2(X, T_X)=0$, and so $\alpha=0$, completing the proof.
\end{proof}

In Section~\ref{dimension counts}, we showed that the locus of stable quintic surfaces of type 1 is $39$-dimensional, so Theorems~\ref{quintic example} and~\ref{vanishing for type 1} imply the following:

\begin{Cor}\label{1 is divisor} The closure of the locus of surfaces of type 1 is a generically smooth Cartier divisor in $\overline{\mathcal{M}}_{5,5}$, lying in the closure of the type I component of $\mathcal{M}_{5,5}$. 
\end{Cor}

\subsubsection{The 2a component}\label{2a component}

Let $W$ be a stable numerical quintic surface of type 2a, 2a', or 2a'' and let $S$ denote its minimal resolution. Then there is a map $\tilde{f}: X \rightarrow \tilde{Z}$, which is the double cover of the blowup of $Z=\mathbb{P}^1\times\mathbb{P}^1$ in two points $p$ and $q$ lying on a fiber $D$. The branch locus $\tilde{B}$ of $\tilde{f}$ is the proper transform of an irreducible curve $B \sim 6\Delta$ which has either a node or an $A_2$ singularity at each of $p$ and $q$ and is smooth elsewhere. Denote by $\Gamma_1$ and $\Gamma_2$ generic rulings of $\tilde{Z}$ so that $\Gamma_2\sim \tilde{D}+E_1+E_2$, where $\tilde{D}$ is the proper transform of $D\subset Z$.

\begin{Theorem}\label{vanishing for type 2a} Let $W$ be a stable numerical quintic surface of type 2a, 2a', or 2a'', let $X$ be its minimal resolution and $C$ the $(-4)$-curve on $X$. Then $H^2(X, T_X(\log C)))=0.$
\end{Theorem}

We begin with a lemma.

\begin{Lemma}\label{another lemma for type 2a}  $H^0(\tilde{Z}, \Omega_{\tilde{Z}}^1(\log\tilde{D}+\tilde{B})(K_{\tilde{Z}}))=0.$
\end{Lemma}

\begin{proof}
We have the following exact sequence of sheaves on $\tilde{Z}$:
$$0\rightarrow \Omega_{\tilde{Z}}^1\rightarrow \Omega_{\tilde{Z}}^1(\log(\tilde{D}+\tilde{ B}))\rightarrow\mathcal{O}_{\tilde{D}}\oplus\mathcal{O}_{\tilde{B}}\rightarrow 0$$
where $\Omega_{\tilde{Z}}^1(\log{\tilde{D}+\tilde B})\rightarrow\mathcal{O}_{\tilde{D}+\tilde{B}}$ is the residue map. Twisting by $K_{\tilde{Z}}$ gives the exact sequence
\begin{equation}\label{SES with K}
0\rightarrow \Omega_{\tilde{Z}}^1(K_{\tilde{Z}})\rightarrow \Omega_{\tilde{Z}}^1(\log{\tilde{D}+\tilde B})(K_{\tilde{Z}})\rightarrow(\mathcal{O}_{\tilde{D}}\oplus\mathcal{O}_{\tilde{B}})(K_{\tilde{Z}})\rightarrow 0.
\end{equation}
Note that 
\begin{equation}
K_{\tilde{Z}}=\sigma^*(K_{\mathbb{P}^1\times\mathbb{P}^1})+E_1+E_1=-2\Gamma_1-2\Gamma_2+E_1+E_2\sim -2\Gamma_1-2\tilde{D}-E_1-E_2,
\end{equation}
and so $-K_{\tilde{Z}}$ is effective. Thus $H^0(\tilde{Z}, \Omega_{\tilde{Z}}^1(K_{\tilde{Z}}))\subset H^0(\tilde{Z}, \Omega_{\tilde{Z}}^1)$. Since the irregularity of ${\tilde{Z}}$ is zero, we have
$H^0({\tilde{Z}}, \Omega_{\tilde{Z}}^1)(K_{\tilde{Z}})=0$. Moreover, noting that $\sigma^*(B)=\tilde{B}+2E_1+2E_2$ and $\sigma^*(K_Z)=K_{\tilde{Z}}-E_1-E_2$, we have
$$K_{\tilde{Z}}\cdot\tilde{B}=-24<0$$
and
$$K_{\tilde{Z}}\cdot \tilde{D}=0.$$

Therefore $H^0(\tilde{Z}, (\mathcal{O}_{\tilde{D}}\oplus\mathcal{O}_{\tilde{B}})(K_{\tilde{Z}}))=\mathbb{C}$, so the cohomology group
$$H^0(\tilde{Z}, \Omega_{\tilde{Z}}^1(\log(\tilde{D}+\tilde{B}))(K_{\tilde{Z}}))$$
is $0$ if and only if the connecting homomorphism 
$$\delta:H^0(\tilde{Z}, (\mathcal{O}_{\tilde{D}}\oplus\mathcal{O}_{\tilde{B}})(K_{\tilde{Z}}))\rightarrow H^1(\tilde{Z}, \Omega_{\tilde{Z}}^1(K_{\tilde{Z}}))$$
is injective.

Since $-K_{\tilde{Z}}$ is effective, we have a section $s\in H^0(\tilde{Z}, \Omega_{\tilde{Z}}^1(-K_{\tilde{Z}}))$, so we have a map from the short exact sequence~\eqref{SES with K} to the short exact sequence
$$0\rightarrow \Omega_{\tilde{Z}}^1\rightarrow \Omega_{\tilde{Z}}^1(\log(\tilde{D}+\tilde{B}))\rightarrow\mathcal{O}_{\tilde{D}}\oplus\mathcal{O}_{\tilde{B}}\rightarrow 0.$$
where the map is given by tensoring with $s$.
The connecting homomorphism 
$$\delta_2: H^0(\tilde{Z}, \mathcal{O}_{\tilde{D}}\oplus\mathcal{O}_{\tilde{B}})\rightarrow H^1(\tilde{Z}, \Omega_{\tilde{Z}}^1)$$ 
of the corresponding short exact sequence
is the first Chern class map. That is,  if $1_{\tilde{D}}$ and $1_{\tilde{B}}$ are generators of $H^0(\tilde{Z}, \mathcal{O}_{\tilde{D}}\oplus\mathcal{O}_{\tilde{B}})$, then $\delta_2(1_{\tilde{D}})=c_1(\tilde{D})$ and $\delta_2(1_{\tilde{B}})=c_1(B)$. Thus, the map $\delta_2$ is injective if and only if the curves $\tilde{D}$ and $\tilde{B}$ are linearly independent in the Picard group of $\tilde{Z}$. Recalling that Pic$(\tilde{Z})$ is generated by $\Gamma_1$, $\Gamma_2$, $E_1$ and $E_2$, and that $\tilde{B}\sim 6\Gamma_1+6\Gamma_2-2E_1-2E_2$ and $\tilde{D}\sim \Gamma_2-E_1-E_2$, we see that the two divisors are indeed linearly independent.

Thus, we have a diagram
\[\xymatrix{
H^0(\tilde{Z}, (\mathcal{O}_{\tilde{D}}\oplus\mathcal{O}_{\tilde{B}})(K_{\tilde{Z}}))\ar[r]^{\delta} \ar[d]_{\otimes s}  & H^1(\tilde{Z}, \Omega_{\tilde{Z}}^1(K_{\tilde{Z}})) \ar[d]^{\otimes s}\\
H^0(\tilde{Z}, \mathcal{O}_{\tilde{D}}\oplus\mathcal{O}_{
\tilde{B}})\ar[r]^{\delta_2}& H^1(\tilde{Z}, \Omega_{\tilde{Z}}^1)
}\]
where the bottom arrow is injective. We see that $\delta$ is injective as long as the map on the left is injective. But this map simply takes a section of $(\mathcal{O}_{\tilde{D}}\oplus\mathcal{O}_{\tilde{B}})(K_{\tilde{Z}})$ and multiplies it by $s$. Since $s\neq0$, the map is injective.
\end{proof}

\begin{proof}[Proof of Theorem~\ref{vanishing for type 2a}]

We show that $H^2(X, T_X(\log C))=0$, where $X$ is the minimal resolution of $W$ and $C$ is the $(-4)$-curve on $X$. By Serre duality, it is enough to show that $H^0(X, \Omega_X^1(\log C)(K_X))=0$. Recall that $C=f^*\tilde{D}$ and $K_X=f^*(K_Y+\tilde{L})$. By the projection formula 
$$f_*(\Omega_X^1(\log C)(K_X))=(f_*\Omega_X^1(\log C))\otimes (K_Y+\tilde{L}).$$
Together with Lemma~\ref{decomposition into eigenspaces}, this gives
$$
f_*(\Omega_X^1(\log C)(K_S))=\Omega_Y^1(\log\tilde{D})(K_Y+\tilde{L})\oplus\Omega_Y^1(\log\tilde{D}+\tilde{B})(K_Y).
$$

By Lemma~\ref{another lemma for type 2a}, we have $H^0(Y, \Omega_Y^1(\log\tilde{D}+\tilde{B})(K_Y))=0$. 
It remains to show that $H^0(Y, \Omega_Y^1(\log\tilde{D})(K_Y+\tilde{L}))=0$, which we do via the projection formula.
By Lemma~\ref{blowup of two points}, we have $\sigma_*\Omega_Y^1(\log\tilde{D})=\Omega_Z^1(\log D)\otimes \mathfrak{M}_{p,q}$, where $\mathfrak{M}_{p,q}$ is the ideal sheaf of $p$ and $q$ which are the centers of $\sigma$. Noting that $(K_Y+\tilde{L})=f^*(\Delta)$, the projection formula gives
\begin{eqnarray*}
\sigma_*(\Omega_Y^1(\log\tilde{D})(K_Y+\tilde{L})) &=&(\Omega_{\mathbb{P}^1\times\mathbb{P}^1}^1(\log D)\otimes \mathfrak{M}_{p,q})\otimes \mathcal{O}(\Delta)\\
&=&[(p_1^*\Omega_{\mathbb{P}^1}^1(\log D)\otimes \mathfrak{M}_{p,q})\oplus(p_2^*\Omega_{\mathbb{P}^1}^1\otimes\mathfrak{M}_{p,q})]\otimes \mathcal{O}(\Delta)\\
&=&(\mathcal{O}(0,1)\otimes\mathfrak{M}_{p,q})\oplus(\mathcal{O}(1,-1)\otimes \mathfrak{M}_{p,q}).
\end{eqnarray*}

We have $H^0(\mathbb{P}^1\times\mathbb{P}^1, \mathcal{O}(1,-1)\otimes \mathfrak{M}_{p,q})=0$, because $H^0(\mathbb{P}^1\times\mathbb{P}^1, \mathcal{O}(a,b))=0$ for $a<0$ or $b<0$. And $H^0(\mathbb{P}^1\times\mathbb{P}^1, \mathcal{O}(0,1)\otimes\mathfrak{M}_{p,q})=0$, since $p$ and $q$  lie on $D\in |1,0|$. 
\end{proof}

By Theorem~\ref{counting}, the locus of 2a surfaces is $39$-dimensional. Moreover, Theorem~\ref{Friedman} shows that every 2a surfaces may be obtained as the stable limit of a family of numerical quintic surfaces of type IIa. Together with Theorem~\ref{vanishing for type 2a}, this implies the following

\begin{Cor}\label{2a is divisor} The closure of the locus of surfaces of type 2a is a generically smooth Cartier divisor in $\overline{\mathcal{M}}_{5,5}$, lying in the closure of the type IIa component of $\mathcal{M}_{5,5}.$ 
\end{Cor}


\section{Deformations of 2b surfaces}\label{2b}

We study the versal $\mathbb{Q}$-Gorenstein deformation space $\textrm{Def}^{QG}(W)$~\cite{hacking2004} where $W$ is a general 2b surface. All deformation functors considered are functors of Artinian rings. However, because $W$ is a stable surface, we often abuse notation and view $\textrm{Def}^{QG}(W)$ as an analytic germ of a point $[W]$ in the KSBA moduli space. The same notational ambiguity applies to other deformation functors we consider which admit a moduli space. This enables us to study the moduli space $\overline{\mathcal{M}}$ using analytic methods of Horikawa~\cite{horikawa1975, horikawa1976}. The main theorem is

\begin{Theorem}\label{the best} The locus of stable numerical quintic surfaces whose unique non Du Val singularity is a $\frac{1}{4}(1,1)$ singularity forms a divisor in $\overline{\mathcal{M}}_{5,5}$ which consists of two 39-dimensional components $\bar{1}$ and $\overline{\mbox{2a}}$ meeting, transversally at a general point, in a 38-dimensional component $\overline{\mbox{2b}}$. This divisor is Cartier at general points of the $\bar{1}$, $\overline{\mbox{2a}}$, and  $\overline{\mbox{2b}}$ components. These components are the closures of the loci of 1, 2a, and 2b surfaces described at the beginning of Section~\ref{dimension counts}. Moreover, the type  $\bar{1}$, $\overline{\mbox{2a}}$, and  $\overline{\mbox{2b}}$ components belong to the closure of the components in $\mathcal{M}_{5,5}$ of types I, IIa, and IIb, respectively.
\end{Theorem}

The proof will consist of several pieces. Theorems~\ref{vanishing for type 1} and~\ref{vanishing for type 2a} showed that obstructions to deformations of surfaces of types 1 and 2a vanish, and so the closures of their corresponding $39$-dimensional loci in $\overline{\mathcal{M}}_{5,5}$ are generically smooth Cartier divisors. In Theorem~\ref{obstruction}, we show that deformations of 2b surfaces are obstructed and that the obstruction space is one-dimensional. This implies that the space $\textrm{Def}^{QG}(W)$ of $\mathbb{Q}$-Gorenstein deformations of a generic 2b surface $W$ is a hypersurface singularity. We show that the space of equisingular $\mathbb{Q}$-Gorenstein deformations of $W$ consists of irreducible components $\bar{1}$ and $\overline{\mbox{2a}}$ meeting, transversally at a general point, in the $\overline{\mbox{2b}}$ component. Together with Horikawa's description of $\mathcal{M}_{5,5}$, and the smoothings described in Theorems~\ref{quintic example} and~\ref{Friedman}, this implies that the space $\textrm{Def}^{QG}(W)$ has two irreducible components. 

By Theorem~\ref{Friedman}, there exists a $\mathbb{Q}$-Gorenstein smoothing of a 2b surface to a numerical quintic surface of type IIb which induces a versal deformation of the singularity.  Therefore, the map $\textrm{Def}^{QG}(W)\rightarrow\textrm{Def}^{QG}_{\mbox{loc}}(p)$ to local $\mathbb{Q}$-Gorenstein deformations of the $\frac{1}{4}(1,1)$ singularity $(p\in W)$ is surjective. This latter space is one-dimensional,  and since $\textrm{Def}^{QG}(W)$ has two irreducible components, the space $\textrm{Def}^{QG}(W)$ is analytically isomorphic to $\textrm{Def}^{QG}_{\mbox{e.s.}}(W)\times\mathbb{A}^1$.

The key to Theorem~\ref{the best} is the proof of the fact that the space $\textrm{Def}^{QG}_{\mbox{e.s.}}(W)$ of equisingular $\mathbb{Q}$-Gorenstein deformations of a general 2b surface $W$ consists of two irreducible components meeting transversally at a general point. This space is isomorphic to the deformation space of pairs $\textrm{Def}(X,C)$, where $X$ is the minimal resolution of  $W$, containing $(-4)$-curve $C$.
In~\ref{equisingular}, we describe a subfunctor of the deformation functor of pairs $\mathpzc{Def}(X,C)$, and show that this subfunctor has no obstructions. This will imply that the space $\textrm{Def}(X,C)$ contains a smooth component corresponding to the 2a locus. Thus, to prove Theorem~\ref{the best},
it suffices to show that the degree two part of the Kuranishi map, given by the Schouten bracket, is nonzero and not a square. Horikawa makes a similar argument in~\cite{horikawa1975} and~\cite{horikawa1976}. In~\ref{technicalities} and~\ref{main proof}, we extend his work to the log setting.

We use the following notation throughout this section.  Let $X$ be the minimal resolution of a surface of type 2b. We recall the construction of $X$. Let $\sigma: \tilde{\mathbb{F}}_2 \rightarrow \mathbb{F}_2$ be the blowup of $\mathbb{F}_2$ in two distinct points $p$ and $q$ lying on a fiber $D$. Denote by $\tilde{D}$ and $\Gamma$ the proper transforms of $D$ and a generic fiber, respectively, and let $E_1$ and $E_2$ be the exceptional divisors of $\sigma$. By abuse of notation, we denote by $\Delta_0$ the proper transform of the negative section $\Delta_0$ on $\mathbb{F}_2$. Let $B$ be a reduced, irreducible divisor in the linear system $|6\Delta_0+12\Gamma|$ on $\mathbb{F}_2$ with simple nodes at $p$ and $q$ and no other singularities. Let $\tilde{B}$ be its proper transform and note that $\tilde{B}\sim 2\tilde{L}$ for some smooth divisor $L$ on $\tilde{\mathbb{F}}_2$. Then $X$ is given by the double cover $f: X\rightarrow \tilde{\mathbb{F}}_2$ branched over $\tilde{B}$. The curve $C$ given by $f^*(\tilde{D})$ is the $(-4)$-curve on $X$. Moreover $X$ contains four $(-2)$-curves: $F_1$ and $F_2$ mapping to $\Delta_0$, and $\bar{E}_1$ and $\bar{E}_2$ mapping to $E_1$ and $E_2$, respectively. We denote by $\pi:\mathbb{F}_2\rightarrow \mathbb{P}^1$ and $g:X\rightarrow \mathbb{P}^1$ the projection maps to $\mathbb{P}^1$.

\subsection{The obstruction}\label{the obstruction}

To begin with, we show that the obstruction space is one-dimensional.

\begin{Theorem}\label{obstruction} Let $X$ be the minimal resolution of a 2b surface, and let $C$ denote the $(-4)$-curve on $X$. Then $H^2(X, T_X(\log C))=\mathbb{C}$.
\end{Theorem}

The proof of Theorem~\ref{obstruction} requires two lemmas.

\begin{Lemma}\label{coh vanishing on the blowup of F2} Let $Z=\mathbb{F}_2$ and $\tilde{Z}$ the blowup of $Z$ in $p$ and $q$. Then $H^0(\tilde{Z}, \Omega^1_{\tilde{Z}}(\log(\tilde{D}+\tilde{B}+\Delta_0))(K_{\tilde{Z}}))=0$.
\end{Lemma}

\begin{proof}
The proof is very similar to that of Lemma~\ref{another lemma for type 2a}.

We have the following exact sequence of sheaves on $\tilde{Z}$:
$$0\rightarrow\Omega_{\tilde{Z}}^1(K_{\tilde{Z}})\rightarrow\Omega_{\tilde{Z}}^1(\log(\tilde{D}+\tilde{B}+\Delta_0))(K_{\tilde{Z}})\rightarrow (\mathcal{O}_{\tilde{D}}\oplus\mathcal{O}_{\tilde{B}}\oplus\mathcal{O}_{\Delta_0})(K_{\tilde{Z}})\rightarrow 0.$$
where $\Omega_{\tilde{Z}}^1(\log{\tilde{D}+\tilde B})\rightarrow\mathcal{O}_{\tilde{D}+\tilde{B}}$ is the residue map. Twisting by $K_{\tilde{Z}}$ gives the exact sequence
\begin{equation}\label{SES with K for F2}
0\rightarrow \Omega_{\tilde{Z}}^1(K_{\tilde{Z}})\rightarrow \Omega_{\tilde{Z}}^1(\log{\tilde{D}+\tilde B+\Delta_0})(K_{\tilde{Z}})\rightarrow(\mathcal{O}_{\tilde{D}}\oplus\mathcal{O}_{\tilde{B}}\oplus\mathcal{O}_{\Delta_0})(K_{\tilde{Z}})\rightarrow 0.
\end{equation}
Note that 
\begin{equation*}
K_{\tilde{Z}}=\sigma^*(K_{\mathbb{F}_2})+E_1+E_1\sim -2\Delta_0-4\tilde{D}-3E_1-3E_2,
\end{equation*}
and so $-K_{\tilde{Z}}$ is effective. Thus $H^0(\tilde{Z}, \Omega_{\tilde{Z}}^1(K_{\tilde{Z}}))\subset H^0(\tilde{Z}, \Omega_{\tilde{Z}}^1)$. Since the irregularity of ${\tilde{Z}}$ is zero, we have
$H^0({\tilde{Z}}, \Omega_{\tilde{Z}}^1)(K_{\tilde{Z}})=0$. Moreover, because $\sigma^*(B)=\tilde{B}+2E_1+2E_2$ and $\sigma^*(K_Z)=K_{\tilde{Z}}-E_1-E_2$, we have $K_{\tilde{Z}}\cdot\tilde{B}=-24<0$, $K_{\tilde{Z}}\cdot \tilde{D}=0$,
and 
$K_{\tilde{Z}}\cdot \Delta_0=0.$
Therefore $$H^0(\tilde{Z}, (\mathcal{O}_{\tilde{D}}\oplus\mathcal{O}_{\tilde{B}}\oplus\mathcal{O}_{\Delta_0})(K_{\tilde{Z}}))=\mathbb{C}^2,$$ so the cohomology group 
$$H^0(\tilde{Z}, \Omega_{\tilde{Z}}^1(\log(\tilde{D}+\tilde{B}+\Delta_0))(K_{\tilde{Z}}))$$
is $0$ if and only if the connecting homomorphism 
$$\delta:H^0(\tilde{Z}, (\mathcal{O}_{\tilde{D}}\oplus\mathcal{O}_{\tilde{B}}\oplus\mathcal{O}_{\Delta_0})(K_{\tilde{Z}}))\rightarrow H^1(\tilde{Z}, \Omega_{\tilde{Z}}^1(K_{\tilde{Z}}))$$
is injective.

Since $-K_{\tilde{Z}}$ is effective, we have a section $s\in H^0(\tilde{Z}, \Omega_{\tilde{Z}}^1(-K_{\tilde{Z}}))$, so we have a map from the short exact sequence~\eqref{SES with K for F2} to the short exact sequence
$$0\rightarrow \Omega_{\tilde{Z}}^1\rightarrow \Omega_{\tilde{Z}}^1(\log(\tilde{D}+\tilde{B}+\Delta_0))\rightarrow\mathcal{O}_{\tilde{D}}\oplus\mathcal{O}_{\tilde{B}\oplus\mathcal{O}_{\Delta_0}}\rightarrow 0$$
given by tensoring with $s$.
The connecting homomorphism 
$$\delta_2: H^0(\tilde{Z}, \mathcal{O}_{\tilde{D}}\oplus\mathcal{O}_{\tilde{B}}\oplus\mathcal{O}_{\Delta_0})\rightarrow H^1(\tilde{Z}, \Omega_{\tilde{Z}}^1)$$ 
of the corresponding short exact sequence
is the first Chern class map. That is,  if $1_{\tilde{D}}$, $1_{\tilde{B}}$, and $1_{\Delta_0}$ are generators of $H^0(\tilde{Z}, \mathcal{O}_{\tilde{D}}\oplus\mathcal{O}_{\tilde{B}}\oplus\mathcal{O}_{\Delta_0})$, then $\delta_2(1_{\tilde{D}})=c_1(\tilde{D})$, $\delta_2(1_{\tilde{B}})=c_1(\tilde{B})$, and $\delta_2(1_{\Delta_0})=c_1(\Delta_0)$. Thus, the map $\delta_2$ is injective if and only if the curves $\tilde{D}$, $\tilde{B}$, and $\Delta_0$ are linearly independent in the Picard group of $\tilde{Z}$. Recalling that Pic$(\tilde{Z})$ is generated by $\Delta_0$, $\Gamma$, $E_1$ and $E_2$, that $\tilde{B}\sim 6\Delta_0+12\Gamma-2E_1-2E_2$ and that $\tilde{D}\sim \Gamma-E_1-E_2$, we see that the three divisors are indeed linearly independent.

Thus, we have a diagram
\[\xymatrix{
H^0(\tilde{Z}, (\mathcal{O}_{\tilde{D}}\oplus\mathcal{O}_{\tilde{B}}\oplus\mathcal{O}_{\Delta_0})(K_{\tilde{Z}}))\ar[r]^{\delta} \ar[d]_{\otimes s}  & H^1(\tilde{Z}, \Omega_{\tilde{Z}}^1(K_{\tilde{Z}})) \ar[d]^{\otimes s}\\
H^0(\tilde{Z}, \mathcal{O}_{\tilde{D}}\oplus\mathcal{O}_{
\tilde{B}}\oplus\mathcal{O}_{\Delta_0})\ar[r]^{\delta_2}& H^1(\tilde{Z}, \Omega_{\tilde{Z}}^1)
}\]
where the bottom arrow is injective. We see that $\delta$ is injective as long as the map on the left is injective. But this map simply takes a section of $(\mathcal{O}_{\tilde{D}}\oplus\mathcal{O}_{\tilde{B}})(K_{\tilde{Z}})$ and multiplies it by $s$. Since $s\neq0$, the map is injective.
\end{proof}

\begin{Lemma}\label{where the obstruction happens} $H^0(\tilde{\mathbb{F}}_2,\Omega_{\tilde{\mathbb{F}}_2}^1(\log \tilde{D})(K_{\tilde{\mathbb{F}}_2}+\tilde{L}))=\mathbb{C}.$
\end{Lemma}

\begin{proof}
By the projection formula we have
\begin{eqnarray*}
\sigma_*(\Omega_{\tilde{\mathbb{F}}_2}^1(\log \tilde{D})(K_{\tilde{\mathbb{F}}_2}+\tilde{L}))&=&\sigma_*(\Omega_{\tilde{\mathbb{F}}_2}^1(\log \tilde{D}))(K_{\mathbb{F}_2}+L)\\
&=& \sigma_*(\Omega_{\tilde{\mathbb{F}}_2}^1(\log \tilde{D}))\otimes\mathcal{O}(\Delta_0+2\Gamma).
\end{eqnarray*}
Lemma~\ref{blowup of two points} gives
$$\sigma_*(\Omega_{\tilde{\mathbb{F}}_2}^1(\log \tilde{D}))=\Omega^1_{\mathbb{F}_2}(\log D)\otimes\mathfrak{M}_{p,q}.$$
Thus,
$$ \sigma_*(\Omega_{\tilde{\mathbb{F}}_2}^1(\log \tilde{D})(K_{\tilde{\mathbb{F}}_2}+\tilde{L}))=\Omega^1_{\mathbb{F}_2}(\log D)\otimes\mathfrak{M}_{p,q}\otimes\mathcal{O}(\Delta_0+2\Gamma).$$

Let $\pi:\mathbb{F}_2\rightarrow \mathbb{P}^1$ be the projection map, and suppose that $\pi(D)=a$. We have the short exact sequence
$$0\rightarrow\pi^*\Omega_{\mathbb{P}^1}^1(\log a)\rightarrow\Omega_{\mathbb{F}_2}^1(\log D)\rightarrow\mathcal{O}_{\mathbb{F}_2}(-2\Delta_0-2\Gamma)\rightarrow 0.$$
The sheaf $\mathcal{O}_{\mathbb{F}_2}(-2\Delta_0-2\Gamma)$ is free, so $\textrm{Tor}_1(\mathfrak{M}_{p,q}\otimes \mathcal{O}(\Delta_0+2\Gamma), \mathcal{O}_{\mathbb{F}_2}(-2\Delta_0-2\Gamma))=0$. 
Thus,  tensoring $\mathfrak{M}_{p,q}\otimes \mathcal{O}(\Delta_0+2\Gamma)$ with the above short exact sequence yields the new short exact sequence
\begin{eqnarray*}
0\rightarrow\pi^*\Omega_{\mathbb{P}^1}^1(\log a)\otimes \mathfrak{M}_{p,q}\otimes \mathcal{O}(\Delta_0+2\Gamma)&\rightarrow&\Omega_{\mathbb{F}_2}^1(\log D)\otimes \mathfrak{M}_{p,q}\otimes \mathcal{O}(\Delta_0+2\Gamma)\\
&\rightarrow& \mathcal{O}_{\mathbb{F}_2}(-\Delta_0)\otimes\mathfrak{M}_{p,q}\rightarrow 0.
\end{eqnarray*}
Since $\mathbb{F}_2$ is projective, the sheaf $\mathcal{O}_{\mathbb{F}_2}(-\Delta_0)\otimes\mathfrak{M}_{p,q}$ has no global holomorphic sections, and so
$$H^0(\mathbb{F}_2, \pi^*\Omega_{\mathbb{P}^1}^1(\log a)\otimes \mathfrak{M}_{p,q}\otimes \mathcal{O}(\Delta_0+2\Gamma))\cong H^0(\mathbb{F}_2,\Omega_{\mathbb{F}_2}^1(\log D)\otimes \mathfrak{M}_{p,q}\otimes \mathcal{O}(\Delta_0+2\Gamma)).$$
The sheaf $\Omega^1_{\mathbb{P}^1}$ is isomorphic to $\mathcal{O}_{\mathbb{P}^1}(-2)$, so by the projection formula we have
$$H^0(\mathbb{F}_2, \pi^*\Omega_{\mathbb{P}^1}^1(\log a)\otimes \mathfrak{M}_{p,q}\otimes \mathcal{O}(\Delta_0+2\Gamma))=H^0(\mathbb{F}_2, \mathcal{O}_{\mathbb{F}_2}(\Delta_0+\Gamma)\otimes\mathfrak{M}_{p,q}).$$
The divisor $\Delta_0$ satisfies $\Delta_0\cdot (\Delta_0+\Gamma)=-1,$
so that $\Delta_0$ is a fixed part of the linear system $|\Delta_0+\Gamma|$. Since $D$ is the only fiber containing both $p$ and $q$, this implies that 
$$H^0(\mathbb{F}_2, \mathcal{O}_{\mathbb{F}_2}(\Delta_0+\Gamma)\otimes\mathfrak{M}_{p,q})=\mathbb{C},$$
completing the proof.
\end{proof}

We can now prove the theorem.

\begin{proof}[Proof of Theorem~\ref{obstruction}] By Serre duality, it suffices to show that  $$H^0(X, \Omega_X^1(\log C)(K_X))=\mathbb{C}.$$

By Lemma~\ref{decomposition into eigenspaces} and the projection formula, we have
$$f_*(\Omega_X^1(\log C)(K_X))=\Omega_{\tilde{\mathbb{F}}_2}^1(\log \tilde{D})(K_{\tilde{\mathbb{F}}_2}+\tilde{L})\oplus\Omega^1_{\tilde{\mathbb{F}}_2}(\log(\tilde{D}+\tilde{B}))(K_{\tilde{\mathbb{F}}_2}).$$

By Lemma~\ref{where the obstruction happens}, we have $H^0(\Omega_{\tilde{\mathbb{F}}_2}^1(\log \tilde{D})(K_{\tilde{\mathbb{F}}_2}+\tilde{L}))=\mathbb{C}.$ Moreover, 
$$\Omega^1_{\tilde{\mathbb{F}}_2}(\log(\tilde{D}+\tilde{B}))(K_{\tilde{\mathbb{F}}_2})\subset \Omega^1_{\tilde{\mathbb{F}}_2}(\log(\tilde{D}+\tilde{B}+\Delta_0))(K_{\tilde{\mathbb{F}}_2}).$$ 
Thus, $H^0(\Omega^1_{\tilde{\mathbb{F}}_2}(\log(\tilde{D}+\tilde{B}))(K_{\tilde{\mathbb{F}}_2}))=0$ by Lemma~\ref{coh vanishing on the blowup of F2}.
\end{proof}

\subsection{Deformations of pairs and the equisingular locus}\label{equisingular}

Let $f:X\rightarrow Y$ be the double cover of a smooth surface $Y$ branched over a smooth curve $B$. Define $\textrm{Def}_{X\rightarrow Y}$ to be the space of deformations of $X$ that are double covers of deformations of $Y$. The group $\mathbb{Z}/2\mathbb{Z}$ acts on $X$ by deck transformations, and the sheaf $f_*T_X$ decomposes into invariant and anti-invariant subspaces as
$$f_*T_X\simeq T_Y(\log B) \oplus T_Y(-L),$$
where $2L\sim B$~\cite{pardini1991}.

\begin{Theorem}\label{cvs}~\cite{cynk-vanstraten2006} Via the decomposition of $f_*T_X$ into its invariant and anti-invariant subspaces, the deformation space $\textrm{Def}(X\rightarrow Y)$ of double covers of deformations of $Y$ may be identified with the deformation space $\textrm{Def}(Y, B)$ of deformations of pairs, where $B$ is the branch divisor of $f$. 
\end{Theorem}

The proof of Theorem~\ref{cvs} involves identifying the space of first order infinitesimal deformations of double covers of deformations of $Y$ with the anti-invariant subspace $H^1_+(X, T_X)\subset H^1(X, T_X)$. Then using the decomposition of $f_*(T_X)$ above, this space is isomorphic to $H^1(Y, T_Y(\log B))$.

Using Lemma~\ref{decomposition into eigenspaces}, the same analysis works in the presence of the curves $C\subset X$ and $D\subset Y$, as long as $D$ intersects $B$ transversally. More explicitly, define $\mathpzc{Def}_{(X, C)\rightarrow(Y,D)}$ to be the functor of Artinian local rings which associates to an Artinian local ring $A$ the set of isomorphism classes of deformations over $A$ of squares
$$\xymatrix{
X \ar[r]  & Y \\
C \ar@{^{(}->}[u]\ar[r]& D \ar@{^{(}->}[u]
}$$
where the top and bottom maps are double covers and the left and right maps are embeddings of the smooth curves $C$ and $D$ into $X$ and $Y$, respectively.
Then the functor $\mathpzc{Def}_{(X, C)\rightarrow(Y,D)}$ may be identified with the functor $\mathpzc{Def}_{(Y, B, D)}$ of deformations of triples. The space of first-order infinitesimal deformations of triples $(Y, B, D)$ is therefore $H^1_+(X, T_X(\log C))$. By Lemma~\ref{decomposition into eigenspaces}, we have 
$$H^1_+(X, T_X(\log C))\simeq H^1(Y, T_Y(\log (B+D))).$$

There is a forgetful map $\alpha:\mathpzc{Def}_{(X, C)\rightarrow(Y,D)}\rightarrow\mathpzc{Def}_{(X,C)}$. This map is an analytic embedding, because the differential 
$$d\alpha: H^1(Y, T_Y(\log (B+D)))\rightarrow H^1_+(X, T_X(\log C))\subset H^1(X, T_X(\log C))$$
is an isomorphism onto its image. 

Suppose now that $W$ is a stable numerical quintic surface of type 2b, $X$ its minimal resolution, and $C$ the $(-4)$ curve on $X$. Then we have the commutative square
$$\xymatrix{
X \ar[r]^{\tilde{f}}  & \tilde{Z} \\
C \ar@{^{(}->}[u]\ar[r]& \tilde{D} \ar@{^{(}->}[u]
}$$
where $\tilde{f}$ is the double cover of $\tilde{Z}$, where $\tilde{Z}$ is the blowup of $Z=\mathbb{F}_2$ in two points lying on a fiber $D$. The branch curve of $\tilde{f}$ is a smooth curve $\tilde{B}$, which intersects the proper transform $\tilde{D}$ of $D$ transversally. 
Deformations of this square can be identified with deformations of the triple $(\tilde{Z}, B, \tilde{D})$. The following lemma shows that in this case, the image of $\alpha$ is a neighborhood of $[W]$ in the $\overline{\mbox{2a}}$ component of $\overline{\mathcal{M}}_{5,5}$. We note that Lemma~\ref{coh vanishing on the blowup of F2} implies that there are no obstructions, so the image of $\alpha$ is smooth. 

\begin{Theorem}\label{2a to 2b come from double covers} Let $\mathcal{W}\rightarrow T$ be a stable family whose fibers are all 2a or 2b surfaces and $\mathcal{X}\rightarrow \mathcal{W}$ be its simultaneous minimal resolution over $T$, which exists by~\cite[Theorem 7.68]{kollarmori}. Then there exists a double cover $j:\mathcal{X}\rightarrow \tilde{\mathcal{Z}}$ of schemes smooth over $T$, where $\tilde{\mathcal{Z}}$ is a smooth family of Hirzebruch surfaces of type $\mathbb{F}_2$ or $\mathbb{F}_0$ blown up at two points on a fiber. 
\end{Theorem}

\begin{proof}
Let $\psi:\mathcal{X}\rightarrow \mathcal{Y}$ be the canonical model of $\mathcal{X}$ over $T$. Then the canonical map given by the linear system $|\omega_{\mathcal{Y}/T}|$ is a double cover $f: \mathcal{Y}\rightarrow \mathcal{Z}$ over $T$, where fibers of $\mathcal{Z}\rightarrow T$ are either smooth or singular quadrics. Let $\mathcal{B}$ denote the branch divisor of $f$ and suppose that $\mathcal{Z}_{t_0}$ is singular for some $t_0\in T$. Then because the fibers of $\mathcal{X}\rightarrow T$ are 2a or 2b surfaces, the branch divisor $\mathcal{B}_{t_0}$ of the map $f|_{t_0}$ is disjoint from the node in $\mathcal{Z}_{t_0}$. 

Let $\sigma_1: \mathcal{Z}_1\rightarrow \mathcal{Z}$ be a simultaneous resolution of singularities of $\mathcal{Z}$. Then the simultaneous resolutions $\sigma_1$ and $\psi$ are locally analytically isomorphic in a neighborhood of each singularity of $\mathcal{Z}$, because the branch divisor $\mathcal{B}$ does not intersect the singularities of $\mathcal{Z}$. Thus, no finite base change of $T$ is required to construct $\mathcal{Z}_1$. Letting $\mathcal{Y}_1$ denote the double cover $f_1$ of $\mathcal{Z}_1$ branched over the preimage $\mathcal{B}_1$ of $\mathcal{B}$, there is a map $\psi_1:\mathcal{Y}_1\rightarrow\mathcal{Y}$ such that the following diagram is commutative:

\[\xymatrix{
\mathcal{Y}_1\ar[r]^{f_1} \ar[d]_{\psi_1} & \mathcal{Z}_1 \ar[d]^{\sigma_1}\\
\mathcal{Y}\ar[r]^f& \mathcal{Z}
}\]

Now let $\mathcal{B}_1$ denote the preimage of $\mathcal{B}$ under $\sigma_1$. On each fiber, $\mathcal{B}_1$ has two $A_1$ singularities. Let $X$ denote the double section of $\mathcal{B}_1\rightarrow T$ passing through these singularities, and let $\sigma_2:\tilde{\mathcal{Z}}\rightarrow \mathcal{Z}_1$ be the blowup of $\mathcal{Z}_1$ in $X$. Then $\sigma_2$ is a simultaneous embedded resolution of singularities of $\mathcal{B}_1$. Thus, no finite base change of $T$ is required to construct $\mathcal{Z}$, and the map $j:\mathcal{X}\rightarrow\tilde{\mathcal{Z}}$ defined by $\sigma_1\circ\sigma_2\circ j=f\circ\psi$ is a double cover over $T$.
\end{proof}

Thus, the space of equisingular deformations of $W$ contains a smooth $39$-dimensional component corresponding to the closure of the 2a locus in $\overline{\mathcal{M}}_{5,5}$.

\subsection{Three technical lemmas}\label{technicalities}

Our goal is to describe the degree two part of the Schouten bracket. We use a method similar to~\cite{horikawa1975, horikawa1976}. In this section, we prove three technical lemmas, analogous to Lemmas 24, 29, and 31 in~\cite{horikawa1975}.

We recall the definition of the Kuranishi deformation space in more generality. Suppose that $X$ is a smooth surface, and let $\textrm{Def}(X)$ be the space of deformations of $X$. The tangent space to $\textrm{Def}(X)$, that is the space of first order infinitesimal deformations of $X$, is isomorphic via the Kodaira--Spencer map to the cohomology group $H^1(X, T_X)$. Let $\rho_1,\ldots, \rho_{n}$ be a basis of $H^1(X, T_X)$, and let $t_1,\ldots t_n$ be a dual basis. Then $\textrm{Def}(X)$ is locally analytically isomorphic to a subspace of $\mathbb{C}^{n}$ with coordinates $t_1,\ldots, t_{n}$, and is given by the kernel of the Kuranishi map $k: H^1(X, T_X)\rightarrow H^2(X, T_X)$, which is a certain infinite series in $t_1,\ldots t_n$. Catanese's article \cite{cataneseguide} gives an excellent exposition of the construction of the Kuranishi map. For us, the important part is that the degree two part of the Kuranishi map is given by the Schouten bracket, which we now describe.

The Schouten bracket is the bilinear map
$$[,]: H^1(X, T_X)\otimes H^1(X, T_X)\rightarrow H^2(X, T_X)$$
defined as the composition of the cup product
$\cup: H^1(X, T_X)\otimes H^1(X, T_X) \rightarrow H^2(X, T_X\otimes T_X)$ followed by the Lie bracket $H^2(X, T_X\otimes T_X)\rightarrow H^2(X, T_X)$.
If $S_{\rho}$ is the infinitesimal first order deformation corresponding, via the Kodaira--Spencer map, to $\rho\in H^1(X, T_X)$, then $[\rho, \rho]$ is the cohomology class corresponding to the obstruction to extending the deformation $S_{\rho}$ to the second order.

\begin{Lemma}\label{zeta is surjective} The map 
$$\zeta_*:H^1(X, T_X(\log C))\rightarrow H^1(F_1\amalg F_2, \mathcal{N}_{F_1\amalg F_2})$$
induced by the surjection $T_X|_{F_1\amalg F_2}\rightarrow \mathcal{N}_{F_1\amalg F_2}$ is surjective.
\end{Lemma}
\begin{proof}
It suffices to show that 
$$H^1(\tilde{\mathbb{F}}_2, f_*(T_X(\log C)))\rightarrow H^1(\Delta_0, f_*(\mathcal{N}_{F_1\amalg F_2}))$$
is surjective. To do this, recall that the surface $\tilde{\mathbb{F}}_2$ admits an action of $\mathbb{Z}/2\mathbb{Z}$ via deck transformations. By Lemma~\ref{decomposition into eigenspaces}, the sheaf $f_*(T_X(\log C))$  decomposes into invariant and anti-invariant eigenspaces as
\begin{eqnarray*}
f_*(T_X(\log C))_+= T_{\tilde{\mathbb{F}}_2}(\log (\tilde{D}+\tilde{B})) & \textrm{  and  }& f_*(T_X(\log C))_-= T_{\tilde{\mathbb{F}}_2}(\log \tilde{D})\otimes \mathcal{O}(-\tilde{L}).
\end{eqnarray*}

We have a similar decomposition of $f_*(\mathcal{N}_{F_1\amalg F_2})$ as follows.  By the projection formula, we have 
$$f_*(\mathcal{N}_{F_1\amalg F_2})=f_*(f^*(\mathcal{N}_{\Delta_0}))= \mathcal{N}_{\Delta_0}\otimes(\mathcal{O}_{\tilde{\mathbb{F}}_2} \oplus \mathcal{O}_{\tilde{\mathbb{F}}_2}(-\tilde{L})).$$
Thus,
\begin{eqnarray*}
 f_*(\mathcal{N}_{F_1\amalg F_2}))_+=\mathcal{N}_{\Delta_0} & \textrm{  and  } &  f_*(\mathcal{N}_{F_1\amalg F_2}))_-=\mathcal{N}_{\Delta_0}\otimes\mathcal{O}(-\tilde{L}) \simeq \mathcal{N}_{\Delta_0}.
\end{eqnarray*}

We show that the maps
$$\zeta_+:H^1(\tilde{\mathbb{F}}_2,  T_{\tilde{\mathbb{F}}_2}(\log (\tilde{D}+\tilde{B})))\rightarrow H^1(\Delta_0, \mathcal{N}_{\Delta_0}) $$
and
$$\zeta_-:H^1(\tilde{\mathbb{F}}_2, T_{\tilde{\mathbb{F}}_2}(\log \tilde{D})\otimes \mathcal{O}(-\tilde{L}))\rightarrow H^1(\Delta_0, \mathcal{N}_{\Delta_0})$$
are surjective.

To show the first, we have the exact sequence
$$0\rightarrow T_{\tilde{\mathbb{F}}_2}(\log(\Delta_0+\tilde{D}+\tilde{B}))\rightarrow T_{\tilde{\mathbb{F}}_2}(\log(\tilde{D}+\tilde{B}))\rightarrow \mathcal{N}_{\Delta_0}\rightarrow 0$$ 
and so it suffices to show that $H^2(\tilde{\mathbb{F}}_2, T_{\tilde{\mathbb{F}}_2}(\log(\Delta_0+\tilde{D}+\tilde{B})))=0$. By Serre duality, this is equivalent to the vanishing of $H^0(\tilde{\mathbb{F}}_2, \Omega^1_{\tilde{\mathbb{F}}_2}(\log(\Delta_0+\tilde{D}+\tilde{B}))\otimes\mathcal{O}(K))$. This is the statement of Lemma~\ref{coh vanishing on the blowup of F2}.

For the second, note that we have the exact sequence
$$0\rightarrow T_{\tilde{\mathbb{F}}_2}(\log\tilde{D}+\Delta_0)\otimes \mathcal{O}(-\tilde{L})\rightarrow T_{\tilde{\mathbb{F}}_2}(\log\tilde{D})\otimes \mathcal{O}(-\tilde{L})\rightarrow \mathcal{N}_{\Delta_0}\rightarrow 0.$$
By Lemma~\ref{where the obstruction happens}, we have $H^2(T_{\tilde{\mathbb{F}}_2}(\log\tilde{D})\otimes \mathcal{O}(-\tilde{L}))=\mathbb{C}$. Moreover, $H^2(\mathcal{N}_{\Delta_0})=0$, and thus the map $\zeta_-$ is surjective as long as $H^2(\tilde{\mathbb{F}}_2, T_{\tilde{\mathbb{F}}_2}(\log \tilde{D}+\Delta_0)\otimes\mathcal{O}(-\tilde{L}))=\mathbb{C}$. Equivalently, we show that  $H^0(\tilde{\mathbb{F}}_2, \Omega^1_{\tilde{\mathbb{F}}_2}(\log \tilde{D}+\Delta_0)\otimes\mathcal{O}(K+\tilde{L}))=\mathbb{C}$.

By Lemma~\ref{blowup of two points} and the projection formula, we have 
$$\sigma_*\Omega^1_{\tilde{\mathbb{F}}_2}(\log (\tilde{D}+\Delta_0))\otimes\mathcal{O}(K+\tilde{L})=\Omega^1_{\mathbb{F}_2}(\log(D+\Delta_0))\otimes\mathcal{O}(\Delta)\otimes\mathfrak{M}_{p,q}.$$
So we now want 
$$H^0(\mathbb{F}_2, \Omega^1_{\mathbb{F}_2}(\log(D+\Delta_0))\otimes\mathcal{O}(\Delta)\otimes\mathfrak{M}_{p,q})=\mathbb{C}.$$

We claim that the sheaf $T_{\mathbb{F}_2}(\log(D+\Delta_0))$ fits into an exact sequence as 
$$0\rightarrow\mathcal{O}(G)\rightarrow T_{\mathbb{F}_2}(\log(D+\Delta_0))\rightarrow \pi^*T_{\mathbb{P}^1}(-a)\rightarrow 0$$
where $G$ is a divisor on $\mathbb{F}_2$ and  $\pi(D)=a\in\mathbb{P}^1$. To see this, note first that $\pi^*T_{\mathbb{P}^1}(-a)\simeq \mathcal{O}_{\mathbb{F}_2}(D)$. Let $U\subset \mathbb{F}_2$ be an open neighborhood of the point $0\in D\cap \Delta_0$ with coordinates $(x,y)$ so that $D$ has local equation $x$ and $\Delta_0$ has local equation $y$. Then the map 
$$T_{\mathbb{F}_2}(\log(D+\Delta_0))\rightarrow \mathcal{O}_{\mathbb{F}_2}(D)$$
is locally given by 
$$x\frac{\partial}{\partial x} \mapsto x\textrm{\hspace{1cm}    } y\frac{\partial}{\partial y}\mapsto 0.$$
Thus the map is surjective. Since $T_{\mathbb{F}_2}(\log(D+\Delta_0))$ is a vector bundle of rank two and $\mathcal{O}_{\mathbb{F}_2}(D)$ is a line bundle, the kernel of the map $T_{\mathbb{F}_2}(\log(D+\Delta_0))\rightarrow \mathcal{O}_{\mathbb{F}_2}(D)$ is a vector bundle of rank one. All such vector bundles are given by $\mathcal{O}_{\mathbb{F}_2}(G)$ for some divisor $G$ on $\mathbb{F}_2$. We find $G$ by calculating the Chern class of $T_{\mathbb{F}_2}(\log(D+\Delta_0))$. 
The determinant line bundle $\bigwedge^2T_{\mathbb{F}_2}(\log(D+\Delta_0))$ is given by $-\mathcal{O}(-K_{\mathbb{F}_2}-D-\Delta_0)=\mathcal{O}(\Delta_0+3\Gamma)$, so $c_1(T_{\mathbb{F}_2}(\log( D+\Delta_0)))=\Delta_0+3\Gamma$. 
Thus, $G = \Delta_0+2\Gamma$. 

Dualizing the above exact sequence and tensoring with $\mathcal{O}(\Delta)\otimes \mathfrak{M}_{p,q}$ results in the exact sequence
$$0\rightarrow \mathcal{O}(\Delta_0+\Gamma)\otimes \mathfrak{M}_{p,q}\rightarrow \Omega^1_{\mathbb{F}_2}(\log D+\Delta_0)\otimes \mathcal{O}(\Delta)\otimes \mathfrak{M}_{p,q}\rightarrow \mathcal{O}_{\mathbb{F}_2}\otimes \mathfrak{M}_{p,q}\rightarrow 0.$$
The sheaf on the right has no global sections, since the only section of $\mathcal{O}_{\mathbb{F}_2}$ vanishing at $p$ and $q$ is zero. Moreover, since $\Delta_0\cdot (\Delta_0+\Gamma)=-1$ every divisor in the linear system $|\Delta_0+\Gamma|$ is a union of two divisors $\Delta_0$ and $\Gamma$. Since there is only one such divisor passing through $p$ and $q$, namely the divisor $\Delta_0+D$, we have 
$$ H^0(\mathbb{F}_2, \Omega^1_{\mathbb{F}_2}(\log D+\Delta_0)\otimes \mathcal{O}(\Delta)\otimes \mathfrak{M}_{p,q})\simeq H^0(\mathbb{F}_2, \mathcal{O}(\Delta_0+D)\otimes \mathfrak{M}_{p,q} )=\mathbb{C},$$
as we wished to show.
\end{proof}

A key ingredient of Horikawa's description in~\cite{horikawa1976} is a map 
$$\gamma: H^1(X, T_X)\rightarrow H^0(G, \mathcal{O}(K_X|_{G}) ),$$
where $K_X=2G+F$ and $G$ is a generic fiber of the map $g: X\rightarrow \mathbb{P}^1$. 

\begin{Lemma}\label{ideal} Let $X$ be a smooth surface with a surjective map $g: X\rightarrow \mathbb{P}^1$ such that $g_*\mathcal{O}_X=\mathcal{O}_{\mathbb{P}^1}$ and let $G$ denote a generic fiber of $g$. Suppose that $K_X=2G+F$ for some smooth divisor $F$ on $X$ such that $G\not\subset F$ and let 
$$\zeta_*:H^1(X, T_X)\rightarrow H^1(F, \mathcal{N}_{F})$$
be the map induced by the surjection $T_X|_{F}\rightarrow \mathcal{N}_{F}$. 
If the irregularity $q(X)=0$,  $h^1(X, \mathcal{O}(G))=0$, 
and $h^0(F, \mathcal{O}(K-G)|_{F})=0$, then there is a map
$\gamma: H^1(X, T_X)\rightarrow H^0(F, \mathcal{O}(K_X|_{F}) )$, defined in~\cite[4.8]{mythesis}
with the property that $\textrm{Ker }\gamma=\textrm{Ker }\zeta_*$.
\end{Lemma}
\begin{proof}

This is a generalization of the map defined in~\cite{horikawa1976}, Section 5, after the proof of Lemma 24. See~\cite{mythesis} for details.
\end{proof}

\begin{Lemma}\label{zero bracket gives restriction} With the same hypotheses as Lemma~\ref{ideal}, if $[\rho, \rho]=0$ then $(\gamma(\rho))^2$ is in the image of the restriction map $H^0(X, \mathcal{O}(2K))\rightarrow H^0(F, \mathcal{O}(2K|_{F}))$.
\end{Lemma}
\begin{proof}
The proof is a generalization of that of Lemma 31 in~\cite{horikawa1975}. See~\cite{mythesis} for details.
\end{proof}

\subsection{Proof of the main theorem}\label{main proof}

We describe the space $\textrm{Def}^{QG}(W)$ of $\mathbb{Q}$-Gorenstein deformations of a general 2b surface $W$. 

\begin{Lemma}\label{coh works well}\cite[Lemma 6.3]{horikawa1976} Let $X$ be the minimal resolution of a surface of type 2b. Then  $h^1(X, \mathcal{O}(G))=2$ and $h^1(X, \mathcal{O}(G+F_1+F_2))=0$.
\end{Lemma}
\begin{proof}Horikawa proves this in the case that $X$ is a double cover of $\mathbb{F}_2$ with a smooth branch divisor. The proof uses Riemann-Roch and Serre duality together with the fact that the canonical divisor on $X$ is given by $K_X=2G+F_1+F_2$. Because it only relies on numerical characteristics of $X$, $G$, $F_1$ and $F_2$, Horikawa's proof works in our case as well.
\end{proof}

By Lemma~\ref{coh works well}, we can define the map $\gamma$ as in Lemma~\ref{ideal}, where $F=F_1+F_2$. By abuse of notation, we let 
$$\gamma: H^1(X, T_X(\log C))\rightarrow H^0(F_1\amalg F_2, \mathcal{O}_{F_1\amalg F_2})$$
be the restriction of this map to $H^1(X, T_X(\log C))$. We note that this map is the restriction  to $H^1(X, T_X(\log C))\subset H^1(X, T_X)$ of the corresponding map defined in~\cite{horikawa1976} under the assumption that the branch locus is smooth.

\begin{proof}[Proof of Theorem~\ref{the best}]
The deformation space $\textrm{Def}^{QG, \mbox{e.s}}(W)$ is locally analytically isomorphic to the zero-set of the Kuranishi map 
$$k: H^1(X, T_X(\log C))\rightarrow H^2(X, T_X(\log C))=\mathbb{C}.$$ Choose a basis $\rho_1,\rho_2,\ldots, \rho_{40}$ of $H^1(X, T_X(\log C))$.
Let $t_1, t_2,\ldots, t_{40}$ be the dual basis. A priori, the Kuranishi map is some power series in $t_1,\ldots, t_{40}$. However in this case, Theorems~\ref{2a to 2b come from double covers} and~\ref{counting} imply that $\textrm{Def}^{QG, \mbox{e.s}}(W)$ contains a smooth $39$-dimensional subspace corresponding to deformations of a 2b surface to a 2a surface.
This implies that if we choose a basis $\rho_1,\rho_2,\ldots, \rho_{40}$ of $H^1(X, T_X(\log C))$ such that $\rho_1\in H^1_-(X,T_X(\log C))$ and $\rho_i\in H^1_+(X, T_X(\log C))$ for $i>2$, then the corresponding dual basis has the property that the Kuranishi function factors into at least two terms, one of which has linear term $t_1$. To show that $\textrm{Def}^{QG, \mbox{e.s}}(W)$ is locally a product of two smooth $39$-dimensional components meeting transversally in a 38-dimensional component, it therefore suffices to show that the degree two part of the Kuranishi map is nonzero and not a square. The degree two part is given by the Schouten bracket, defined above.

We restrict the Schouten bracket $[,]$ to $H^1(X, T_X(\log C))\otimes H^1(X, T_X(\log C))$. We claim that the Lie bracket $H^2(X, T_X(\log C)\otimes T_X(\log C))\rightarrow H^2(X, T_X)$ has image in $H^2(X, T_X(\log C))$. Let $\{U_i\}$ be a sufficiently fine open covering of $X$ and let $U_{ijk}=U_i\cap U_j\cap U_k$. Let $\rho$ be an element of $H^2(X, T_X(\log C)\otimes T_X(\log C))$, represented by a 2-cocycle $\{\rho_{ij}\otimes\rho_{jk}\}$, where $\{\rho_{ij}\}$ is a 1-cocycle with coefficients in $T_X(\log C)$. Then $\rho_{ij}$ is a vector field that fixes the ideal sheaf of the $C$. Thus, $\rho_{ij}\otimes\rho_{jk}$ also fixes the ideal sheaf of $C$. Therefore the Lie bracket
$[\rho_{ij}, \rho_{jk}]$ gives a vector field on $U_{ijk}$ which also fixes the ideal sheaf of $C$. Thus, the form $[,]$ gives a $2$-cocycle with coefficients in $T_X(\log C)$; that is
$$[,]: H^1(X, T_X(\log C))\otimes H^1(X, T_X(\log C))\rightarrow H^2(X, T_X(\log C))\simeq \mathbb{C}.$$

Because $\rho_1,\ldots, \rho_{40}$ and $t_1,\ldots t_{40}$ are dual bases, the degree two part of the Kuranishi map $k$ is given by 
$$\sum_{1\le i,j\le 40}[\rho_i,\rho_j]t_it_j.$$ 
Moreover, because $k$ factors into a product, one term of which has linear term $t_1$, we have that $[\rho_i, \rho_j]=0$ for $2\le i, j\le 40$. It therefore suffices to show that $[\rho_1, \rho_1]=0$ and $[\rho_1, \rho_i]$ is nonzero for some $i>1$. 

Recall that $K_X=2G+F_1+F_2$ and consider the exact sequence 
$$0\rightarrow  \mathcal{O}_X(2G)\rightarrow \mathcal{O}_X(K)\rightarrow \mathcal{O}_{F_1}\oplus\mathcal{O}_{F_2}\rightarrow 0.$$
On $X$, we have $p_g=4$ and $h^0(2G)=3$, so the image of the map $r: H^0(X, \mathcal{O}_X(K))\rightarrow H^0(X, \mathcal{O}_{F_1}\oplus\mathcal{O}_{F_2})$ is one-dimensional. 
Moreover, the image of $r$ is contained in the ``diagonal" in $H^0(X, \mathcal{O}_{F_1}\oplus\mathcal{O}_{F_2})\simeq \mathbb{C}^2$. That is, if $X$ is a nonzero global section of $\mathcal{O}_{F_1}\oplus\mathcal{O}_{F_2}$ in the image of $r$, then $s|_{F_1}\neq 0$ and $s|_{F_2}\neq 0$. More precisely, we have the commutative diagram below, where the arrow on the left is an isomorphism and the inclusion of $H^0(\tilde{\mathbb{F}_2}, \Delta_0)$ into $H^0(X, \mathcal{O}_{F_1}\oplus\mathcal{O}_{F_2})$ sends a section to the section of $\mathcal{O}_{F_1}\oplus\mathcal{O}_{F_2}$ whose restrictions to $F_1$ and $F_2$ are equal.
$$
\xymatrix{ H^0(X, \mathcal{O}_X(K))\ar[r]^r& H^0(X, \mathcal{O}_{F_1}\oplus\mathcal{O}_{F_2})\\
H^0(\tilde{\mathbb{F}}_2, \Delta_0+2\Gamma)\ar@{^{(}->}[u]_{\simeq}\ar[r]&H^0(\tilde{\mathbb{F}}_2, \Delta_0)\ar@{^{(}->}[u]
}
$$

By Lemmas~\ref{ideal}, \ref{coh works well} and~\ref{zeta is surjective}, the map $\gamma: H^1(X, T_X(\log C))\rightarrow H^0(X,\mathcal{O}_{F_1}\oplus\mathcal{O}_{F_2})$ is surjective. Thus, we can choose $\rho\in H^1(X, T_X(\log C))$ such that $\gamma(\rho)\neq 0$, and $\gamma(\rho)^2$ is not in the image of $r$. But then $\gamma(\rho)^2$ is not a restriction of an element of $H^0(X, \mathcal{O}(2K))$, so by Lemma~\ref{zero bracket gives restriction}, we conclude that $[\rho,\rho]\neq0$. Thus, the Schouten bracket
$$[,]:H^1(X, T_X(\log C))\times H^1(X,T_X( \log C))\rightarrow H^2(X, T_X(\log C))\simeq \mathbb{C}$$ is surjective. 

Because it is locally given by the composition of the cup product followed by the Lie bracket of vector fields, the Schouten bracket is $\mathbb{Z}/2\mathbb{Z}$-equivariant under the action of $\mathbb{Z}/2\mathbb{Z}$ by deck transformations. By Lemma~\ref{obstruction}, the invariant part of $H^2(X, T_X(\log C))$ is zero, and so $[\rho_i,\rho_j]$ is nonzero if and only if $[\rho_i,\rho_j]$ is anti-invariant under the action of $\mathbb{Z}/2\mathbb{Z}$. Suppose that $\rho\otimes\eta$ is an element of $H^1(X,T_X(\log C))\otimes H^1(X, T_X(\log C))$, where $\rho$ and $\eta$ are either both invariant or both anti-invariant. Then $[\rho, \eta]$ is invariant, that is $[\rho,\eta]\in H_+^2(X, T_X(\log C))$. By Lemma~\ref{obstruction}, this space is zero, so $[\rho,\eta]=0$. Thus, by choice of basis, $[\rho_i, \rho_i]=0$ for all $i$; in particular, $[\rho_1, \rho_1]=0$. 

Now suppose that $\rho\in H^1_+(X, T_X(\log C))$ is invariant and $\eta\in H^1_-(X, T_X(\log C))$ is anti-invariant. Then $[\rho,\eta]\in H^2_-(X, T_X(\log C))$ is anti-invariant. Since $[, ]$ is surjective, there exists, by choice of basis, $i>1$ such that $[\rho_1, \rho_i]\neq0$, completing the proof.
\end{proof}


\bibliographystyle{alpha}
\bibliography{myrefsquintics}

\end{document}